\documentclass[reqno]{amsart}
\usepackage[margin=1in]{geometry}
\usepackage{amssymb}
\usepackage{mathtools}
\usepackage{mathrsfs}
\usepackage{stmaryrd}
\usepackage{enumitem}
\usepackage{graphicx}
\usepackage{float}
\usepackage[dvipsnames]{xcolor}
\usepackage{tikz-cd}
\usetikzlibrary{positioning,arrows.meta,calc,decorations.pathmorphing}
\usepackage[final,expansion=false]{microtype}
\usepackage[colorlinks]{hyperref}
\hypersetup{
  colorlinks=true,
  linkcolor=blue,
  citecolor=blue,
  filecolor=blue,
  urlcolor=blue
}
\usepackage[noabbrev,capitalize]{cleveref}
\crefname{subsection}{subsection}{subsections}
\Crefname{subsection}{Subsection}{Subsections}
\definecolor{mathblue}{RGB}{20,55,180}
\definecolor{softblue}{RGB}{246,249,255}
\colorlet{figref}{mathblue}
\definecolor{figarrow}{RGB}{204,85,0}

\newcommand{\figref}[2]{\hyperref[#2]{#1\,\ref*{#2}}}

\newcommand{\FFF}{\mathcal{F}}
\newcommand{\A}{\mathcal{A}}
\newcommand{\g}{\mathfrak{g}}

\renewcommand{\a}{\mathfrak{a}}
\newcommand{\K}{\mathcal{K}}
\renewcommand{\L}{\mathcal{L}}

\newcommand{\Z}{\mathbb{Z}}
\renewcommand{\b}{\mathfrak{b}}

\newcommand{\R}{\mathbb{R}}
\newcommand{\G}{\mathcal{G}}
\newcommand{\C}{\mathcal{C}}

\newcommand{\U}{\mathcal{U}}
\newcommand{\tto}{\rightrightarrows}

\newcommand{\sss}{\mathsf{s}}
\newcommand{\ccc}{\mathsf{c}}
\newcommand{\ttt}{\mathsf{t}}
\newcommand{\mtoo}{\longmapsto}
\newcommand{\lto}{\rightsquigarrow}
\newcommand{\HHH}{\mathcal{H}}
\DeclareMathOperator{\ad}{ad}

\newcommand{\D}{\mathcal{D}}
\newcommand{\T}{\mathcal{T}}
\newcommand{\h}{\mathfrak h}
\newcommand{\mmm}{\mathrm{m}}

\DeclareMathOperator{\MC}{MC}
\newcommand{\ppp}{\operatorname{pr}}
\newcommand{\im}{\operatorname{im}}
\newcommand{\defn}[1]{\emph{\textbf{#1}}}
\newcommand{\courant}[1]{\left\llbracket #1 \right\rrbracket}

\usepackage{aliascnt}
\theoremstyle{plain}
\newtheorem{theorem}{Theorem}[section]
\newtheorem{introtheorem}{Theorem}

\crefname{introtheorem}{theorem}{theorems}
\Crefname{introtheorem}{Theorem}{Theorems}
\newcommand{\newsharedtheorem}[2]{
  \newaliascnt{#1}{theorem}
  \newtheorem{#1}[#1]{#2}
  \aliascntresetthe{#1}}
\newsharedtheorem{lemma}{Lemma}
\newsharedtheorem{proposition}{Proposition}
\newsharedtheorem{corollary}{Corollary}
\theoremstyle{definition}
\newsharedtheorem{definition}{Definition}
\newsharedtheorem{question}{Question}
\newsharedtheorem{example}{Example}
\newsharedtheorem{remark}{Remark}
\numberwithin{equation}{section}
\DeclareMathOperator{\Alt}{Alt}
\DeclareMathOperator{\wt}{wt}

\DeclareMathOperator{\id}{id}
\DeclareMathOperator{\rank}{rank}

\DeclareMathOperator{\Sym}{Sym}

\begin{document}
\title{Integration and quantization of 1-shifted coisotropics}

\author[Yassine Ait Mohamed]{Yassine Ait Mohamed}
\address{D\'epartement de math\'ematiques, Universit\'e de Sherbrooke, 2500 Bd de l'Universit\'e, Sherbrooke, QC, J1K 2R1, Canada}
\email{yassine.ait.mohamed@usherbrooke.ca}

\author[Maxence Mayrand]{Maxence Mayrand}
\address{D\'epartement de math\'ematiques, Universit\'e de Sherbrooke, 2500 Bd de l'Universit\'e, Sherbrooke, QC, J1K 2R1, Canada}
\email{maxence.mayrand@usherbrooke.ca}

\date{\today}

\begin{abstract}
To quantize the reduction of a coisotropic submanifold of a Poisson manifold, Cattaneo and Felder constructed a homotopy Poisson algebra together with a deformation quantization. Its degree-zero cohomology is the reduced Poisson algebra, while the full homotopy structure retains information about the formal embedding even when the reduced space is singular or trivial. On the integration side, Cattaneo showed that a coisotropic submanifold integrates to a Lagrangian subgroupoid of any symplectic groupoid integrating the ambient Poisson manifold. It is now recognized that both constructions fit into shifted symplectic geometry: the Lagrangian subgroupoid presents a $1$-shifted Lagrangian morphism of differentiable stacks, and the homotopy Poisson algebra describes the associated $0$-shifted Poisson structure. However, such morphisms form only a special class of $1$-shifted Lagrangians. We extend this picture to $1$-shifted Lagrangians and, more generally, $1$-shifted coisotropics in $1$-shifted symplectic differentiable stacks. We establish an integration theorem for their infinitesimal description in terms of twisted Dirac structures and construct a natural Poisson bracket on invariant functions. Under regularity hypotheses, we equip the associated Lie algebroid complex with a flat $P_\infty$-structure and a curved $A_\infty$-quantization, independent of the auxiliary choices up to isomorphism. The induced bracket in degree zero recovers the reduced Poisson bracket, and suitable cohomological vanishing conditions yield a deformation quantization of invariant functions. This extends the Cattaneo--Felder picture to more general reduction procedures, with the same infinitesimal data governing integration and quantization.
\end{abstract}

\maketitle

\raggedbottom

\tableofcontents

\section{Introduction}\label{1SCC47LDM}

The interplay between \emph{integration}---passing from infinitesimal data to global structures---and \emph{quantization}---passing from classical to quantum systems---has driven some of the most influential developments in modern mathematics and physics. On the physics side, a leading example is Feynman's path integral formulation of quantum mechanics \cite{Feynman1948}. On the mathematical side, a central milestone is Weinstein's notion of symplectic groupoids \cite{Weinstein1987}, which integrate Poisson manifolds and were introduced as a geometric framework for their quantization.

Using insights from Feynman's path integrals, Cattaneo and Felder \cite{CattaneoFelder2000,CattaneoFelder2001} clarified the link between symplectic groupoids and quantization: for an integrable Poisson manifold, the \emph{Poisson sigma model}, a topological field theory, produces both its integrating symplectic groupoid, through classical reduction, and Kontsevich's star product \cite{Kontsevich2003}, through perturbative quantization. Integration and quantization thus emerge from the same model.

This framework was later extended to a relative setting: given a Poisson manifold $M$ together with a coisotropic submanifold $N\subset M$, one seeks both an integration of $N$ and a relative quantization of the pair $(M,N)$ \cite{CattaneoFelder2004Branes,Cattaneo2007,Cattaneo2004OnTheIntegration}. The integration question has a concrete answer. The conormal bundle $\nu_N^*$ is a Lie subalgebroid of $T^*M$ and, when $M$ is integrable, integrates to a possibly immersed Lagrangian subgroupoid $\L$ of a symplectic groupoid $\G$ integrating $M$. The quantization question leads to a more substantial change of viewpoint.

To see why, write $\pi$ for the Poisson bivector on $M$. The characteristic distribution $\pi^\sharp(\nu_N^*)$ of a coisotropic submanifold defines a possibly singular foliation $\mathcal F$ on $N$. If the leaf space $\underline N:=N/\mathcal F$ is a smooth manifold and the quotient map is a submersion, then $\underline N$ inherits a Poisson structure. Without any smoothness assumption on the quotient, the algebra
\[
C^\infty(N)^{\mathcal F}
:=\{f\in C^\infty(N):df|_{\pi^\sharp(\nu_N^*)}=0\}
\]
still carries a natural Poisson bracket, and one would like to quantize it. In the singular case, however, the existence of this bracket gives no general guarantee of a deformation quantization. As Cattaneo and Felder emphasize \cite{Cattaneo2007}, the attempt to obtain a star product from the coisotropic data can encounter \emph{anomalies}, or cohomological obstructions. Moreover, the invariant algebra can be trivial even when the coisotropic embedding has nontrivial geometry. Thus both the obstructions to quantization and the information lost on passing to invariant functions lead one to reconsider the object being quantized.

Cattaneo and Felder's approach is to quantize the complex from which the invariant algebra arises: the Chevalley--Eilenberg complex $\Gamma(\wedge^\bullet\nu_N)$ of the conormal Lie algebroid, whose degree-zero cohomology is $C^\infty(N)^{\mathcal F}$. They equip it with a homotopy Poisson, or $P_\infty$, structure: an $L_\infty$-structure whose brackets are multiderivations of the exterior product and whose induced bracket on degree-zero cohomology is the original Poisson bracket \cite{Cattaneo2007}. Its higher brackets are obtained from the ambient Poisson structure by Voronov's derived-bracket formalism \cite{voronov2005higher}; in the symplectic case, such a homotopy structure was also obtained by Oh and Park \cite{oh2005deformations}, as noted in \cite{Cattaneo2007}.

The crucial point is that this $P_\infty$-algebra always admits a quantization as a possibly curved $A_\infty$-algebra. The anomalies appear when one tries to pass from this quantization to a star product on invariant functions: one must remove the curvature and lift classical invariant functions to quantum ones. These are cohomological problems, with obstructions in degrees two and one, respectively. The homotopy algebra therefore provides a quantization even when this passage is obstructed, and identifies the conditions under which it recovers a quantization of the original reduced Poisson algebra.

In retrospect, one can recognize in the Cattaneo--Felder--Oh--Park construction a differential-geometric incarnation of a general principle in shifted Poisson geometry: an $n$-shifted coisotropic in an $n$-shifted Poisson stack inherits an $(n-1)$-shifted Poisson structure \cite{MelaniSafronov2018II}. Indeed, the symplectic groupoid $\G\tto M$ integrating $M$ presents a $1$-shifted symplectic stack, and the Lagrangian subgroupoid $\L\tto N$ integrating $N$ presents a $1$-shifted Lagrangian morphism
\[
[N/\L]\longrightarrow[M/\G].
\]
The general principle therefore predicts $0$-shifted Poisson geometry on the source. At the infinitesimal level, the source is modeled by the dg manifold $\nu_N^*[1]$, whose algebra of functions is precisely the conormal Lie algebroid complex. A $0$-shifted Poisson structure on this dg manifold is a $P_\infty$-structure with unary bracket equal to the Lie algebroid differential \cite{pridham2018outline}. This is the form of the structure constructed by Cattaneo--Felder and Oh--Park: the homotopy Poisson algebra introduced to address the quantization problem is also the infinitesimal structure expected from shifted geometry. This interpretation does not, however, make the differential-geometric construction a direct corollary of the abstract derived-algebraic results. Constructing the higher brackets on the conormal Lie algebroid complex requires a choice of tubular neighbourhood, and showing that different choices yield $L_\infty$-isomorphic structures is a further geometric result \cite{cattaneo2008equivalences}.

The shifted viewpoint suggests that this differential-geometric construction should extend beyond coisotropic submanifolds. Their integrations account for only a special class of $1$-shifted Lagrangians on differentiable stacks, so extending the classical construction to general $1$-shifted Lagrangians already raises new questions. Moreover, the shifted principle applies more generally to $1$-shifted coisotropics, whose description in terms of Lie groupoids and Dirac structures was developed in \cite{Maxence2025}. This provides a convenient framework in which to ask whether these more general objects, including $1$-shifted Lagrangians, admit an infinitesimal description and a quantization extending the classical coisotropic construction. More precisely:
\begin{enumerate}
    \item What is the infinitesimal analogue of a $1$-shifted coisotropic structure on a differentiable stack, and how does it integrate to a morphism of Lie groupoids?
    \item How can one construct the $0$-shifted Poisson structure associated with an infinitesimal $1$-shifted coisotropic as a homotopy Poisson algebra?
    \item How can one construct a deformation quantization of this homotopy Poisson structure?
\end{enumerate}
The goal of this paper is to address all three questions, under regularity hypotheses for the homotopy Poisson and quantization constructions.

The interest in this extension comes in part from generalized Hamiltonian reduction. Quasi-symplectic groupoids encompass ordinary, Poisson--Lie, and group-valued moment maps \cite{xu2004momentum}, while reduction along submanifolds and strong Dirac maps includes constructions such as symplectic cutting and implosion \cite{CrooksMayrand2022,bualibanu2026reduction}. Hamiltonian Dirac spaces and reduction levels can be described as $1$-shifted coisotropics, with reduction expressed through their intersections \cite{Maxence2025}. Extending the integration and quantization picture to these coisotropics is thus motivated by reduction procedures already present in differential geometry and geometric representation theory.

The passage from the shifted principle to the differential-geometric constructions in these questions requires substantial work. Abstract derived-algebraic existence results do not immediately produce the higher multiderivations on a smooth Lie algebroid complex, relate them to the Poisson bracket on invariant functions, or establish their independence of geometric choices.  General quantization results are available, notably Pridham's work on $0$-shifted Poisson and $1$-shifted coisotropic structures, including differential-geometric formulations \cite{Pridham2025DerivedPoisson,pridham2018outline}. Even with these results, one must identify the relevant structures in the Dirac model and establish their relation to the classical constructions. We carry out this passage directly: the same Lie algebroid supplies the infinitesimal object to be integrated and the cochain algebra to be quantized, and geometric Poisson models extend the Cattaneo--Felder construction to the Dirac setting (see also \v{S}evera \cite{Severa2005DiracQuantization} for a related approach to deformation quantization of Dirac structures).

We now explain the results, beginning with the integration question. To pass beyond the symplectic-groupoid picture, one replaces $\G$ by a quasi-symplectic groupoid $(\G,\omega)\tto(M,\phi)$: a Lie groupoid equipped with a multiplicative $2$-form $\omega$ and a closed $3$-form $\phi$ satisfying compatibility and nondegeneracy conditions \cite{xu2004momentum,bursztyn2004integration}. Up to symplectic Morita equivalence, these groupoids present $1$-shifted symplectic differentiable stacks. Correspondingly, a $1$-shifted Lagrangian morphism $\C\to\G$ is described by a $2$-form on $N:=\C^0$ satisfying suitable conditions \cite{PTVV2013}. The passage from Lagrangians to coisotropics replaces this $2$-form by a twisted Dirac structure $L_N$ \cite{Maxence2025}.

Quasi-symplectic groupoids integrate twisted Dirac structures, just as symplectic groupoids integrate Poisson structures; the relation with infinitesimal shifted symplectic geometry is developed in \cite{PymSafronov2020}. Writing $L_M$ for the $\phi$-twisted Dirac structure integrated by $\G$, differentiation of a $1$-shifted coisotropic therefore leads to a smooth map
\begin{equation}\label{blfbsbsc}
\ccc\colon(N,L_N,\psi)\longrightarrow(M,L_M,\phi)
\end{equation}
satisfying
\begin{equation}\label{8z2ialzl}
\ccc^*\phi=\psi.
\end{equation}
No forward- or backward-Dirac condition on $\ccc$ is imposed. Instead, the infinitesimal coisotropic condition singles out the bundle
\begin{equation}\label{w9xz42xi}
B_\ccc
:=\left\{((u,\alpha),(v,\beta))\in L_N \times L_M:
\ccc_*u=v,\ \alpha=\ccc^*\beta\right\}.
\end{equation}
For a $1$-shifted coisotropic, the Lie algebroid $A_\C$ identifies with $B_\ccc$. This observation led to the following candidate for the infinitesimal notion \cite[Definition 3.13]{Maxence2025}.

\begin{definition}\label{wccbgf99into}
An \defn{infinitesimal $1$-shifted coisotropic} is a map \eqref{blfbsbsc} satisfying \eqref{8z2ialzl} for which $B_\ccc$ has constant rank.
\end{definition}

The constant-rank condition makes $B_\ccc$ a Lie algebroid, with anchor $((u,\alpha),(v,\beta))\mapsto u$, whose projections to $L_N$ and $L_M$ are Lie algebroid morphisms (Theorem \ref{l5e5f2ec}\ref{rojq4m6n}). Thus the candidate already has the algebraic structure needed for integration. The remaining question is geometric: when these Lie algebroid data integrate, does $L_N$ automatically satisfy the compatibility condition with the quasi-symplectic form on $\G$? This is the integration conjecture of \cite[Conjecture 1.6 and \S3.4]{Maxence2025}, which we establish in the following form.

\begin{introtheorem}\label{wvihpe5p}
Let $\ccc\colon(N,L_N,\psi)\to(M,L_M,\phi)$ be an infinitesimal $1$-shifted coisotropic. Suppose that $(M,L_M,\phi)$ integrates to a quasi-symplectic groupoid $(\G,\omega)\tto(M,\phi)$ and that $B_\ccc$ integrates to a source-simply-connected Lie groupoid $\C\tto N$. Let $\ccc\colon\C\to\G$ integrate the projection $B_\ccc\to L_M$. Then $L_N$ is a $1$-shifted coisotropic structure on this groupoid morphism.
\end{introtheorem}

For a coisotropic inclusion $N\hookrightarrow(M,\pi)$ with $L_N=TN\oplus0$ and $L_M=\Gamma_\pi$, one has $B_\ccc\cong\nu_N^*$. Thus the theorem recovers the integration of coisotropic submanifolds to Lagrangian subgroupoids \cite{Cattaneo2004OnTheIntegration}. For a general inclusion with these Dirac structures, $B_\ccc$ identifies with $\nu_N^*\cap(\pi^\sharp)^{-1}(TN)$, and its constant-rank condition is precisely the pre-Poisson condition \cite{CattaneoZambon2009,cattaneo2007pre}. The theorem therefore also recovers Cattaneo--Zambon's integration of pre-Poisson submanifolds to isotropic subgroupoids \cite[Proposition 7.2]{CattaneoZambon2009}. In this case, $B_\ccc$ is also the stabilizer Lie algebroid governing reduction along a submanifold \cite{CrooksMayrand2022}. More generally, allowing an arbitrary twisted Dirac structure $L_M$ and taking $L_N$ to be the graph of a $2$-form recovers the reduction levels of \cite{bualibanu2026reduction}. Thus these examples connect the infinitesimal definition directly with the geometric data used in generalized reduction.

We next construct the Poisson algebra of invariant functions.

\begin{introtheorem}\label{6sipd2id}
Let $\ccc\colon(\C\tto N,L_N)\to(\G\tto M,\omega,\phi)$ be a $1$-shifted coisotropic, and let $L_M$ be the induced Dirac structure on $M$. Then $C^\infty(N)^\C$ carries a natural Poisson bracket
\[
\{f,g\}=df(v_g),\quad f,g\in C^\infty(N)^\C,
\]
where $v_g$ is any smooth vector field for which there is a smooth section $\beta_g$ of $\ccc^*T^*M$ satisfying
\[
(v_g,dg+\ccc^*\beta_g)\in L_N,
\quad
(\ccc_*v_g,\beta_g)\in L_M.
\]
Such pairs exist, and the bracket is independent of their choice. Coisotropic Morita equivalences induce isomorphisms of these Poisson algebras.
\end{introtheorem}

This bracket recovers the usual coisotropic reduced bracket when $N$ is a coisotropic submanifold of a Poisson manifold $M$. In the general setting of the theorem, it defines a Poisson structure on $N/\C$ whenever this quotient is a smooth manifold and the quotient map is a submersion. The invariant Poisson algebra remains available when the orbit space is singular.

To approach the quantization of this Poisson algebra, we follow the Cattaneo--Felder construction and pass to the Chevalley--Eilenberg algebra
\[
\a_\ccc:=\Gamma(\wedge^\bullet B_\ccc^*),
\quad d_{B_\ccc}\colon\a_\ccc^\bullet\longrightarrow\a_\ccc^{\bullet+1}.
\]
Its degree-zero cohomology is the algebra $C^\infty(N)^{B_\ccc}$ of functions annihilated by the anchor of $B_\ccc$. For an integrated $1$-shifted coisotropic with source-connected $\C$, this is $C^\infty(N)^\C$. We call the infinitesimal datum \emph{regular} if $L_N\cap TN$ and $L_M\cap TM$ both have constant rank, as in \Cref{1SCHTFL8N}. This hypothesis concerns the Dirac structures; it does not require the orbit space to be smooth.

\begin{introtheorem}\label{djsx2lmg}
Every regular infinitesimal $1$-shifted coisotropic determines a flat $P_\infty$-structure $(\ell_k)_{k\geq0}$ on $\a_\ccc$ with $\ell_1=d_{B_\ccc}$. Its $L_\infty$-isomorphism class is independent of the auxiliary choices in its construction. If the datum comes from a $1$-shifted coisotropic $\C\to\G$ with $\C$ source-connected, the bracket induced by $\ell_2$ on $H^0(\a_\ccc,\ell_1)$ is the Poisson bracket of Theorem \ref{6sipd2id}.
\end{introtheorem}

The construction proceeds through pre-Poisson submanifolds, following Cattaneo--Zambon \cite{CattaneoZambon2009}. Their coisotropic embedding theorem allows the Cattaneo--Felder construction to be applied inside a Poisson transversal, a method they already used to obtain deformation quantization of pre-Poisson reduction. Together with their uniqueness theorem and the equivalence results of \cite{cattaneo2008equivalences}, this yields a $P_\infty$-structure on the stabilizer cochains, independent of choices up to $L_\infty$-isomorphism. We use this construction as the starting point for the passage to $1$-shifted coisotropics.

For general Dirac data, we also use coisotropic Poisson extensions of Dirac manifolds with constant-rank kernel \cite{CattaneoZambon2009}. We replace the map $\ccc$ by its graph in a product of auxiliary twisted Poisson manifolds. The equality $\psi=\ccc^*\phi$ allows the product twist to be removed near the graph, yielding an ordinary pre-Poisson submanifold whose associated Lie algebroid is $B_\ccc$. Higher derived brackets then give the required structure. Comparison of the choices uses Poisson neighbourhood equivalences, including the results of \cite{CattaneoZambon2009}, together with the equivalence theory of higher derived brackets \cite{cattaneo2008equivalences}.

Finally, the relative formality theorem of Cattaneo--Felder quantizes these Poisson models, yielding a curved $A_\infty$-structure on the Lie algebroid complex. Its first-order terms recover the homotopy Poisson structure constructed above.

\begin{introtheorem}\label{d4fmvhyq}
Let $\ccc$ be a regular infinitesimal $1$-shifted coisotropic. There is a curved $A_\infty$-structure $(\mu_k)_{k\geq0}$ on $\a_\ccc\llbracket\epsilon\rrbracket$, with multidifferential structure maps, such that
\[
\mu_2(a,b)=a\wedge b+O(\epsilon),\quad
\mu_k=O(\epsilon)\quad(k\neq2),\quad
\mu_0=O(\epsilon^2),
\]
whose semiclassical limit is the $P_\infty$-structure of Theorem \ref{djsx2lmg}:
\[
\mu_1=\epsilon\,d_{B_\ccc}+O(\epsilon^2),\qquad
\mu_2(a,b)-(-1)^{|a||b|}\mu_2(b,a)=\epsilon\,\ell_2(a,b)+O(\epsilon^2),
\]
and more generally $\mu_k\circ\Alt_k=\pm\,\epsilon\,\ell_k+O(\epsilon^2)$ for every $k$, where $\Alt_k$ is graded antisymmetrization and the sign depends only on $k$ (\Cref{1SCQTHM01}). Its curved $A_\infty$-isomorphism class is independent of the choices of local models, Poisson transversals, tubular neighbourhoods, and normal-bundle connections, as well as of the globalization data fixed on $N$ (\Cref{1SCRMKTFC}).

If $H^2(B_\ccc)=0$, the curvature can be removed by a formal change of base point. If also $H^1(B_\ccc)=0$, the resulting degree-zero cohomology can be identified with $C^\infty(N)^{B_\ccc}\llbracket\epsilon\rrbracket$, and its induced product is a star product deforming the Poisson bracket induced by $\ell_2$ on $C^\infty(N)^{B_\ccc}$, which is the bracket of Theorem \ref{6sipd2id} when the datum comes from a $1$-shifted coisotropic $\C\to\G$ with $\C$ source-connected.
\end{introtheorem}

The two cohomological hypotheses have distinct roles: $H^2$ controls the removal of curvature, while $H^1$ controls the lifting of invariant functions. In the pre-Poisson case, the resulting deformation of invariant functions recovers the existence result of \cite[Theorem 6.2]{CattaneoZambon2009}. When $M$ is a point, $B_\ccc=L_N\cap TN$ and $\a_\ccc$ is the foliated de Rham complex of the regular foliation integrating $L_N\cap TN$; Theorems \ref{djsx2lmg} and \ref{d4fmvhyq} then recover \cite[Proposition 8.6 and Theorem 8.5]{CattaneoZambon2009}. The homotopy construction also applies when $H^1(B_\ccc)$ or $H^2(B_\ccc)$ does not vanish, retaining the curvature and higher operations as part of the quantized structure.

These results bring us back to the question of quantizing a reduced phase space. In Hamiltonian reduction, one restricts to a level set of the moment map and then divides by the corresponding symmetries. The shifted coisotropic framework also allows more general reduction procedures, such as symplectic reduction along a submanifold \cite{CrooksMayrand2022} and reduction of strong Dirac maps \cite{bualibanu2026reduction}; see \S\ref{uxxbg680}. Under the hypotheses above, the Lie algebroid governing the reduction supplies a homotopy Poisson algebra whose degree-zero cohomology carries the reduced Poisson bracket, and this homotopy algebra admits a curved $A_\infty$-quantization. When the cohomological obstructions vanish, it yields a star product on the invariant functions. Even when the reduced space is smooth, its Poisson algebra does not retain all the information carried by the reduction data. The homotopy construction retains part of this information and makes explicit the obstructions to obtaining an associative deformation of invariant functions. 


\subsection*{Organization of the paper}
In \S\ref{1SCINTSEC}, we recall $1$-shifted coisotropic structures on Lie groupoids and prove the integration theorem. In \S\ref{1SC69WYG8}, we construct the Poisson bracket on invariant functions and show that it is invariant under coisotropic Morita equivalence; this invariance also reduces the Jacobi identity to the case of an inclusion. In \S\ref{1SCYR8NZ4}, we construct the flat $P_\infty$-structure first for pre-Poisson submanifolds and then for regular infinitesimal $1$-shifted coisotropics through local Poisson models, proving Theorem \ref{djsx2lmg}. Finally, \S\ref{1SCC6G85S} quantizes these structures using the relative formality theorem and proves Theorem \ref{d4fmvhyq}.

\subsection*{Acknowledgments}

We thank Hamza Bakhouch, Damien Calaque, Peter Crooks, Matias del Hoyo and Mohamed Moussadek Maiza for useful discussions. Y.A.M.\ thanks the Institut des sciences math\'ematiques (ISM) for financial support. M.M.\ acknowledges support from a Discovery Grant (RGPIN-2023-04587) from the Natural Sciences and Engineering Research Council of Canada (NSERC).

\section{Differentiation and integration of 1-shifted coisotropics}\label{1SCINTSEC}

In this section, we recall the description of $1$-shifted coisotropic structures in terms of Lie groupoids and twisted Dirac structures \cite{Maxence2025}, and prove the integration theorem announced in Theorem \ref{wvihpe5p}. We follow the conventions of \cite{Mackenzie2005}.

The integration theorem is new even in the more familiar setting of $1$-shifted Lagrangians. Working directly with $1$-shifted coisotropics provides a convenient framework for proving this result, as well as greater generality. In particular, the coisotropic formulation is symmetric: both the source and the target carry twisted Dirac structures, whereas the Lagrangian case involves a $2$-form on the source and a twisted Dirac structure on the target.

\subsection{Preliminaries on Lie algebroids and multiplicative forms}\label{1SCA7Z3Y4}

We recall some standard facts about Lie algebroids, their cohomology, and multiplicative forms, fixing the notation and conventions used below.

A \defn{Lie algebroid} over a manifold $M$ is a vector bundle $A\to M$ with a Lie bracket on $\Gamma(A)$ and a bundle map $\rho\colon A\to TM$, called the \defn{anchor}, satisfying
\[
[v,fw]=f[v,w]+\rho(v)(f)w,
\quad v,w\in\Gamma(A),\quad f\in C^\infty(M).
\]
The graded commutative algebra $\Gamma(\wedge^\bullet A^*)$ carries the \defn{Chevalley--Eilenberg differential} $d_A$, given on $\eta\in\Gamma(\wedge^kA^*)$ by
\[
\begin{aligned}
(d_A\eta)(v_0,\ldots,v_k)
&=\sum_{i=0}^k(-1)^i\rho(v_i)
\bigl(\eta(v_0,\ldots,\widehat{v_i},\ldots,v_k)\bigr)\\
&\quad+\sum_{0\leq i<j\leq k}(-1)^{i+j}
\eta([v_i,v_j],v_0,\ldots,\widehat{v_i},\ldots,\widehat{v_j},\ldots,v_k).
\end{aligned}
\]
It is a derivation of degree $1$ satisfying $d_A^2=0$, and is determined by its values on functions and $1$-forms:
\[
\begin{aligned}
(d_Af)(v)&=\rho(v)(f),\\
(d_A\alpha)(v,w)
&=\rho(v)\bigl(\alpha(w)\bigr)
-\rho(w)\bigl(\alpha(v)\bigr)-\alpha([v,w]).
\end{aligned}
\]
The cohomology of this differential graded algebra is the \defn{Lie algebroid cohomology}
\begin{equation}\label{1SCYQ3DDG}
H^k(A,d_A):=
\frac{\ker\bigl(d_A\colon\Gamma(\wedge^kA^*)\to\Gamma(\wedge^{k+1}A^*)\bigr)}
{\operatorname{im}\bigl(d_A\colon\Gamma(\wedge^{k-1}A^*)\to\Gamma(\wedge^kA^*)\bigr)}.
\end{equation}
In particular,
\[
H^0(A,d_A)=\{f\in C^\infty(M):\rho(v)(f)=0
\text{ for all }v\in\Gamma(A)\}.
\]

Every Lie groupoid $\mathcal K\tto X$, with source $\sss$, target $\ttt$, and unit section $\mathsf{1}$, has a Lie algebroid
\[
A_{\mathcal K}=\ker\sss_*|_{\mathsf{1}(X)},
\quad
\rho_{\mathcal K}=\ttt_*|_{A_{\mathcal K}},
\]
viewed as a vector bundle over $X$. Its bracket is induced by the commutator of right-invariant vector fields \cite{Mackenzie2005}.
A $2$-form $\theta$ on $\mathcal K$ is \defn{multiplicative} if
\[
\mmm^*\theta=\ppp_1^*\theta+\ppp_2^*\theta
\]
on the manifold $\mathcal K^{(2)}$ of composable pairs, where $\mmm$ is multiplication and $\ppp_i$ are the projections. Its \defn{infinitesimal multiplicative form} (IM form) is the bundle map
\[
\sigma_\theta\colon A_{\mathcal K}\longrightarrow T^*X,
\quad
\sigma_\theta(a)(w)=\theta_{\mathsf{1}_x}(a,\mathsf{1}_*w),
\]
for $a\in(A_{\mathcal K})_x$ and $w\in T_xX$.

\subsection{1-shifted coisotropics and Morita equivalence}

We recall the Lie groupoid description of $1$-shifted coisotropic structures and their Morita equivalences from \cite{Maxence2025}. We also establish a relation between the base Dirac structures of symplectically Morita equivalent groupoids, which will be used later to prove Morita invariance of the Poisson bracket on invariant functions.

Let $(\G,\omega)\tto(M,\phi)$ be a quasi-symplectic groupoid, $\C\tto N$ a Lie groupoid, and $\ccc\colon\C\to\G$ a Lie groupoid morphism. We use the same symbol for its base map $N\to M$. The quasi-symplectic structure induces a $\phi$-twisted Dirac structure $L_M$ on $M$, with a canonical Lie algebroid isomorphism \cite{bursztyn2004integration}
\[
(\rho_\G,\sigma_\omega)\colon A_\G\overset{\cong}{\longrightarrow}L_M.
\]

A \defn{$1$-shifted coisotropic structure} on $\ccc$ is a twisted Dirac structure $(L_N,\psi)$ on $N$ satisfying the following conditions.
\begin{enumerate}
\item \emph{Compatibility.} We have
\[
\psi=\ccc^*\phi,
\quad
\ttt^*L_N=\sss^*L_N+\Gamma_{\ccc^*\omega}.
\]
Here $\sss$ and $\ttt$ are the source and target of $\C$, pullbacks are backward images of Dirac structures, and addition of $\Gamma_{\ccc^*\omega}$ is the gauge transformation by $\ccc^*\omega$; see \cite[\S2.2]{Maxence2025}.
\item \emph{Nondegeneracy.} The map
\begin{equation}\label{goeol8yy}
\begin{aligned}
A_\C&\longrightarrow L_N\times_\ccc A_\G,\\
a&\mapsto
\bigl((\rho_\C(a),\ccc^*\sigma_\omega(\ccc_*a)),\ccc_*a\bigr)
\end{aligned}
\end{equation}
is fibrewise bijective,\footnote{Our $1$-shifted coisotropics are the strong $1$-shifted coisotropics of \cite{Maxence2025}, whose general definition allows this map only to be surjective. The stronger notion is more natural for the purposes of this paper, since it identifies $A_\C$ with the Lie algebroid determined by the infinitesimal Dirac data. We therefore use the term coisotropic for this notion throughout.} where
\[
L_N\times_\ccc A_\G
:=\left\{((v,\alpha),a)\in L_N\oplus\ccc^*A_\G:
\ccc_*v=\rho_\G(a),\quad \alpha=\ccc^*\sigma_\omega(a)\right\}.
\]
\end{enumerate}
A \defn{$1$-shifted Lagrangian structure} is the special case where $L_N$ is the graph $\Gamma_\gamma = \{(v, \iota_v\gamma) : v \in TN\}$ of a $2$-form $\gamma\in\Omega^2(N)$ \cite{PTVV2013,Maxence2025}. Another example is the identity morphism $\G\to\G$, equipped with the induced Dirac structure $L_M$ on its base \cite[Proposition 3.6]{Maxence2025}.

Let $\ccc_i\colon\C_i\to\G_i$ be Lie groupoid morphisms, with $(\G_i,\omega_i,\phi_i)$ quasi-symplectic for $i=1,2$. A \defn{Morita equivalence} between them consists of a diagram of the following form \cite[Definition 7.1]{Maxence2025}:
\begin{equation}\label{1SCT3HPM7}
\begin{tikzcd}[ampersand replacement=\&,cramped,column sep=small,
row sep=small]
\&\& {\mathcal K} \\
{\C_1} \&\&\&\& {\C_2} \\
\&\& {\mathcal L} \\
{\G_1} \&\&\&\& {\G_2}
\arrow["{\psi_1}"', from=1-3, to=2-1]
\arrow["{\psi_2}", from=1-3, to=2-5]
\arrow["h", from=1-3, to=3-3]
\arrow["{\theta_1}", Rightarrow, shorten <=6pt, shorten >=6pt, from=2-1, to=3-3]
\arrow["{\ccc_1}"', from=2-1, to=4-1]
\arrow["{\theta_2}"', Rightarrow, shorten <=6pt, shorten >=6pt, from=2-5, to=3-3]
\arrow["{\ccc_2}", from=2-5, to=4-5]
\arrow["{\varphi_1}"', from=3-3, to=4-1]
\arrow["{\varphi_2}", from=3-3, to=4-5]
\end{tikzcd}
\end{equation}
in which the $\psi_i$ are Morita morphisms, the span $\G_1\leftarrow\L\to\G_2$ is a symplectic Morita equivalence with $1$-shifted Lagrangian structure $\gamma\in\Omega^2(\L^0)$, and $\theta_i\colon\ccc_i\psi_i\Rightarrow\varphi_i h$ are natural transformations.

If $\ccc_1$ and $\ccc_2$ carry $1$-shifted coisotropic structures $L_1$ and $L_2$, the diagram \eqref{1SCT3HPM7} is a \defn{coisotropic Morita equivalence} if in addition
\begin{equation}\label{1SCU7G6KZ}
\psi_2^*L_2 = \psi_1^*L_1 + \Gamma_\delta,
\quad \delta := -h^*\gamma + \theta_1^*\omega_1 - \theta_2^*\omega_2.
\end{equation}
The $2$-form $\delta$ is called the \defn{connecting $2$-form} of the equivalence. Two $1$-shifted coisotropics are equivalent in this way if and only if they represent isomorphic coisotropic structures on the underlying stacks \cite[\S10]{Maxence2025}.

The proof of Morita invariance in \Cref{1SCD764G8} requires the following result.

\begin{lemma}\label{1SCT66PWH}
Let $(\G_1 \tto M_1,\omega_1,\phi_1)$ and $(\G_2 \tto M_2,\omega_2,\phi_2)$ be quasi-symplectic groupoids with induced base Dirac structures $L_{M_1},L_{M_2}$. Suppose there is a symplectic Morita equivalence
\[
\begin{tikzcd}
  & \L \arrow[dl, "\varphi_1"'] \arrow[dr, "\varphi_2"] & \\
  \G_1 & & \G_2
\end{tikzcd}
\]
with $1$-shifted Lagrangian structure $\gamma\in\Omega^2(\L^0)$. Then $\varphi_1^*L_{M_1} = \varphi_2^*L_{M_2} + \Gamma_\gamma$.
\end{lemma}

\begin{proof}
The symplectic Morita equivalence gives
\begin{equation}\label{1SCD9DUZ4}
\begin{aligned}
\varphi_1^*\omega_1-\varphi_2^*\omega_2
&=\ttt^*\gamma-\sss^*\gamma,\\
\varphi_1^*\phi_1-\varphi_2^*\phi_2&=-d\gamma.
\end{aligned}
\end{equation}
Both sides of the claimed equality are pointwise Lagrangian, so it suffices to prove one inclusion. Let
$(v,\varphi_1^*\alpha)\in\varphi_1^*L_{M_1}$. Choose $a\in A_{\G_1}$ such that
\[
\varphi_{1*}v=\rho_{\G_1}(a),
\quad
\alpha=\sigma_{\omega_1}(a).
\]
By \cite[Lemma 5.5]{Maxence2025}, there is $\widetilde a\in A_\L$ with
\[
\varphi_{1*}\widetilde a=a,
\quad
\rho_\L(\widetilde a)=v.
\]
Set $b=\varphi_{2*}\widetilde a$, so that $\rho_{\G_2}(b)=\varphi_{2*}v$.
Taking IM forms in the first identity of \eqref{1SCD9DUZ4} and evaluating at $\widetilde a$ gives
\[
\varphi_1^*\alpha-\varphi_2^*\sigma_{\omega_2}(b)=i_v\gamma.
\]
Since $(\varphi_{2*}v,\sigma_{\omega_2}(b))\in L_{M_2}$, it follows that $(v,\varphi_1^*\alpha)\in\varphi_2^*L_{M_2}+\Gamma_\gamma$.
\end{proof}

\subsection{Infinitesimal 1-shifted coisotropics}

We now identify the infinitesimal data associated with a $1$-shifted coisotropic and prove that these data recover the coisotropic structure whenever the relevant Lie algebroids are integrable. We conclude with examples relating this correspondence to coisotropic submanifolds, reduction levels, and strong Dirac maps.

Given a $1$-shifted coisotropic
$\ccc\colon(\C\tto N,L_N)\to(\G,\omega,\phi)$, the base map is a smooth map between twisted Dirac manifolds
\[
\ccc\colon(N,L_N,\psi)\longrightarrow(M,L_M,\phi)
\]
with $\psi=\ccc^*\phi$, but satisfying no other compatibility conditions between the Dirac structures such as forward or backward Dirac. Under the identification $A_\G\cong L_M$, the nondegeneracy condition \eqref{goeol8yy} gives a vector bundle isomorphism
\[
A_\C\cong B_\ccc,
\]
where
\[
B_\ccc:=\left\{((u,\alpha),(v,\beta))\in L_N\oplus\ccc^*L_M:
\ccc_*u=v,\quad\alpha=\ccc^*\beta\right\}.
\]
In particular, $B_\ccc$ has constant rank and depends only on the base map and the two twisted Dirac structures. This motivates the following definition, proposed in \cite[Definition 3.13]{Maxence2025}.

\begin{definition}\label{wccbgf99}
An \defn{infinitesimal $1$-shifted coisotropic} is a smooth map
\[
\ccc\colon(N,L_N,\psi)\longrightarrow(M,L_M,\phi)
\]
between twisted Dirac manifolds such that $\psi=\ccc^*\phi$ and $B_\ccc$ has constant rank.
\end{definition}

Recall that a twisted Dirac manifold $(M,L_M,\phi)$ is integrable if $L_M$ is integrable as a Lie algebroid. In that case, its source-simply-connected integration carries a unique multiplicative $2$-form making it a quasi-symplectic groupoid with induced Dirac structure $L_M$ \cite{bursztyn2004integration}. We call any quasi-symplectic groupoid with this induced twisted Dirac structure an integration of $(M,L_M,\phi)$.

The following theorem establishes the integration conjecture of \cite{Maxence2025}.

\begin{theorem}\label{l5e5f2ec}
Let $\ccc\colon(N,L_N,\psi)\to(M,L_M,\phi)$ be an infinitesimal $1$-shifted coisotropic.
\begin{enumerate}[label={\textup{(\arabic*)}}]
\item \label{rojq4m6n}
$B_\ccc$ is a smooth vector bundle over $N$ with a unique Lie algebroid structure such that the two projections 
\[
\begin{tikzcd}
    & B_\ccc \arrow{dl} \arrow{dr} & \\
    L_N & & L_M
\end{tikzcd}
\]
are Lie algebroid morphisms.
\item \label{fot1xcbs}
Suppose that $(M,L_M,\phi)$ integrates to a quasi-symplectic groupoid $(\G,\omega)\tto(M,\phi)$ and that $B_\ccc$ integrates to a source-simply-connected Lie groupoid $\C\tto N$. Let $\ccc\colon\C\to\G$ integrate the projection $B_\ccc\to L_M\cong A_\G$. Then $L_N$ is a $1$-shifted coisotropic structure on $\ccc$.
\end{enumerate}
\end{theorem}

\begin{proof}
\ref{rojq4m6n}
The set $B_\ccc$ is the kernel of the vector bundle map
\[
L_N\oplus\ccc^*L_M\longrightarrow \ccc^*TM\oplus T^*N,
\quad
((u,\alpha),(v,\beta))
\mapsto
(\ccc_*u-v,\alpha-\ccc^*\beta).
\]
Hence, the constant-rank assumption implies that $B_\ccc$ is a
smooth vector bundle.

Let $P=N\times M$, let $j\colon N\to P$ be the graph embedding
$j(x)=(x,\ccc(x))$, and write $S=j(N)$. Consider the twisted Dirac
structure
\[
D=L_N\times(-L_M)
\]
on $P$, with background $3$-form
\[
\eta=\ppp_N^*\psi-\ppp_M^*\phi,
\]
where $-L_M=\{(v,-\beta):(v,\beta)\in L_M\}$. The map
\[
L_N\times L_M\longrightarrow D,
\quad
((u,\alpha),(v,\beta))
\mapsto ((u,v),(\alpha,-\beta))
\]
is an isomorphism of Lie algebroids. Under this identification,
$B_\ccc$ corresponds to
\[
E=D|_S\cap\bigl(TS\oplus(TS)^\circ\bigr),
\]
where $(TS)^\circ\subset T^*P|_S$ is the conormal bundle.
Indeed, $(u,v)\in TS$ precisely when $v=\ccc_*u$, and
$(\alpha,-\beta)\in(TS)^\circ$ precisely when
$\alpha=\ccc^*\beta$. Moreover,
\[
j^*\eta=\psi-\ccc^*\phi=0.
\]

We claim that $E$ is a Lie subalgebroid of $D$. Its anchor takes
values in $TS$ by construction. Let $(X_i,\xi_i)$, $i=1,2$,
be local sections of $D$ whose restrictions to $S$ take values
in $E$. Then
\[
X_i\circ j=j_*u_i,
\quad
j^*\xi_i=0
\]
for some vector fields $u_i$ on $N$. Their twisted Dorfman
bracket is
\[
\llbracket(X_1,\xi_1),(X_2,\xi_2)\rrbracket_\eta
=
\bigl([X_1,X_2],
\mathcal L_{X_1}\xi_2-i_{X_2}d\xi_1
+i_{X_2}i_{X_1}\eta\bigr).
\]
The vector component is tangent to $S$, and the covector
component annihilates $TS$, since
\begin{align*}
j^*\bigl(
\mathcal L_{X_1}\xi_2-i_{X_2}d\xi_1
+i_{X_2}i_{X_1}\eta
\bigr)
&=
\mathcal L_{u_1}j^*\xi_2
-i_{u_2}d j^*\xi_1
+i_{u_2}i_{u_1}j^*\eta\\
&=0.
\end{align*}
Since $D$ is involutive, this bracket restricts to a section
of $E$, proving the claim.

The induced Lie algebroid structure on $E$ therefore gives
a Lie algebroid structure on $B_\ccc$, with anchor
\[
((u,\alpha),(v,\beta))\mapsto u.
\]
The two projections are Lie algebroid morphisms, being the
compositions of the inclusion into the product Lie algebroid
$L_N\times L_M$ with its projections.

Finally, any Lie algebroid structure on $B_\ccc$ for which
both projections are morphisms makes their product
$B_\ccc\to L_N\times L_M$ a Lie algebroid morphism over $j$.
Since this map is a vector bundle embedding over an embedding
of manifolds, its bracket and anchor must agree with the
induced subalgebroid structure. This proves uniqueness.

\ref{fot1xcbs}
Write $\Omega=\ccc^*\omega$. We first show that
\[
\ttt^*L_N=\sss^*L_N+\Gamma_\Omega
\]
holds on a neighbourhood of the unit section.

Let $(\mathcal H,\gamma)\tto(N,\psi)$ be a local
quasi-symplectic integration of $L_N$. By Part \ref{rojq4m6n}, the
projection $B_\ccc\to L_N$ is a Lie algebroid morphism.
It therefore integrates to a local Lie groupoid morphism
\[
f\colon U\longrightarrow\mathcal H
\]
over the identity of $N$, where $U\subset\C$ is a
neighbourhood of the unit section. The proof of
\cite[Proposition 3.6]{Maxence2025} applies to local
quasi-symplectic groupoids and gives
\[
\ttt^*L_N=\sss^*L_N+\Gamma_\gamma
\]
on $\mathcal H$. Since backward images commute with gauge transformations and $f$ covers the identity of $N$, pulling back by $f$ gives
\[
\ttt^*L_N=\sss^*L_N+\Gamma_{f^*\gamma}
\]
on $U$.

Since $f$ covers the identity of $N$, we have
\[
d(f^*\gamma)
=f^*(d\gamma)
=f^*(\sss^*\psi-\ttt^*\psi)
=\sss^*\psi-\ttt^*\psi.
\]
Similarly, since $\ccc$ is a groupoid morphism and
$\ccc^*\phi=\psi$,
\[
d\Omega
=\ccc^*(d\omega)
=\ccc^*(\sss^*\phi-\ttt^*\phi)
=\sss^*(\ccc^*\phi)-\ttt^*(\ccc^*\phi)
=\sss^*\psi-\ttt^*\psi.
\]
Thus, $f^*\gamma$ and $\Omega$ have the same differential. We now compare their IM forms.
Let
\[
a=((u,\alpha),(v,\beta))\in(B_\ccc)_x,
\quad w\in T_xN.
\]
Under the identifications $A_{\mathcal H}=L_N$ and
$A_\G=L_M$, we have
\[
f_*a=(u,\alpha),
\quad
\ccc_*a=(v,\beta).
\]
Since $f$ covers the identity of $N$ and $\gamma$
integrates $L_N$,
\[
\sigma_{f^*\gamma}(a)(w)
=\gamma(f_*a,\mathsf{1}_*w)
=\sigma_\gamma(u,\alpha)(w)
=\alpha(w).
\]
Similarly, since $\omega$ integrates $L_M$,
\[
\sigma_\Omega(a)(w)
=\omega(\ccc_*a,\mathsf{1}_*\ccc_*w)
=\sigma_\omega(v,\beta)(\ccc_*w)
=\beta(\ccc_*w)
=\alpha(w),
\]
where the last equality follows from
$\alpha=\ccc^*\beta$ in the definition of $B_\ccc$.
Thus, $f^*\gamma$ and $\Omega$ have the same
IM form.
The uniqueness argument in
\cite[Proposition 3.5(iii)]{bursztyn2004integration}
applies equally to local Lie groupoids, after shrinking
to a neighbourhood of the unit section with connected
source fibres. Hence, $f^*\gamma=\Omega$ after shrinking $U$.

Now let
\[
\mathcal K=
\{g\in\C:
(\ttt^*L_N)_g=(\sss^*L_N+\Gamma_\Omega)_g\}.
\]
We show that $\mathcal K$ is a subgroupoid of $\C$.
It contains the unit section since $U\subset\mathcal K$.
Let $\iota\colon\C\to\C$ be inversion. Since
$\sss\circ\iota=\ttt$, $\ttt\circ\iota=\sss$, and
$\iota^*\Omega=-\Omega$, pulling back the defining
identity at $g\in\mathcal K$ by $\iota$ gives
\[
(\sss^*L_N)_{g^{-1}}
=(\ttt^*L_N+\Gamma_{-\Omega})_{g^{-1}}.
\]
Applying the gauge transformation by $\Omega$ shows
that $g^{-1}\in\mathcal K$.

Let $h,g\in\mathcal K$ satisfy $\sss(h)=\ttt(g)$, and let
\[
(v,\ttt^*\alpha)\in(\ttt^*L_N)_{hg},
\quad
(\ttt_*v,\alpha)\in(L_N)_{\ttt(h)}.
\]
Since multiplication is a submersion, choose
\[
(v_h,v_g)\in T_{(h,g)}\C^{(2)}
\]
such that $\mmm_*(v_h,v_g)=v$. Then
$\ttt_*v_h=\ttt_*v$, so the defining identity for
$h\in\mathcal K$ gives
\[
\ttt^*\alpha=\sss^*\beta+i_{v_h}\Omega,
\quad
(\sss_*v_h,\beta)\in(L_N)_{\sss(h)}
\]
for some $\beta\in T^*_{\sss(h)}N$.
Since $\sss_*v_h=\ttt_*v_g$, the defining identity for
$g\in\mathcal K$ similarly gives
\[
\ttt^*\beta=\sss^*\delta+i_{v_g}\Omega,
\quad
(\sss_*v_g,\delta)\in(L_N)_{\sss(g)}
\]
for some $\delta\in T^*_{\sss(g)}N$.

At $(h,g)$, multiplicativity of $\Omega$ now implies
\begin{align*}
\mmm^*(\ttt^*\alpha-\sss^*\delta-i_v\Omega)
&=
\ppp_1^*(\ttt^*\alpha-\sss^*\beta-i_{v_h}\Omega)\\
&\quad+\ppp_2^*(\ttt^*\beta-\sss^*\delta-i_{v_g}\Omega)\\
&=0.
\end{align*}
Since $\mmm$ is a submersion, it follows that
\[
\ttt^*\alpha=\sss^*\delta+i_v\Omega.
\]
Moreover,
\[
(\sss_*v,\delta)
=(\sss_*v_g,\delta)\in(L_N)_{\sss(hg)}.
\]
Hence,
\[
(\ttt^*L_N)_{hg}
\subseteq(\sss^*L_N+\Gamma_\Omega)_{hg}.
\]
Both sides are Lagrangian subspaces, so they are equal.
Thus, $hg\in\mathcal K$, proving that $\mathcal K$ is
a subgroupoid.

Since $\C$ is source-connected, every arrow $g$ can be joined to $\mathsf{1}_{\sss(g)}$ by a path $g(t)$ in its source fibre. A sufficiently fine subdivision $0=t_0<\cdots<t_k=1$ satisfies
\[
g(t_i)g(t_{i-1})^{-1}\in U\subset\mathcal K
\]
for every $i$. Their product is $g$, so $\mathcal K=\C$.
This proves compatibility.

Finally, under the identifications
$A_\C=B_\ccc$ and $A_\G=L_M$, the map
\eqref{goeol8yy} is the identity of $B_\ccc$.
It is therefore bijective, proving nondegeneracy.
\end{proof}

In particular, this proves Theorem \ref{wvihpe5p} from the introduction.

For a $1$-shifted coisotropic $\ccc\colon\C\to\G$, the canonical identification $A_\C\cong B_\ccc$ is an isomorphism of Lie algebroids. Indeed, its projection to $L_M$ is the differential of $\ccc$, while its projection to $L_N$ is
\[
(\rho_\C,\sigma_{\ccc^*\omega})\colon A_\C\longrightarrow L_N,
\]
which is a Lie algebroid morphism by the IM-form identities
\cite[Proposition 3.5]{bursztyn2004integration}.
The uniqueness in Part \ref{rojq4m6n} gives the claim.
Thus, when $\C$ is source-simply-connected, its infinitesimal data recover both $\C$ and the morphism $\ccc\colon\C\to\G$, with the target integration fixed.

Theorem \ref{l5e5f2ec} therefore justifies Definition \ref{wccbgf99}: once the relevant Lie algebroids are integrable, there is no further obstruction to the coisotropic compatibility condition. We record some special cases.

\begin{example}\label{integrationexamples}\
\begin{enumerate}
\item \emph{Coisotropic and pre-Poisson submanifolds.}
Let $(M,\pi)$ be a Poisson manifold and
$\ccc\colon N\hookrightarrow M$ a coisotropic submanifold.
Equip $M$ with $L_M=\Gamma_\pi$ and $N$ with
$L_N=TN\oplus0$. Then
\[
(TN)^\circ\overset{\cong}{\longrightarrow}B_\ccc,
\quad
\beta\mapsto
\bigl((\pi^\sharp\beta,0),(\pi^\sharp\beta,\beta)\bigr),
\]
where $(TN)^\circ\subset T^*M|_N$ is the conormal bundle.
The anchor is $\pi^\sharp|_{(TN)^\circ}$, and the bracket
is inherited from the cotangent Lie algebroid of $M$.

Suppose that $(M,\pi)$ integrates to a symplectic groupoid
$(\G,\omega)\tto M$. Since $(TN)^\circ$ is a Lie
subalgebroid of $T^*M$, it is integrable. Let $\C$ be its
source-simply-connected integration. The induced morphism
$\ccc\colon\C\to\G$ is an immersion, and the compatibility
condition of Theorem \ref{l5e5f2ec} reduces to
\[
\ccc^*\omega=0.
\]
Moreover,
\[
\dim\C=\dim N+\operatorname{rank}(TN)^\circ
=\dim M=\tfrac12\dim\G.
\]
Thus, $\ccc$ is a Lagrangian immersion of Lie groupoids,
recovering Cattaneo's integration of coisotropic
submanifolds \cite{Cattaneo2004OnTheIntegration}.

More generally, for an arbitrary submanifold $N\subset M$
with the same Dirac structures, one has
\[
B_\ccc\cong
(TN)^\circ\cap(\pi^\sharp)^{-1}(TN).
\]
The constant-rank condition is equivalent to the usual pre-Poisson condition that $TN+\pi^\sharp((TN)^\circ)$ have constant rank \cite{CattaneoZambon2009}. Its integration is
an isotropic immersion $\C\to\G$, again equipped with
the trivial $1$-shifted Lagrangian structure, recovering \cite[Proposition 7.2]{CattaneoZambon2009}. It need
not be Lagrangian as an immersion. These integrations
are also the stabilizer subgroupoids used in symplectic
reduction along a submanifold
\cite{CrooksMayrand2022}.

\item \emph{Infinitesimal $1$-shifted Lagrangians and
reduction levels.}
Suppose that $L_N=\Gamma_\gamma$ for a $2$-form
$\gamma\in\Omega^2(N)$. The twisting condition becomes
\[
d\gamma=-\ccc^*\phi,
\]
and the defining bundle is naturally identified with
\[
B_\ccc\cong
\left\{
(u,(v,\beta))\in TN\oplus\ccc^*L_M:
\ccc_*u=v,\quad i_u\gamma=\ccc^*\beta
\right\}.
\]
Thus, an infinitesimal $1$-shifted Lagrangian consists
of such a pair $(\ccc,\gamma)$ for which this bundle
has constant rank. Whenever the integrations in
Theorem \ref{l5e5f2ec} exist, its compatibility
condition becomes
\[
\Gamma_{\ttt^*\gamma}
=\Gamma_{\sss^*\gamma+\ccc^*\omega},
\]
or equivalently,
\[
\ccc^*\omega=\ttt^*\gamma-\sss^*\gamma.
\]
Together with nondegeneracy, this makes $\gamma$
a $1$-shifted Lagrangian structure on $\ccc\colon\C\to\G$.

When $\ccc$ is the inclusion of a submanifold, the
second projection identifies $B_\ccc$ with
\[
\left\{
(v,\beta)\in L_M|_N:
v\in TN,\quad \ccc^*\beta=i_v\gamma
\right\}.
\]
The conditions on $(N,\gamma)$ are therefore exactly
those of a reduction level in
\cite[Definition 2.1]{bualibanu2026reduction}, and this
bundle is its stabilizer Lie algebroid. If $L_M$ is
integrable, so is this Lie subalgebroid. The theorem
then shows that its source-simply-connected integration
satisfies the multiplicative-form identity required
of a stabilizer subgroupoid in
\cite[Definition 5.1]{bualibanu2026reduction}.
Thus, every reduction level in an integrable twisted Dirac manifold admits such a stabilizer integration.

\item \emph{Strong Dirac maps.}
Suppose that
\[
\ccc\colon(N,L_N,\psi)\longrightarrow(M,L_M,\phi)
\]
is a strong Dirac map, i.e.\
\[
\ccc_*L_N=L_M,
\quad
\ker\ccc_*\cap(L_N\cap TN)=0,
\quad
\psi=\ccc^*\phi.
\]
For each $(v,\beta)\in(L_M)_{\ccc(x)}$, there is a unique
$u\in T_xN$ such that
\[
\ccc_*u=v,
\quad
(u,\ccc^*\beta)\in(L_N)_x.
\]
Consequently, the second projection gives a vector
bundle isomorphism
\[
B_\ccc\cong\ccc^*L_M.
\]
Its anchor is the infinitesimal action of $L_M$ on $N$,
so $B_\ccc$ is the corresponding action Lie algebroid.

If this action is complete and $\G$ is the
source-simply-connected integration of $L_M$, it
integrates to an action of $\G$ on $N$. Then
$\C=\G\ltimes N$, and the integrated morphism
$\ccc\colon\C\to\G$ is the projection.
Writing $\mathsf a\colon\G\ltimes N\to N$ for the
action, the compatibility condition becomes
\[
\mathsf a^*L_N
=\ppp_N^*L_N+\Gamma_{\ppp_\G^*\omega}.
\]
This recovers the Hamiltonian groupoid action
associated with a complete strong Dirac map
\cite[Theorem 4.7]{bur-cra:05}, and its interpretation
as a $1$-shifted coisotropic in \cite{Maxence2025}.
\end{enumerate}
\end{example}

Finally, the Lie algebroid appearing in the integration
theorem also determines the invariant functions.
Since $\C$ is source-connected,
\begin{align*}
C^\infty(N)^\C
&=\{f\in C^\infty(N):\rho_{B_\ccc}(b)f=0
\text{ for all }b\in\Gamma(B_\ccc)\}\\
&=H^0\bigl(\Gamma(\wedge^\bullet B_\ccc^*),d_{B_\ccc}\bigr).
\end{align*}
Thus, the same infinitesimal object that determines the integrating groupoid supplies the cochain algebra on which we will construct the homotopy Poisson structure and its quantization.

\section{The Poisson algebra of invariant functions}\label{1SC69WYG8}
Let $(\G, \omega) \tto (M, \phi)$ be a quasi-symplectic groupoid and $\ccc\colon\C\to\G$ a morphism of Lie groupoids equipped with a $1$-shifted coisotropic structure $L_N$. By abuse of language, we often call $\ccc$ itself a $1$-shifted coisotropic structure. The algebra of \defn{$\C$-invariant functions} on the base $N = \C^0$ is
\begin{equation}\label{1SCUPN9J7}
C^\infty(N)^\C := \{f\in C^\infty(N) : \sss^*f = \ttt^*f\}.
\end{equation}
By \eqref{1SCYQ3DDG} it is the zeroth Lie algebroid cohomology of $A_\C$ when $\C$ is source-connected. The goal of this section is to equip $C^\infty(N)^\C$ with a Poisson bracket that depends only on  the coisotropic Morita equivalence class of $\ccc$. We first define the bracket in general and check the Jacobi identity by hand when the base map is an inclusion. We then reduce the general case to this one by Morita transport.

\subsection{The bracket}\label{1SCYF566F}
We now construct the bracket on $C^\infty(N)^\C$. Let $L_M$ be the Dirac structure on $M$ induced by $(\G, \omega) \tto (M, \phi)$. In this setting, an admissible vector comes paired with a covector on $M$.

Let
\[
L_N\times_\ccc L_M := \{((u,\alpha),(v,\beta))\in L_N\times L_M :
\ccc_*u = v,\ \alpha = \ccc^*\beta\}.
\]
The nondegeneracy condition says that
\[
A_\C \longrightarrow L_N\times_\ccc L_M,\quad
b \mtoo \bigl((\rho_\C b,\ccc^*\sigma_\omega\ccc_*b),
(\rho_\G\ccc_*b,\sigma_\omega\ccc_*b)\bigr)
\]
is surjective. For $f\in C^\infty(N)^\C$, differentiating $\sss^*f=\ttt^*f$ along $A_\C=\ker\sss_*|_N$ gives $df\circ\rho_\C=0$, so surjectivity forces $df(u) = 0$ for every $((u,\alpha),(v,\beta))\in L_N\times_\ccc L_M$. In particular $df(u) = 0$ for all $(u,\alpha)\in L_N\cap\ccc^*L_M$, so pointwise
\begin{equation}\label{1SC5LWG8H}
(0,df)\in(L_N\cap\ccc^*L_M)^\perp = L_N^\perp+(\ccc^*L_M)^\perp
= L_N+\ccc^*L_M.
\end{equation}
Hence there are $v_f\in TN$ and $\beta_f\in T^*M$ with
\begin{equation}\label{1SCJCKYC8}
(v_f,df+\ccc^*\beta_f)\in L_N
\quad\text{and}\quad
(\ccc_*v_f,\beta_f)\in L_M.
\end{equation}
Such a vector $v_f$ is said to be \defn{admissible} for $f$. It is unique up to a vector $u$ with $(u,\gamma)\in L_N\cap\ccc^*L_M$ for some $\gamma$. We set
\begin{equation}\label{1SC2RZH94}
\{f,g\} := df(v_g).
\end{equation}
The value does not depend on the choice: if $\tilde{v}_g$ is another admissible vector, then $v_g-\tilde{v}_g\in L_N\cap\ccc^*L_M$ and \eqref{1SC5LWG8H} gives $df(v_g) = df(\tilde{v}_g)$. The bracket is bilinear: if $(v_g,\beta_g)$ and $(v_{g'},\beta_{g'})$ satisfy \eqref{1SCJCKYC8} for $g$ and $g'$, then $(v_g+\lambda v_{g'},\beta_g+\lambda\beta_{g'})$ satisfies it for $g+\lambda g'$.

\begin{lemma}\label{1SCLV9K9Y}
The pair $(v_f,\beta_f)$ can be chosen smooth.
\end{lemma}

\begin{proof}
We use the following fact: if $\Phi\colon E\to F$ is a constant-rank bundle map and $\tau$ a smooth section of $F$, then $\Phi^{-1}(\tau(N))$ is a smooth affine bundle as soon as every fibre is nonempty. Indeed, $\tau$ then takes values in the subbundle $\im\Phi$, on which $\Phi$ admits smooth local right inverses, and the resulting local solutions glue by a partition of unity since the fibres are affine. Set $E = TN\oplus\ccc^*T^*M$ and
\[
B = \{(u,\beta)\in E : (u,\ccc^*\beta)\in L_N,\ (\ccc_*u,\beta)\in L_M\},
\]
the image of the embedding $L_N\times_\ccc L_M\hookrightarrow E$, a smooth subbundle. With $F = \bigl((TN\oplus T^*N)/L_N\bigr)\oplus\ccc^*\bigl((TM\oplus T^*M)/L_M\bigr)$,
the map $\Phi\colon E\to F$, $(u,\beta)\mtoo((u,\ccc^*\beta),(\ccc_*u,\beta))$, has kernel $B$ and hence constant rank. Let $\tau$ be the smooth section of $F$ given by $\tau(n):=\bigl([(0,-df_n)],\,0\bigr)$, where $[\,\cdot\,]$ denotes the class modulo $L_N$. For $(v,\beta)\in E_n$, the condition $(v,df_n+\ccc^*\beta)\in L_N$ says that $(v,\ccc^*\beta)\equiv(0,-df_n)$ modulo $L_N$, and $(\ccc_*v,\beta)\in L_M$ says that the second component of $\Phi(v,\beta)$ vanishes; hence
\[
S:=\Phi^{-1}(\tau(N))=\{(v,\beta)\in E : (v,df+\ccc^*\beta)\in L_N,\ (\ccc_*v,\beta)\in L_M\}.
\]
Its fibres are nonempty by \eqref{1SCJCKYC8}, so $S$ is a smooth affine bundle. A smooth section of $S$ gives smooth $(v_f,\beta_f)$ with the required properties.
\end{proof}

\begin{lemma}\label{1SCDG4T4W}
The bracket $\{\cdot,\cdot\}$ is antisymmetric and satisfies the Leibniz rule.
\end{lemma}

\begin{proof}
For antisymmetry, isotropy of $L_N$ applied to $(v_f,df+\ccc^*\beta_f)$ and $(v_g,dg+\ccc^*\beta_g)$ gives
\[
df(v_g)+\beta_f(\ccc_*v_g)+dg(v_f)+\beta_g(\ccc_*v_f) = 0,
\]
while isotropy of $L_M$ applied to $(\ccc_*v_f,\beta_f)$ and $(\ccc_*v_g,\beta_g)$ gives $\beta_f(\ccc_*v_g)+\beta_g(\ccc_*v_f) = 0$. Subtracting yields $df(v_g)+dg(v_f) = 0$.

For the Leibniz rule, closure of $L_N$ under multiplication gives $(gv_h+hv_g,\,d(gh)+g\ccc^*\beta_h+h\ccc^*\beta_g)\in L_N$, so $v_{gh} = gv_h+hv_g+u$ with $(u,\gamma)\in L_N\cap\ccc^*L_M$, and \eqref{1SC5LWG8H} gives $\{f,gh\} = g\{f,h\}+h\{f,g\}$.
\end{proof}

The bracket takes values in $C^\infty(N)^\C$. Let $\gamma\in\C$ be an arrow from $x$ to $y$, and let $(v_g,\beta_g)$ satisfy \eqref{1SCJCKYC8} at $x$. Choose $\tilde v\in T_\gamma\C$ with $\sss_*\tilde v=v_g$ and set $w:=\ttt_*\tilde v$. The compatibility $\ttt^*L_N=\sss^*L_N+\Gamma_{\ccc^*\omega}$, together with the compatibility condition of \cite[Proposition 3.6]{Maxence2025} for $\G$ applied to $\ccc_*\tilde v$, gives $\beta'$ with $(\ccc_*w,\beta')\in L_M$ and $(w,dg+\ccc^*\beta')\in L_N$. So $w$ is admissible for $g$ at $y$, and $\{f,g\}(y)=(\ttt^*df)(\tilde v)=(\sss^*df)(\tilde v)=\{f,g\}(x)$.

\begin{definition}\label{1SCR24WYK}
The algebra $C^\infty(N)^\C$ with the bracket $\{\cdot,\cdot\}$ is the \defn{almost Poisson algebra} of the $1$-shifted coisotropic structure.
\end{definition}

The Jacobi identity requires control over $\ccc\colon N\to M$. We first prove it when $\ccc$ is an inclusion. The general case is reduced to this one in \Cref{1SCEMXYJB}.

\begin{lemma}\label{1SCDYXFD7}
If $\ccc\colon N\to M$ is an inclusion, then $\{f,\{g,h\}\}+\{g,\{h,f\}\}+\{h,\{f,g\}\} = 0$ for all $f,g,h\in C^\infty(N)^\C$.
\end{lemma}

\begin{proof}
Write $a_i = (v_i,df_i+\ccc^*\beta_i)$, and similarly $a_j,a_k$, where $f_i = f$, $f_j = g$, $f_k = h$ and the $v_\ell,\beta_\ell$ are smooth by \Cref{1SCLV9K9Y}. The $\ccc^*\phi$-twisted Courant brackets are
\begin{align*}
\courant{a_i,a_j}_{\ccc^*\phi}
&= \bigl([v_i,v_j],\,\L_{v_i}\ccc^*\beta_j+d\iota_{v_i}df_j
-\iota_{v_j}d\ccc^*\beta_i+\iota_{v_j}\iota_{v_i}\ccc^*\phi\bigr),\\
\courant{a_j,a_k}_{\ccc^*\phi}
&= \bigl([v_j,v_k],\,\L_{v_j}\ccc^*\beta_k+d\iota_{v_j}df_k
-\iota_{v_k}d\ccc^*\beta_j+\iota_{v_k}\iota_{v_j}\ccc^*\phi\bigr),\\
\courant{a_k,a_i}_{\ccc^*\phi}
&= \bigl([v_k,v_i],\,\L_{v_k}\ccc^*\beta_i+d\iota_{v_k}df_i
-\iota_{v_i}d\ccc^*\beta_k+\iota_{v_i}\iota_{v_k}\ccc^*\phi\bigr).
\end{align*}
Isotropy of $L_N$ gives $\langle\courant{a_i,a_j}_{\ccc^*\phi},a_k\rangle = 0$ and its cyclic permutations:
\begin{align*}
0 &= df_k[v_i,v_j]+\ccc^*\phi(v_i,v_j,v_k)
+\bigl(\ccc^*\beta_k[v_i,v_j]+\ccc^*\beta_j[v_k,v_i]
+\ccc^*\beta_i[v_j,v_k]\bigr)\\
&\quad+(\L_{v_i}df_j)(v_k)
+\bigl(v_k\ccc^*\beta_i(v_j)+v_i\ccc^*\beta_j(v_k)
+v_j\ccc^*\beta_k(v_i)\bigr),\\
0 &= df_i[v_j,v_k]+\ccc^*\phi(v_j,v_k,v_i)
+\bigl(\ccc^*\beta_i[v_j,v_k]+\ccc^*\beta_k[v_i,v_j]
+\ccc^*\beta_j[v_k,v_i]\bigr)\\
&\quad+(\L_{v_j}df_k)(v_i)
+\bigl(v_i\ccc^*\beta_j(v_k)+v_j\ccc^*\beta_k(v_i)
+v_k\ccc^*\beta_i(v_j)\bigr),\\
0 &= df_j[v_k,v_i]+\ccc^*\phi(v_k,v_i,v_j)
+\bigl(\ccc^*\beta_j[v_k,v_i]+\ccc^*\beta_i[v_j,v_k]
+\ccc^*\beta_k[v_i,v_j]\bigr)\\
&\quad+(\L_{v_k}df_i)(v_j)
+\bigl(v_j\ccc^*\beta_k(v_i)+v_k\ccc^*\beta_i(v_j)
+v_i\ccc^*\beta_j(v_k)\bigr).
\end{align*}
Summing,
\begin{align*}
0 &= \bigl(df_k[v_i,v_j]+df_i[v_j,v_k]+df_j[v_k,v_i]\bigr)
+3\ccc^*\phi(v_i,v_j,v_k)\\
&\quad+3\bigl(\ccc^*\beta_k[v_i,v_j]+\ccc^*\beta_j[v_k,v_i]
+\ccc^*\beta_i[v_j,v_k]\bigr)\\
&\quad+\bigl((\L_{v_i}df_j)(v_k)+(\L_{v_j}df_k)(v_i)
+(\L_{v_k}df_i)(v_j)\bigr)\\
&\quad+3\bigl(v_k\ccc^*\beta_i(v_j)+v_i\ccc^*\beta_j(v_k)
+v_j\ccc^*\beta_k(v_i)\bigr).
\end{align*}
Using $df_k[v_i,v_j]+df_i[v_j,v_k]+df_j[v_k,v_i]
= 2\bigl(\{f_i,\{f_j,f_k\}\}+\mathrm{c.p.}\bigr)$ and
$(\L_{v_i}df_j)(v_k)+\mathrm{c.p.} = \{f_i,\{f_j,f_k\}\}+\mathrm{c.p.}$, and dividing by $3$,
\begin{align*}
\{f_i,\{f_j,f_k\}\}+\mathrm{c.p.}
&+\ccc^*\phi(v_i,v_j,v_k)
+\bigl(\ccc^*\beta_k[v_i,v_j]+\ccc^*\beta_j[v_k,v_i]
+\ccc^*\beta_i[v_j,v_k]\bigr)\\
&+\bigl(v_k\ccc^*\beta_i(v_j)+v_i\ccc^*\beta_j(v_k)
+v_j\ccc^*\beta_k(v_i)\bigr) = 0.
\end{align*}
It remains to show that
\[\ccc^*\phi(v_i,v_j,v_k)
+\bigl(\ccc^*\beta_k[v_i,v_j]+\ccc^*\beta_j[v_k,v_i]
+\ccc^*\beta_i[v_j,v_k]\bigr)
+\bigl(v_k\ccc^*\beta_i(v_j)+v_i\ccc^*\beta_j(v_k)
+v_j\ccc^*\beta_k(v_i)\bigr) = 0.\]
This follows from the isotropy of $L_M$, transported to $N$. Set $b_\ell = (\ccc_*v_\ell,\beta_\ell)\in\Gamma(L_M|_N)$ and extend each to a smooth section $(\tilde{v}_\ell,\tilde{\beta}_\ell)$ of $L_M$ over a neighbourhood $U$ of $N$ with $\tilde{v}_\ell|_N = \ccc_*v_\ell$ and $\tilde{\beta}_\ell|_N = \beta_\ell$. Isotropy of $L_M$ on $U$ gives
\begin{equation}\label{1SCKZZCES}
\phi(\tilde{v}_i,\tilde{v}_j,\tilde{v}_k)
+\bigl(\tilde{\beta}_k[\tilde{v}_i,\tilde{v}_j]
+\tilde{\beta}_j[\tilde{v}_k,\tilde{v}_i]
+\tilde{\beta}_i[\tilde{v}_j,\tilde{v}_k]\bigr)
+\bigl(\tilde{v}_i\tilde{\beta}_j(\tilde{v}_k)
+\tilde{v}_j\tilde{\beta}_k(\tilde{v}_i)
+\tilde{v}_k\tilde{\beta}_i(\tilde{v}_j)\bigr) = 0.
\end{equation}
By naturality of the Lie bracket \cite[Proposition 8.30]{lee2003smooth}, $[\tilde{v}_i,\tilde{v}_j]|_N = \ccc_*[v_i,v_j]$, so $\tilde{\beta}_k[\tilde{v}_i,\tilde{v}_j]|_N = \beta_k(\ccc_*[v_i,v_j])$, and likewise cyclically. Restricting \eqref{1SCKZZCES} to $N$ and using $\ccc_*v_\ell\,\beta_s(\ccc_*v_r) = v_\ell\,\ccc^*\beta_s(v_r)$ gives the claim.
\end{proof}

\begin{corollary}\label{1SCQQ6Q3H}
If $\ccc$ is an inclusion, $(C^\infty(N)^\C,\{\cdot,\cdot\})$ is a Poisson algebra.
\end{corollary}

\begin{remark}\label{1SCVYWVJN}
Set
\[
L_N-\ccc^*L_M:=\{(v,\alpha):\exists\,\beta,\ (v,\alpha+\ccc^*\beta)\in L_N,\ (\ccc_*v,\beta)\in L_M\}.
\]
By \eqref{1SCJCKYC8}, $v_f$ is admissible for $f$ if and only if $(v_f,df)\in L_N-\ccc^*L_M$. Hence, whenever $L_N-\ccc^*L_M$ is smooth, the bracket \eqref{1SC2RZH94} agrees with the bracket on invariant functions of \textnormal{\cite[Corollary 4.8]{Maxence2025}}.
\end{remark}

\subsection{Morita invariance}\label{1SCD764G8}
We now show that the almost Poisson algebra of invariant functions from \Cref{1SCR24WYK} depends only on the coisotropic Morita equivalence class. This is the algebraic counterpart of \cite[\S10]{Maxence2025}. It will also allow us to extend \Cref{1SCDYXFD7} from inclusions to arbitrary base maps.

\begin{theorem}\label{1SC5V2U32}
Let $\ccc_1\colon(\C_1\tto N_1,L_{N_1})\to(\G_1,\omega_1,\phi_1)$ and $\ccc_2\colon(\C_2\tto N_2,L_{N_2})\to(\G_2,\omega_2,\phi_2)$ be coisotropic Morita equivalent $1$-shifted coisotropic structures. Then there is an isomorphism of almost Poisson algebras
\[
\Psi\colon(C^\infty(N_1)^{\C_1},\{\cdot,\cdot\}_1)\to
(C^\infty(N_2)^{\C_2},\{\cdot,\cdot\}_2).
\]
\end{theorem}

Let $\K\tto P$ be the Lie groupoid of \eqref{1SCT3HPM7}, with Morita morphisms $\psi_1\colon\K\to\C_1$ and $\psi_2\colon\K\to\C_2$. Define $\Psi\colon C^\infty(N_1)^{\C_1}\to C^\infty(N_2)^{\C_2}$ by
\[ \Psi(f)(n_2) := f(\psi_1(z)),\quad z\in\psi_2^{-1}(n_2).
\]
\Crefrange{1SCXJU3SK}{1SCRHRUSQ} below show that $\Psi$ is a well-defined isomorphism of almost Poisson algebras; together they prove \Cref{1SC5V2U32}.

\begin{lemma}\label{1SCXJU3SK}
For each $f\in C^\infty(N_1)^{\C_1}$, the function $\Psi(f)$ is well defined, smooth and $\C_2$-invariant.
\end{lemma}

\begin{proof}
Since $\psi_2$ is a Morita morphism, $\psi_2\colon P\to N_2$ is a surjective submersion and $\K$ is the pullback of $\C_2$ along it \cite[Definition 5.1]{Maxence2025}. Hence, for $z,z'\in P$ and $\eta\in\C_2$ with $\sss_2(\eta)=\psi_2(z)$ and $\ttt_2(\eta)=\psi_2(z')$, there is $k\in\K$ with \begin{equation}\label{1SCKKKRFG}
   \sss(k)=z,\quad \ttt(k)=z',\quad \psi_2(k)=\eta .
   \end{equation}

Let $n_2\in N_2$ and $z,z'\in\psi_2^{-1}(n_2)$, which is nonempty. By \eqref{1SCKKKRFG} with $\eta=1_{n_2}$ there is $k\in\K$ from $z$ to $z'$, and $\gamma:=\psi_1(k)\in\C_1$ satisfies $\sss_1(\gamma)=\psi_1(z)$, $\ttt_1(\gamma)=\psi_1(z')$. Since $\sss_1^* f = \ttt_1^* f$,
\[
f(\psi_1(z')) = f(\ttt_1(\gamma)) = f(\sss_1(\gamma)) = f(\psi_1(z)),
\]
so $\Psi(f)(n_2)$ does not depend on $z$.

For smoothness, let $U \subset N_2$ be an open set over which $\psi_2$ admits a smooth section $\varphi\colon U \to P$. Then $\Psi(f)|_U = f \circ \psi_1 \circ \varphi$ is smooth, and $N_2$ is covered by such open sets.

For $\C_2$-invariance, let $\eta \in \C_2$ with $\sss_2(\eta) = n_2$ and $\ttt_2(\eta) = n_2'$, and choose $z\in\psi_2^{-1}(n_2)$, $z'\in\psi_2^{-1}(n_2')$. By \eqref{1SCKKKRFG} there is $k\in\K$ from $z$ to $z'$, and as above $f(\psi_1(z'))=f(\psi_1(z))$, that is, $\Psi(f)(n_2') = \Psi(f)(n_2)$. So $\sss_2^*\Psi(f) = \ttt_2^*\Psi(f)$.
\end{proof}

\begin{lemma}\label{1SCZJMN36}
The map $\Psi$ is a unital algebra homomorphism.
\end{lemma}

\begin{proof}
Immediate from the definition of $\Psi$ and the pointwise algebra structure on $C^\infty(N_2)^{\C_2}$.
\end{proof}

\begin{lemma}\label{1SC7XJ6UF}
The map $\Psi$ is bijective.
\end{lemma}

\begin{proof}
Define $\Psi'\colon C^\infty(N_2)^{\C_2}\to C^\infty(N_1)^{\C_1}$ by $\Psi'(g)(n_1) := g(\psi_2(z))$ for $z\in\psi_1^{-1}(n_1)$, well defined by  \Cref{1SCXJU3SK} with the roles of $\C_1$ and $\C_2$ exchanged. For $g\in C^\infty(N_2)^{\C_2}$ and $n_2\in N_2$, choose $z\in\psi_2^{-1}(n_2)$. Then
\[
(\Psi\circ\Psi')(g)(n_2) = \Psi'(g)(\psi_1(z)) = g(\psi_2(z)) = g(n_2).
\]
The identity $\Psi'\circ\Psi = \id$ follows by the same computation with $\C_1$ and $\C_2$ exchanged, so $\Psi' = \Psi^{-1}$.
\end{proof}

The bracket is preserved because admissible vectors can be transported through $\K$. This is where the connecting $2$-form and  \Cref{1SCT66PWH} enter.

\begin{lemma}\label{1SC6XHMFW}
Let $n_2\in N_2$, $z\in\psi_2^{-1}(n_2)$, $n_1 := \psi_1(z)$, and $g\in C^\infty(N_1)^{\C_1}$. If $v_g\in T_{n_1}N_1$ satisfies condition \eqref{1SCJCKYC8} for $g$, then for every $\tilde{u}\in T_zP$ with $\psi_{1*}(\tilde{u}) = v_g$, the vector $w := \psi_{2*}(\tilde{u})$ satisfies \eqref{1SCJCKYC8} for $\Psi(g)$.
\end{lemma}

\begin{proof}
Let $\beta_g\in T^*_{\ccc_1(n_1)}M_1$ be given by \eqref{1SCJCKYC8}:
\begin{equation}\label{1SCRGE77K}
(v_g,\,dg+\ccc_1^*\beta_g)\in L_{N_1},\quad
(\ccc_{1*}v_g,\,\beta_g)\in L_{M_1}.
\end{equation}
Differentiating $\Psi(g)\circ\psi_2 = g\circ\psi_1$ gives
\begin{equation}\label{1SCVPX3CR}
\psi_2^*\,d\Psi(g) = \psi_1^*\,dg.
\end{equation}
We construct $\beta'\in T^*_{\ccc_2(n_2)}M_2$ with
\begin{equation}\label{1SCRT3ZWU}
(w,\,d\Psi(g)+\ccc_2^*\beta')\in L_{N_2},\quad
(\ccc_{2*}w,\,\beta')\in L_{M_2}.
\end{equation}
From \eqref{1SCRGE77K}, $(\tilde{u},\psi_1^*(dg+\ccc_1^*\beta_g))\in\psi_1^*L_{N_1}$. By \eqref{1SCU7G6KZ},
$(\tilde{u},\,\psi_1^*(dg+\ccc_1^*\beta_g)+\iota_{\tilde{u}}\delta)\in\psi_2^*L_{N_2}$,
so there is $\alpha\in T^*_{n_2}N_2$ with
\begin{equation}\label{1SCHB5SCS}
(w,\alpha)\in L_{N_2},\quad
\psi_2^*\alpha = \psi_1^*(dg+\ccc_1^*\beta_g)+\iota_{\tilde{u}}\delta.
\end{equation}
Substituting \eqref{1SCVPX3CR} into the second identity of \eqref{1SCHB5SCS},
\begin{equation}\label{1SC2UQLKP}
\psi_2^*(\alpha-d\Psi(g)) = \psi_1^*(\ccc_1^*\beta_g)+\iota_{\tilde{u}}\delta.
\end{equation}
Recall that $\delta = -h^*\gamma+\theta_1^*\omega_1-\theta_2^*\omega_2$, with
\begin{equation}\label{1SCTYPD8Y}
\sss_i\circ\theta_i = \ccc_i\circ\psi_i,\quad
\ttt_i\circ\theta_i = \varphi_i\circ h.
\end{equation}
Since $(\ccc_{1*}v_g,\beta_g)\in L_{M_1}$ and $\sss_{1*}(\theta_{1*}(\tilde{u})) = \ccc_{1*}v_g$, we have $(\theta_{1*}(\tilde{u}),\sss_1^*\beta_g)\in\sss_1^*L_{M_1}$. By the compatibility condition of \cite[Proposition 3.6]{Maxence2025} for $\G_1$ we have
\begin{equation}\label{1SCKQBADU}
\bigl(\theta_{1*}(\tilde{u}),\,
\sss_1^*\beta_g+\iota_{\theta_{1*}(\tilde{u})}\omega_1\bigr)
\in\ttt_1^*L_{M_1},
\end{equation}
so there is $\beta_1$ with $(\ttt_{1*}(\theta_{1*}(\tilde{u})),\beta_1)\in L_{M_1}$ and $\ttt_1^*\beta_1 = \sss_1^*\beta_g+\iota_{\theta_{1*}(\tilde{u})}\omega_1$. Since $\ttt_{1*}(\theta_{1*}(\tilde{u})) = \varphi_{1*}(h_*(\tilde{u}))$,
\begin{equation}\label{1SCGHPXFF}
(h_*(\tilde{u}),\,\varphi_1^*\beta_1)\in\varphi_1^*L_{M_1}.
\end{equation}
\Cref{1SCT66PWH} gives $\varphi_1^*L_{M_1} = \varphi_2^*L_{M_2}+\Gamma_\gamma$, so $(h_*(\tilde{u}),\varphi_1^*\beta_1-\iota_{h_*(\tilde{u})}\gamma)\in\varphi_2^*L_{M_2}$.
Hence there is $\beta_2$ with $(\varphi_{2*}(h_*(\tilde{u})),\beta_2)\in L_{M_2}$ and
\begin{equation}\label{1SCQ7JG7V}
\varphi_2^*\beta_2 = \varphi_1^*\beta_1-\iota_{h_*(\tilde{u})}\gamma.
\end{equation}
By \eqref{1SCTYPD8Y}, $\varphi_{2*}(h_*(\tilde{u})) = \ttt_{2*}(\theta_{2*}(\tilde{u}))$, so $(\theta_{2*}(\tilde{u}),\ttt_2^*\beta_2)\in\ttt_2^*L_{M_2}$.  For $\G_2$, we have  $(\theta_{2*}(\tilde{u}),\ttt_2^*\beta_2-\iota_{\theta_{2*}(\tilde{u})}\omega_2)\in\sss_2^*L_{M_2}$. Since $\sss_{2*}(\theta_{2*}(\tilde{u})) = \ccc_{2*}w$, there is $\beta'$ with $(\ccc_{2*}w,\beta')\in L_{M_2}$ and
\begin{equation}\label{1SCNHBDA2}
\sss_2^*\beta' = \ttt_2^*\beta_2-\iota_{\theta_{2*}(\tilde{u})}\omega_2.
\end{equation}
It remains to verify $\psi_2^*(\ccc_2^*\beta') = \psi_1^*(\ccc_1^*\beta_g)+\iota_{\tilde{u}}\delta$. By \eqref{1SCTYPD8Y}, $\psi_2^*(\ccc_2^*\beta') = \theta_2^*(\sss_2^*\beta')$, and applying \eqref{1SCNHBDA2}, \eqref{1SCQ7JG7V}, \eqref{1SCTYPD8Y} and \eqref{1SCKQBADU} in turn,
\begin{align*}
\theta_2^*(\sss_2^*\beta')
&= \theta_2^*(\ttt_2^*\beta_2)-\iota_{\tilde{u}}(\theta_2^*\omega_2)\\
&= (\varphi_2\circ h)^*\beta_2-\iota_{\tilde{u}}(\theta_2^*\omega_2)\\
&= h^*\!\bigl(\varphi_1^*\beta_1-\iota_{h_*(\tilde{u})}\gamma\bigr)
-\iota_{\tilde{u}}(\theta_2^*\omega_2)\\
&= (\varphi_1\circ h)^*\beta_1-\iota_{\tilde{u}}(h^*\gamma)
-\iota_{\tilde{u}}(\theta_2^*\omega_2)\\
&= \theta_1^*(\ttt_1^*\beta_1)-\iota_{\tilde{u}}(h^*\gamma)
-\iota_{\tilde{u}}(\theta_2^*\omega_2)\\
&= \theta_1^*(\sss_1^*\beta_g)+\iota_{\tilde{u}}(\theta_1^*\omega_1)
-\iota_{\tilde{u}}(h^*\gamma)-\iota_{\tilde{u}}(\theta_2^*\omega_2)\\
&= \psi_1^*(\ccc_1^*\beta_g)+\iota_{\tilde{u}}\delta.
\end{align*}
Since $\psi_2$ is a submersion, \eqref{1SC2UQLKP} gives $\alpha = d\Psi(g)+\ccc_2^*\beta'$. With $(w,\alpha)\in L_{N_2}$ and $(\ccc_{2*}w,\beta')\in L_{M_2}$, this is \eqref{1SCRT3ZWU}.
\end{proof}

\begin{lemma}\label{1SCRHRUSQ}
$\Psi(\{f,g\}_1) = \{\Psi(f),\Psi(g)\}_2$ for all $f,g\in C^\infty(N_1)^{\C_1}$.
\end{lemma}

\begin{proof}
Fix $n_2\in N_2$, pick $z\in\psi_2^{-1}(n_2)$, and set $n_1 := \psi_1(z)$. Let $v_g\in T_{n_1}N_1$ be admissible for $g$ at $n_1$, so $\{f,g\}_1(n_1) = df(v_g)$. Choose $\tilde{u}\in T_zP$ with $\psi_{1*}(\tilde{u}) = v_g$ and set $w := \psi_{2*}(\tilde{u})$. By  \Cref{1SC6XHMFW}, $w$ is admissible for $\Psi(g)$ at $n_2$, so $\{\Psi(f),\Psi(g)\}_2(n_2) = d\Psi(f)(w)$. Differentiating $\Psi(f)\circ\psi_2 = f\circ\psi_1$ gives $\psi_2^*\,d\Psi(f) = \psi_1^*\,df$, hence
\[
\{\Psi(f),\Psi(g)\}_2(n_2)
= d\Psi(f)(w) = \psi_1^*\,df(\tilde{u}) = df(v_g) = \{f,g\}_1(n_1).
\]
By \Cref{1SCXJU3SK}, $\Psi(\{f,g\}_1)(n_2) = \{f,g\}_1(n_1)$.
\end{proof}

\begin{proof}[\textnormal{\textbf{Proof of \Cref{1SC5V2U32}}}]
By \Cref{1SCXJU3SK}, $\Psi(f)$ is well defined and smooth; by \Cref{1SCZJMN36}, $\Psi$ is a unital algebra homomorphism; by \Cref{1SC7XJ6UF}, it is bijective; by  \Cref{1SCRHRUSQ}, it preserves brackets.
\end{proof}

\subsection{Reduction to inclusions and functoriality}\label{1SCEMXYJB}
We now show that every $1$-shifted coisotropic structure is coisotropic Morita equivalent to one whose base map is an inclusion. Combined with  \Cref{1SCDYXFD7,1SC5V2U32}, this gives the Jacobi identity in general.

Let $\G\tto M$ be a Lie groupoid, $N$ a manifold, and $\ppp_M\colon N\times M\to M$ the projection. The \defn{pullback Lie groupoid} is
\[
\ppp_M^!\G := \bigl((N\times M)\times(N\times M)\bigr) \times_{\sss,\ttt,\ppp_M}\G \tto N\times M,
\]
with structure maps inherited from $\G$. The projection $\ppp := (\ppp_2,\ppp_M)\colon\ppp_M^!\G\to\G$ is a morphism, and when $(\G,\omega,\phi)$ is quasi-symplectic, $(\ppp_M^!\G,\ppp^*\omega,\ppp_M^*\phi)$ is Morita equivalent to $(\G,\omega,\phi)$ by \cite[Proposition 4.8]{xu2004momentum}. For a morphism $\ccc\colon\C\tto N\to\G\tto M$, define the lift $\tilde{\ccc}\colon\C\tto N\to\ppp_M^!\G\tto N\times M$ by
\[
\tilde{\ccc}(g) = \bigl((\sss(g),\ccc\sss(g)),(\ttt(g),\ccc\ttt(g)),\ccc(g)\bigr),\quad
\tilde{\ccc}(n) = (n,\ccc(n)).
\]
Its base map $n\mapsto(n,\ccc(n))$ is an inclusion.

\begin{proposition}\label{1SC3P2RJ6}
Let $\ccc\colon(\C\tto N,L)\to(\G\tto M,\omega,\phi)$ be a $1$-shifted coisotropic structure. Then:
\begin{itemize}
\item[(1)] $\tilde{\ccc}\colon(\C\tto N,L)\to
(\ppp_M^!\G\tto N\times M,\ppp^*\omega,\ppp_M^*\phi)$ carries a $1$-shifted coisotropic structure;
\item[(2)] $\ccc$ and $\tilde{\ccc}$ are coisotropic Morita equivalent.
\end{itemize}
\end{proposition}

\begin{proof}
By \cite[Proposition 4.8]{xu2004momentum}, $(\ppp^*\omega,\ppp_M^*\phi)$ is quasi-symplectic, and $\ppp$ is a Morita morphism since $\ppp_M$ is a surjective submersion. The span $\G\xleftarrow{\ppp}\ppp_M^!\G\xrightarrow{\id}\ppp_M^!\G$ with $\gamma=0$ satisfies \eqref{1SCD9DUZ4}, so it is a symplectic Morita equivalence \cite[Lemma 6.3]{Maxence2025}. Take $\K=\C$, $\psi_1=\psi_2=\id$ and $h=\tilde\ccc$ in \eqref{1SCT3HPM7}. Since $\ppp\circ\tilde\ccc=\ccc$, both natural transformations are identities, and the connecting form \eqref{1SCU7G6KZ} vanishes because $1^*\omega=0$ for a multiplicative form. \cite[Theorem 7.2]{Maxence2025} gives (1), with the same $L$, and (2).
\end{proof}

\begin{theorem}\label{1SCCJY5LE}
Every $1$-shifted coisotropic structure $\ccc\colon\C\to\G$ induces a Poisson algebra structure on $C^\infty(N)^\C$, defined by \eqref{1SC2RZH94}.
\end{theorem}

\begin{proof}
By \Cref{1SC3P2RJ6}, $\ccc$ is coisotropic Morita equivalent to $\tilde{\ccc}$, whose base map is an inclusion. The bracket on $C^\infty(N)^{\C}$ defined by $\tilde \ccc$ satisfies the Jacobi identity by \Cref{1SCDYXFD7}, and \Cref{1SC5V2U32} transports it to $C^\infty(N)^\C$.
\end{proof}

\begin{corollary}[Morita invariance]\label{1SC5ZDRV2}
If $\ccc_1\colon\C_1\to\G_1$ and $\ccc_2\colon\C_2\to\G_2$ are coisotropic Morita equivalent, then $C^\infty(N_1)^{\C_1}$ and $C^\infty(N_2)^{\C_2}$ are isomorphic as Poisson algebras.
\end{corollary}

Together with Theorem \ref{1SCCJY5LE}, this proves Theorem \ref{6sipd2id} from the introduction.

The construction is functorial with respect to coisotropic Morita equivalences, composed by homotopy fibre product as in \cite[\S 8.3]{Maxence2025}. On objects, we write
\[
\mathscr{T}(\C,\ccc,L):=(C^\infty(N)^\C,\{\cdot,\cdot\}).
\]
For a Morita equivalence $\K\tto P$ with Morita morphisms $\psi_i\colon\K\to\C_i$, let $\mathscr{T}(\K)$ be the Poisson algebra isomorphism $\Psi(f)(n_2)=f(\psi_1(z))$ of \Cref{1SC5V2U32}. The identity equivalence $\C\xleftarrow{\id}\C\xrightarrow{\id}\C$ gives the identity. For composable equivalences $\K_{12}$ and $\K_{23}$ with composite $\K_{13}:=\K_{12}\times^h_{\C_2}\K_{23}$, write $\Psi_{ij}:=\mathscr{T}(\K_{ij})$. For $f\in C^\infty(N_1)^{\C_1}$ and $n_3\in N_3$, choose $z_{23}\in(\psi_2^{(23)})^{-1}(n_3)$, set $n_2:=\psi_1^{(23)}(z_{23})$, and choose $z_{12}\in(\psi_2^{(12)})^{-1}(n_2)$. Since $(z_{12},1_{n_2},z_{23})$ is an object of $\K_{13}$ lying in $(\psi_2^{(13)})^{-1}(n_3)$,
\[
\Psi_{13}(f)(n_3)
= f\bigl(\psi_1^{(13)}(z_{12},1_{n_2},z_{23})\bigr)
= f\bigl(\psi_1^{(12)}(z_{12})\bigr)
= \Psi_{12}(f)(n_2)
= (\Psi_{23}\circ\Psi_{12})(f)(n_3).
\]
Hence
\[
\mathscr{T}(\K_{23}\circ\K_{12})=\mathscr{T}(\K_{23})\circ\mathscr{T}(\K_{12}).
\]

\begin{remark}
This compatibility can be formalized as a functor between categories by extending the partial category of coisotropic Morita equivalences to a genuine category using the Wehrheim--Woodward trick \cite{weh-woo:10}. We do not pursue these details here. For comparison, see \cite[\S\S 3.2--3.3]{crooks2024moore}, where this construction is carried out for $1$-shifted Lagrangian correspondences.
\end{remark}

\subsection{Examples}\label{uxxbg680}
The Poisson algebra of Theorem \ref{1SCCJY5LE} recovers the following familiar constructions.

\begin{example}\label{39d9asdt}\
\begin{enumerate}
\item \emph{Coisotropic reduction.}
Let $(M, \pi)$ be a Poisson manifold and $N \subset M$ a coisotropic submanifold, with characteristic foliation $\mathcal F=\pi(TN^\circ)$. For the $1$-shifted coisotropic of Example \ref{integrationexamples}\textup{(1)}, the invariant algebra is $C^\infty(N)^{\mathcal F}$, with bracket
\[
\{f,g\}=\{F,G\}_M|_N,
\]
where $F,G$ extend $f,g$ and their Hamiltonian vector fields are tangent to $N$. Indeed, $(v_g,\beta_g)=(-\pi(dG)|_N,-dG|_N)$ is admissible in \eqref{1SCJCKYC8}. This recovers the usual coisotropic reduced bracket.

\item \emph{Hamiltonian reduction.}
A $1$-shifted coisotropic on the trivial quasi-symplectic groupoid $\{*\}\tto\{*\}$ is equivalent to a Lie groupoid $\C\tto N$ with an untwisted Dirac structure $L_N$ satisfying $\sss^*L_N=\ttt^*L_N$ and for which $\rho_\C\colon A_\C\to\ker L_N$ is bijective; see \cite[Proposition 3.11]{Maxence2025}. If the orbit space $N/\C$ is smooth and the quotient map is a submersion, then $L_N$ is the pullback of a Poisson structure on $N/\C$ \cite[Corollary 5.11]{Maxence2025}. More generally, without requiring the quotient to be smooth, \S\ref{1SCYF566F} constructs a Poisson bracket on $C^\infty(N)^\C$. Explicitly, $\{f,g\}=df(v_g)$, where $v_g$ is any vector field satisfying $(v_g,dg)\in L_N$.

In particular, the fibre product of two $1$-shifted coisotropics $\C_1\to\G\leftarrow\C_2$ is a $1$-shifted coisotropic over a point (under suitable regularity hypotheses) and hence induces a Poisson algebra as above; see \cite[\S4]{Maxence2025}. Under these hypotheses, this principle recovers symplectic reduction \cite{MarsdenWeinstein1974,mey:73} and quasi-hamiltonian reduction \cite{ale-mal-mei:98}, as explained in \cite{cal:14,cal:15,saf:16,cal:21,ane-cal:22}, as well as symplectic reduction along a submanifold \cite[\S6]{CrooksMayrand2022} and reduction of strong Dirac maps \cite[\S5]{bualibanu2026reduction}.
\end{enumerate}
\end{example}

\section{Homotopy Poisson structures on 1-shifted coisotropics}\label{1SCYR8NZ4}\label{1SCT2VEBZ}
The goal of this section is to construct the flat $P_\infty$-structure associated with a regular infinitesimal $1$-shifted coisotropic, as stated in Theorem \ref{djsx2lmg}. We begin with a pre-Poisson submanifold $N$ of a Poisson manifold $(M,\pi)$.\footnote{By abuse of notation, we use the same letter $\pi$ for the anchor $\pi^\sharp\colon T^*M\to TM$.} The derived-bracket construction of Cattaneo and Felder does not apply when $N$ is not coisotropic. We instead embed $N$ coisotropically in a Poisson transversal and transport their structure. We then reduce the general case to this construction through local Poisson models, prove independence of the auxiliary choices, and compare the resulting bracket in degree zero with the Poisson algebra of \S\ref{1SC69WYG8}.

\subsection{Preliminaries on \texorpdfstring{$L_\infty$}{L-infinity}-algebras, \texorpdfstring{$P_\infty$}{P-infinity}-algebras and Maurer--Cartan elements}\label{1SCS7HS5G}
We recall the definitions and basic properties used below, fixing our sign conventions and following \cite{lada1992introduction,stasheff2006differential,voronov2004higher,voronov2005higher}.

An \defn{$L_\infty$-algebra} is a graded vector space $\a=\bigoplus_j\a^j$ over a field $\Bbbk$ of characteristic zero, together with multilinear maps $\ell_k\colon\a^{\otimes k}\to\a$ of degree $2-k$ for $k\ge0$ (with $\a^{\otimes0}:=\Bbbk$), which are graded skew-symmetric,
\begin{equation}\label{1SCE8TBTU}
\ell_k(a_{\sigma(1)},\dots,a_{\sigma(k)})=\chi(\sigma;a)\,\ell_k(a_1,\dots,a_k),
\quad \sigma\in\mathfrak S_k,
\end{equation}
and satisfy, for every $k\geq 0$ and all homogeneous $a_1,\dots,a_k\in\a$,
\begin{equation}\label{1SCCM8TQN}
\sum_{i=0}^{k}
\sum_{\sigma\in \mathrm{Sh}(i,k-i)}
(-1)^{i(k-i)}\,
\chi(\sigma;a)\;
\ell_{k-i+1}\Bigl(
\ell_i(a_{\sigma(1)},\dots,a_{\sigma(i)}),
a_{\sigma(i+1)},\dots,a_{\sigma(k)}
\Bigr)=0 .
\end{equation}
Here $\mathrm{Sh}(i,k-i)$ is the set of permutations $\sigma$ of $\{1,\dots,k\}$ with
$\sigma(1)<\dots<\sigma(i)$ and $\sigma(i+1)<\dots<\sigma(k)$, and $\chi(\sigma;a)=\operatorname{sgn}(\sigma)\,\varepsilon(\sigma;a)$ combines the signature of $\sigma$ with the Koszul sign $\varepsilon(\sigma;a)=(-1)^{\sum_{p<q,\ \sigma(p)>\sigma(q)}|a_{\sigma(p)}|\,|a_{\sigma(q)}|}$.

The element $\ell_0\in\a^2$ is the \defn{curvature}. For $k=0$, \eqref{1SCCM8TQN} reads $\ell_1(\ell_0)=0$; for $k=1$ it reads $\ell_1^2=-\ell_2(\ell_0,\cdot)$. We call $\a$ \defn{flat} if $\ell_0=0$; in that case $\ell_1$ is a differential.

\medskip
Our main source of examples is Voronov's higher derived bracket construction \cite{voronov2004higher,voronov2005higher}. A \defn{curved $V$-datum} is a quadruple $(\g,\a,P,\Pi)$: a graded Lie algebra $\g$, an abelian graded subalgebra $\a\subset\g$, a projection $P\colon\g\to\a$ whose kernel $\h:=\ker P$ is a subalgebra, and an element $\Pi\in\g^1$ with $[\Pi,\Pi]=0$. The \defn{derived brackets}
\[
\{a_1,\dots,a_k\}:=P\bigl[\cdots[[\Pi,a_1],a_2],\dots,a_k\bigr]\ (k\ge1),
\quad \{\varnothing\}:=P(\Pi),
\]
are graded symmetric of degree $1$ with respect to the grading of $\g$. They satisfy the curved $L_\infty[1]$-relations: for every $k\ge0$,
\[
\sum_{i+j=k}\ \sum_{\sigma\in\mathrm{Sh}(i,j)}\varepsilon(\sigma;a)\,
\{\{a_{\sigma(1)},\dots,a_{\sigma(i)}\},a_{\sigma(i+1)},\dots,a_{\sigma(k)}\}=0,
\]
where $\varepsilon(\sigma;a)$ is the Koszul sign for the grading $\deg$. The proof uses $[\Pi,\Pi]=0$, $[\a,\a]=0$, the Jacobi identity in $\g$, and the fact that $\h$ is a subalgebra. To pass to the conventions above, regrade $\a$ by $|a|:=\deg(a)+1$ and set
\begin{equation}\label{EQDECALAGE}
\ell_k(a_1,\dots,a_k):=\operatorname{dec}_k(a)\,\{a_1,\dots,a_k\},\qquad
\operatorname{dec}_k(a):=(-1)^{\sum_{i=1}^{k}(k-i)|a_i|},\qquad
\ell_0:=P(\Pi).
\end{equation}
The sign $\operatorname{dec}_k(a)$ is the Koszul sign of $(s^{-1})^{\otimes k}$ for the desuspension $s^{-1}$ of degree $-1$, and it intertwines the Koszul-symmetric relations above with \eqref{1SCE8TBTU}--\eqref{1SCCM8TQN}: $\operatorname{dec}_k(a_\sigma)\,\chi(\sigma;a)=\operatorname{dec}_k(a)\,\varepsilon(\sigma;a)$ for every $\sigma\in\mathfrak S_k$.
Then $\ell_k$ has degree $2-k$ and satisfies \eqref{1SCE8TBTU} and \eqref{1SCCM8TQN}, and $(\a,\{\ell_k\})$ is flat exactly when $P(\Pi)=0$, for instance when $\Pi\in\h$ \cite[\S 2.6]{Cattaneo2007}.\footnote{The sign $(-1)^{\sum_i(i-1)\deg(a_i)}$ displayed in \cite[\S 2.6]{Cattaneo2007} does not satisfy \eqref{1SCCM8TQN} for $k=2$. This does not affect \cite{Cattaneo2007}: the structure of Proposition 2.1 there is obtained through Proposition 4.4, and the d\'ecalage of Appendix A.6 is the Koszul sign of $s^{\otimes k}$, which differs from $\operatorname{dec}_k$ by $(-1)^{k(k-1)/2}$.}

\medskip
An \defn{$L_\infty$-morphism} $\varrho\colon(\a,\ell_k)\lto(\a',\ell'_k)$
between flat $L_\infty$-algebras is a collection of graded skew-symmetric maps $\varrho_k\colon\a^{\otimes k}\to\a'$ of degree $1-k$, $k\ge1$, satisfying, for every $k\geq 1$,
\begin{equation}\label{1SCFV8US5}
\begin{split}
&\sum_{\substack{i+j=k+1\\ i,j\geq 1}}
\ \sum_{\sigma\in\mathrm{Sh}(i,k-i)}
(-1)^{i(j-1)}\,
\chi(\sigma;a)\,
    \varrho_j\Bigl(
\ell_i(a_{\sigma(1)},\dots,a_{\sigma(i)}),
a_{\sigma(i+1)},\dots,a_{\sigma(k)}
\Bigr)
\\
&=
\sum_{r=1}^k \frac{1}{r!}
\sum_{\substack{k_1+\cdots+k_r=k\\ k_s\geq 1}}
\ \sum_{\sigma\in\mathrm{Sh}(k_1,\dots,k_r)}
\chi'(\sigma;a)\,
\ell'_r\Bigl(
\varrho_{k_1}(a_{\sigma(1)},\dots,a_{\sigma(k_1)}),
\dots,
\varrho_{k_r}(a_{\sigma(k-k_r+1)},\dots,a_{\sigma(k)})
\Bigr),
\end{split}
\end{equation}
where, for $K_s:=k_1+\dots+k_s$,
\[
\chi'(\sigma;a):=\chi(\sigma;a)\,
(-1)^{\sum_{s=1}^{r}(r-s)(k_s-1)}\,
(-1)^{\sum_{s=1}^{r}(k_s-1)\bigl(|a_{\sigma(1)}|+\dots+|a_{\sigma(K_{s-1})}|\bigr)} .
\]
The last factor is the Koszul sign of moving $\varrho_{k_s}$, of degree $1-k_s$, past the arguments preceding its block; the middle factor depends only on the arities. For $k=1$ this gives $\varrho_1\ell_1=\ell'_1\varrho_1$. We call $\varrho$ a \defn{quasi-isomorphism} if $\varrho_1$ induces an isomorphism on cohomology, and an \defn{isomorphism} if $\varrho_1$ is bijective. It is \defn{strict} if $\varrho_k=0$ for every $k\ge2$; since $\deg \varrho_1=0$, \eqref{1SCFV8US5} then reduces to $\varrho_1(\ell_k(a_1,\dots,a_k))=\ell'_k(\varrho_1(a_1),\dots,\varrho_1(a_k))$ for $k\ge1$. Between differential graded Lie algebras, a strict $L_\infty$-morphism is a morphism of DGLAs.

We write $L_\infty$-morphisms as $\varrho\colon(\a,\ell)\lto(\a',\ell')$, and identify a strict $L_\infty$-morphism with its linear component $\varrho_1$. Decorated letters $\varrho,\vartheta$ denote $L_\infty$- or $A_\infty$-morphisms, and $\U$ denotes the formality morphisms of \Cref{1SCC6G85S}.

\medskip
A \defn{$P_\infty$-algebra} \cite{Cattaneo2007} (or \emph{homotopy Poisson algebra} \cite[Definition 2.1]{Mehta2011}) is a graded commutative associative $\Bbbk$-algebra $\a$ with an $L_\infty$-structure $\{\ell_k\}_{k\ge0}$ whose brackets are multiderivations: for every $k\ge1$ and all homogeneous $a_1,\dots,a_{k-1}\in\a$, the map $a\mapsto\ell_k(a_1,\dots,a_{k-1},a)$ is a graded derivation of degree $2-k+\sum_{s=1}^{k-1}|a_s|$. For $k=1$ this says that $\ell_1$ is a derivation of degree $1$. If $\a$ is flat, $\ell_1$ is a differential, the case $k=2$ of \eqref{1SCCM8TQN} says that $\ell_1$ is a derivation of the bracket $\ell_2$, and the case $k=3$ says that $\ell_2$ satisfies the graded Jacobi identity up to the $\ell_1$-homotopy $\ell_3$. Hence the product and $\ell_2$ descend to $\ell_1$-cohomology, which is therefore a graded Poisson algebra.

We often transport $P_\infty$-structures along algebra isomorphisms.

\begin{lemma}\label{1SCU92AQH}
Let $\Phi\colon\a\to\b$ be an isomorphism of degree $0$ between graded commutative associative algebras, and let $\{\ell_k^{\b}\}_{k\ge0}$ be a $P_\infty$-structure on $\b$. Define $\ell_k^{\a}:=\Phi^{-1}\circ \ell_k^{\b}\circ \Phi^{\otimes k}$ for $k\ge1$, and $\ell_0^{\a}:=\Phi^{-1}(\ell_0^{\b})$. Then:
\begin{itemize}
\item[(1)] $\{\ell_k^{\a}\}_{k\ge0}$ is a $P_\infty$-structure on $\a$, flat if and only if $\{\ell_k^{\b}\}_{k\ge0}$ is, and $\Phi$ is a strict isomorphism of $L_\infty$-algebras.
\item[(2)] If moreover $d_{\a}\colon\a\to\a$ is a linear map of degree $1$ with $\Phi\circ d_{\a}=\ell_1^{\b}\circ \Phi$, then $\ell_1^{\a}=d_{\a}$.
\item[(3)] Let $\Phi'\colon\a'\to\b'$ be another such isomorphism, $\{\ell_k^{\b'}\}$ a $P_\infty$-structure on $\b'$ with transport $\{\ell_k^{\a'}\}$, and $\varrho\colon(\b,\ell^{\b})\lto(\b',\ell^{\b'})$ an $L_\infty$-morphism. Then $\Phi'^{-1}\circ\varrho_k\circ\Phi^{\otimes k}$, $k\ge1$, is an $L_\infty$-morphism $(\a,\ell^{\a})\lto(\a',\ell^{\a'})$, which is an isomorphism if and only if $\varrho$ is
\end{itemize}
\end{lemma}

\begin{proof}
Since $\Phi$ is an algebra isomorphism of degree $0$, it preserves Koszul signs and multiderivation properties. Conjugation by $\Phi$ therefore carries the relations \eqref{1SCCM8TQN} for $\{\ell_k^{\b}\}$ termwise into the same relations for $\{\ell_k^{\a}\}$. Part (2) is immediate from the definition of $\ell_1^\a$. For (3), the signs $\chi$ and $\chi'$ in \eqref{1SCFV8US5} depend only on degrees, which $\Phi$ and $\Phi'$ preserve, so conjugation carries \eqref{1SCFV8US5} for $\varrho$ termwise into the same relation for the transported maps.
\end{proof}

\begin{remark}\label{1SCRESCALE}
Let $t\in\Bbbk^\times$. If $\{\ell_k\}_{k\ge0}$ is a $P_\infty$-structure on $\a$, then so is $\{t^{k-1}\ell_k\}_{k\ge0}$: in \eqref{1SCCM8TQN}, the term $\ell_{k-i+1}\circ\ell_i$ acquires the factor $t^{(k-i)+(i-1)}=t^{k-1}$, independent of $i$, and multiderivations and flatness are preserved. Likewise, if $\varrho$ is an $L_\infty$-morphism, then $\{t^{k-1}\varrho_k\}_{k\ge1}$ is an $L_\infty$-morphism between the rescaled structures, since both sides of \eqref{1SCFV8US5} acquire the factor $t^{k-1}$; it is an isomorphism, respectively strict, if and only if $\varrho$ is. The rescaling fixes $\ell_1$ and multiplies $\ell_2$ by $t$.

For $t=-1$ it has the following description. The parity automorphism $D(a):=(-1)^{|a|}a$ of $\a$ satisfies $D^{-1}\circ\ell_k\circ D^{\otimes k}=(-1)^k\ell_k$, because $\ell_k$ has degree $2-k$. If $\ell_k=\ell_k^{\Pi}$ is obtained from a curved $V$-datum $(\g,\a,P,\Pi)$ by \eqref{EQDECALAGE}, then $\ell_k^{-\Pi}=-\ell_k^{\Pi}$, hence
\[
(-1)^{k-1}\ell_k^{\Pi}=D^{-1}\circ\ell_k^{-\Pi}\circ D^{\otimes k}\quad(k\ge0).
\]
Thus the rescaled structure is the derived-bracket structure of $-\Pi$, conjugated by $D$.
\end{remark}

\medskip

Let $\a$ be a differential graded Lie algebra over $\Bbbk$. A \defn{filtration} on $\a$ is a decreasing sequence of graded subspaces $\a=\FFF^{k_0}\a\supset \FFF^{k_0+1}\a\supset\cdots$, indexed by integers $k\ge k_0$ with $k_0\le 1$ and the convention $\FFF^k\a=\a$ for $k\le k_0$, such that $d(\FFF^k\a)\subset \FFF^k\a$ and $[\FFF^k\a,\FFF^l\a]\subset \FFF^{k+l}\a$. It is \defn{complete} if the canonical map $\a\to\varprojlim_k\a/\FFF^k\a$ is an isomorphism.

If $\a$ is complete and $X\in \FFF^1\a^0$, then $\ad_X(\FFF^k\a)\subset \FFF^{k+1}\a$, so $\exp(\ad_X)$ and $\frac{\exp(\ad_X)-1}{\ad_X}:=\sum_{k\ge0}\frac{1}{(k+1)!}\ad_X^{\,k}$ converge on $\a$. The \defn{Maurer--Cartan equation} is
\begin{equation}\label{1SCVMMUSM}
d\varpi+\tfrac12[\varpi,\varpi]=0,\quad
\varpi\in \FFF^1\a^1 ,
\end{equation}
and $\MC(\a)$ denotes its solution set in $\FFF^1\a^1$. For $X\in \FFF^1\a^0$ the \defn{gauge action} is
\begin{equation}\label{1SCFR6FY2}
X\cdot\varpi:=\exp(\ad_X)\varpi-\frac{\exp(\ad_X)-1}{\ad_X}(dX).
\end{equation}
It preserves $\MC(\a)$, and the Baker--Campbell--Hausdorff product $(X,Y)\mapsto X\bullet Y$ makes $\FFF^1\a^0$ a group acting on $\MC(\a)$. Two Maurer--Cartan elements in the same orbit are \defn{gauge equivalent}; see \cite{getzler2009,berglund2023exponential,yalin2016maurer}.

\begin{remark}\label{1SCV8CXM4}
If $d=0$, then $dX=0$, the second term of \eqref{1SCFR6FY2} vanishes, and
\begin{equation}\label{1SCDGGZKR}
X\cdot\varpi=e^{\ad_X}\varpi .
\end{equation}
This applies to the graded Lie algebra of multivector fields used in \Cref{1SCC6G85S}.
\end{remark}

An $L_\infty$-morphism $\varrho=(\varrho_1,\varrho_2,\dots)\colon\g\lto\g'$ between complete filtered differential graded Lie algebras is \defn{filtered} if
\begin{equation}\label{1SCLC9EC8}
\varrho_n\bigl(\FFF^{k_1}\g,\dots,\FFF^{k_n}\g\bigr)\subset\FFF^{k_1+\cdots+k_n}\g'\quad(n\ge1).
\end{equation}
For $\alpha\in\FFF^1\g^1$ we set
\begin{equation}\label{1SCS3HGKV}
\varrho_*(\alpha):=\sum_{n\ge1}\frac1{n!}\,\varrho_n(\alpha,\dots,\alpha).
\end{equation}
The $n$-th term lies in $\FFF^n\g'$, so the series converges.

Following \cite[Definition 3.3]{kraftschnitzer2021}, two filtered $L_\infty$-morphisms $\varrho,\varrho'\colon\g\lto\g'$ are \defn{homotopic} if they are homotopy equivalent Maurer--Cartan elements of the convolution $L_\infty$-algebra, filtered as in \cite[Eq.(22)]{kraftschnitzer2021}. We use the next lemma to transport gauge equivalences along formality morphisms in \Cref{1SCQPGX26}.

\begin{lemma}\label{1SCMF2KEY}
Let $\varrho\colon\g\lto\g'$ be a filtered $L_\infty$-morphism between complete filtered differential graded Lie algebras. Then:
\begin{itemize}
\item[(1)] $\varrho_*$ maps $\MC(\g)$ to $\MC(\g')$;
\item[(2)] $\varrho_*$ maps gauge equivalent elements to gauge equivalent elements;
\item[(3)] if $\varrho$ and $\varrho'$ are homotopic filtered $L_\infty$-morphisms, then $\varrho_*$ and $\varrho'_*$ induce the same map on gauge equivalence classes.
\end{itemize}
\end{lemma}

\begin{proof}
(1) is \cite[Lemma 2.7(iii)]{kraftschnitzer2021}, which goes back to Dolgushev. (2) is \cite[Proposition 2.15]{kraftschnitzer2021}, stated there for the homotopy relation; for differential graded Lie algebras with complete filtrations the homotopy relation and the gauge relation coincide \cite[Theorem 2.12]{kraftschnitzer2021}. (3) is \cite[Proposition 3.5(iii)]{kraftschnitzer2021}, whose proof is given in \cite{bursztyndolgushevwaldmann2012}.
\end{proof}

\subsection{The Cattaneo--Felder construction}\label{1SC9E2V5Z}
Let $N$ be a submanifold of a Poisson manifold $(M,\pi)$. Let $\g$ be the graded Lie algebra of multivector fields on the formal neighbourhood of $N$,
\[
\g=\varprojlim_n\ C^\infty(\wedge^\bullet TM)\big/
I_N^{\,n}\,C^\infty(\wedge^\bullet TM),
\]
where $I_N$ is the ideal of functions vanishing on $N$. In coordinates $(x_1,\dots,x_m,y_1,\dots,y_n)$ in which $N$ is the locus $y_1=\cdots=y_n=0$, the algebra $\g$ is generated by
$C^\infty(x_1,\dots,x_m)\llbracket y_1,\dots,y_n\rrbracket$, by the
$\frac{\partial}{\partial x_i}$ and by the $\frac{\partial}{\partial y_i}$. The class of $\pi$ is a Maurer--Cartan element of $\g$.

Suppose first that $N$ is coisotropic and let $\nu_N:=TM|_N/TN$ be its normal bundle. Then $\nu_N^*$ is a Lie algebroid over $N$, and
\[
\b_N=C^\infty(\wedge^\bullet\nu_N)
\]
is an abelian graded Lie algebra with a surjection $P\colon\g\to\b_N$, restriction to $N$ followed by the projection $TM|_N\to\nu_N$. A tubular neighbourhood of $N$ in $M$ gives an inclusion $j\colon\b_N\hookrightarrow\g$ for which $P$ is a projection onto $\b_N$. Explicitly, a tubular neighbourhood identifies a neighbourhood of $N$ in $M$ with a vector bundle $\ppp_E\colon E\to N$. Since $TE|_N=TN\oplus\ker d\ppp_E$, we may identify $\nu_N$ with $\ker d\ppp_E|_N\subset TM|_N$, and the inclusion $\nu_N\hookrightarrow TM|_N$ extends to $\b_N\hookrightarrow\g$. In coordinates, the image of $\b_N$ is generated by $C^\infty(x_1,\dots,x_m)$ and the $\frac{\partial}{\partial y_i}$, an abelian Lie subalgebra of $\g$. One checks that
\begin{equation}\label{1SCXSWR3W}
P[a,b]=P[Pa,b]+P[a,Pb]
\end{equation}
for all $a,b\in\g$. Since $\b_N$ is abelian and $P$ is a projection, \eqref{1SCXSWR3W} says that $\ker P$ is closed under the Schouten bracket, so $(\g,\b_N,P,\pi)$ is a curved $V$-datum in the sense of \Cref{1SCS7HS5G}. Voronov's construction then makes $\b_N$ a $P_\infty$-algebra, flat because $P(\pi)=0$ by coisotropy. Up to the sign conventions of \Cref{1SCS7HS5G}, this is the Cattaneo--Felder $P_\infty$-structure \cite[Proposition 2.1]{Cattaneo2007}.

\subsection{Reduction to the coisotropic case}\label{1SCS4XYSN}
Recall that $N$ is \defn{pre-Poisson} if $TN+\pi(TN^\circ)$ has constant rank \cite{CattaneoZambon2009}. Let $N$ be pre-Poisson. Then $L:=\pi^{-1}(TN)\cap TN^\circ$ is a Lie subalgebroid of $T^*M$ for the Lie algebroid structure induced by $\pi$, and the natural candidate for the $P_\infty$-algebra is $\a:=C^\infty(\wedge^\bullet L^*)$. Since $\pi^{-1}(TN)^\circ=\pi(TN^\circ)$, the sequence
\begin{equation}\label{1SC2QNAK4}
0\to L\to T^*M|_N\to\bigl(TN+\pi(TN^\circ)\bigr)^*\to0
\end{equation}
is exact; the bundle $TN+\pi(TN^\circ)$ is smooth because $N$ is pre-Poisson.

The two ingredients of \Cref{1SC9E2V5Z} still exist in this setting. Dualizing \eqref{1SC2QNAK4} gives a surjection $TM|_N\to L^*$ with kernel $TN+\pi(TN^\circ)$, which extends to a projection $P\colon\g\to\a$, and a tubular neighbourhood of $N$ in $M$ gives an inclusion $j\colon\a\hookrightarrow\g$ as an abelian subalgebra with $P\circ j=\id$. However, $P$ satisfies \eqref{1SCXSWR3W}, equivalently $\ker P$ is a Lie subalgebra of $\g$, if and only if $N$ is coisotropic. When $N$ is not coisotropic, $(\g,\a,P,\pi)$ is therefore not a curved $V$-datum, and the $P_\infty$-structure on $\a$ cannot be obtained by derived brackets inside $M$. We proceed differently and change the ambient manifold instead of the algebra: we replace $M$ by a smaller Poisson manifold $S\supset N$ in which $N$ is coisotropic, and import the flat structure that \Cref{1SC9E2V5Z} provides there.

Cattaneo and Zambon \cite[Theorem 3.3]{CattaneoZambon2009} proved that every pre-Poisson submanifold $N\subset(M,\pi)$ sits coisotropically inside a \defn{Poisson transversal} (cosymplectic) submanifold $S\subset M$, unique up to neighbourhood equivalence: given two such $S$ and $S'$, after shrinking both to tubular neighbourhoods of $N$, there is a Poisson diffeomorphism $\varphi\colon S\to S'$ restricting to the identity on $N$ \cite[Theorem 4.3]{CattaneoZambon2009}.

Set $\nu_S:=TS|_N/TN$ and $\b_S:=C^\infty(\wedge^\bullet\nu_S)$. By
\cite[Proposition 3.6]{CattaneoZambon2009}, restriction of covectors
\begin{equation}\label{1SCPEDZZK}
r_S\colon L\to\nu_S^*,\quad r_S(\xi)=\xi|_{TS|_N},
\end{equation}
is an isomorphism of Lie algebroids onto the conormal bundle $\nu_S^*$ of $N$ in $S$, inverse to the extension isomorphism constructed there. Dualizing gives an isomorphism of graded algebras
\begin{equation}\label{1SCVX2YHQ}
\Phi_S:=\wedge^\bullet(r_S^*)^{-1}\colon
\a=C^\infty(\wedge^\bullet L^*)\to\b_S .
\end{equation}

Since $S$ is a Poisson transversal, it carries a canonical Poisson structure in which $N$ is coisotropic. The construction of  \Cref{1SC9E2V5Z} with $S$ in place of $M$ yields the graded Lie algebra
\[ \g_S=\varprojlim_n\ C^\infty(\wedge^\bullet TS)\big/
I_N^{\,n}\,C^\infty(\wedge^\bullet TS),
\]
with Maurer--Cartan element $\pi_S\in\g_S$ the class of the Poisson bivector of $S$. A tubular neighbourhood $\psi$ of $N$ in $S$ determines an inclusion $j_S\colon\b_S\hookrightarrow\g_S$ and a projection
$P_S\colon\g_S\to\b_S$,
\[
P_S(X):=(\wedge^\bullet\ppp_S)(X|_N),
\]
where $\ppp_S\colon TS|_N\to\nu_S$ is the projection. The Cattaneo--Felder construction applied to $N\subset S$ makes $\b_S$ a flat $P_\infty$-algebra with $\ell_1^{\b_S}$ the Chevalley--Eilenberg differential of $\nu_S^*$. Its brackets are the derived brackets in the conventions of \Cref{1SCS7HS5G}, rescaled as in \Cref{1SCRESCALE} with $t=-1$: for homogeneous $b_1,\dots,b_k\in\b_S$, with $|b_i|$ the exterior degree, set $\varepsilon(b):=(-1)^{k-1}\operatorname{dec}_k(b)=(-1)^{k-1+\sum_{i=1}^k(k-i)|b_i|}$, with $\operatorname{dec}_k$ as in \eqref{EQDECALAGE}; then
\begin{equation}\label{1SCRZZPXA}
\ell_k^{\b_S}(b_1,\dots,b_k)
=\varepsilon(b)\,P_S\bigl[\cdots[\pi_S,j_S(b_1)],\dots,j_S(b_k)\bigr],
\quad \ell_0^{\b_S}=-P_S(\pi_S)=0.
\end{equation}
The vanishing of $\ell_0^{\b_S}$ is the coisotropy of $N$ in $S$. The factor $(-1)^{k-1}$ leaves $\ell_1^{\b_S}$ unchanged and reverses the sign of $\ell_2^{\b_S}$; it is chosen so that $\ell_2$ induces the reduced Poisson bracket, see \eqref{1SCQHGY3K}. It remains to transport this structure to $\a$ along $\Phi_S$.

\begin{proposition}\label{1SCMMB63H}
Let $S\subset M$ be a Poisson transversal with $N$ coisotropic in $S$, and let $\psi$ be a tubular neighbourhood of $N$ in $S$. Then $\a$ carries a flat $P_\infty$-structure $\{\ell_k^{\a,S,\psi}\}_{k\ge0}$ with
$\ell_1^{\a,S,\psi}=d_L$.
\end{proposition}

\begin{proof}
By \Cref{1SCU92AQH} applied to $\Phi_S$, the maps
$\ell_k^{\a,S,\psi}:=\Phi_S^{-1}\circ\ell_k^{\b_S}\circ\Phi_S^{\otimes k}$ form a flat $P_\infty$-structure on $\a$. Since $\ell_1^{\b_S}$ is the algebroid differential of $\nu_S^*$ and $r_S$ is an isomorphism of Lie algebroids, $\Phi_S\circ d_L=\ell_1^{\b_S}\circ\Phi_S$, so \Cref{1SCU92AQH}(2) gives $\ell_1^{\a,S,\psi}=d_L$.
\end{proof}

\subsection{Independence of the choices}\label{1SC4DJUA9}
The structure of \Cref{1SCMMB63H} depends on the Poisson transversal $S$ and on the tubular neighbourhood $\psi$. We show that its $L_\infty$-isomorphism class does not. There are two comparisons. Two transversals are compared through the Poisson diffeomorphism of \cite[Theorem 4.3]{CattaneoZambon2009} and a tubular neighbourhood of $N$ in $S'$ adapted to $\psi$; this comparison is strict (\Cref{1SCFBTZ34}). Two tubular neighbourhoods inside a fixed $S$ are compared through \cite[Theorem 3.2]{cattaneo2008equivalences}; this comparison need not be strict (\Cref{1SC229L5W}). \Cref{1SC4NFQ6C} composes the two.

The first comparison is a special case of a naturality statement that we will reuse in \Cref{1SCSWUFL4}, so we state it in that generality.

\begin{lemma}[Naturality]\label{1SCEXNP25}
Let $\Theta\colon(M,\pi)\to(M',\pi')$ be a Poisson diffeomorphism between neighbourhoods of pre-Poisson submanifolds $N\subset M$ and $N'\subset M'$, with $\Theta(N)=N'$. Let $S\subset M$ be a Poisson transversal containing $N$ coisotropically, $\psi$ a tubular neighbourhood of $N$ in $S$, and set $S':=\Theta(S)$, $\Lambda:=\nu(\Theta|_S)\colon\nu_S\to\nu_{S'}$ and $\psi':=\Theta\circ\psi\circ\Lambda^{-1}$. Let $\Theta_L\colon L\to L'$ be the restriction of the cotangent lift $(d\Theta^{-1})^*$ and $\widehat\Theta:=\wedge^\bullet(\Theta_L^*)^{-1}\colon\a\to\a'$. Then $S'$ is a Poisson transversal containing $N'$ coisotropically, $\psi'$ is a tubular neighbourhood of $N'$ in $S'$, and $\widehat\Theta$ is a strict isomorphism of $L_\infty$-algebras from $(\a,\{\ell_k^{\a,S,\psi}\})$ to $(\a',\{\ell_k^{\a',S',\psi'}\})$.
\end{lemma}

\begin{proof}
Since $\Theta$ is Poisson and $\Theta(N)=N'$, the image $S'$ is a Poisson transversal in which $N'$ is coisotropic, and $\varphi:=\Theta|_S\colon S\to S'$ is a Poisson diffeomorphism with $\varphi(N)=N'$. Set
\[
\theta:=\wedge^\bullet\Lambda\colon\b_S\to\b_{S'},
\quad
\Upsilon:=\wedge^\bullet\varphi_*\colon\g_S\to\g_{S'} .
\]
Pushforward of multivector fields commutes with the Schouten bracket, and $\varphi(N)=N'$, so $\Upsilon$ is a morphism of graded Lie algebras with $\Upsilon(\pi_S)=\pi_{S'}$. From the definitions of $P_S$ and $j_S$, and $\psi'\circ\Lambda=\varphi\circ\psi$, we get $P_{S'}\circ\Upsilon=\theta\circ P_S$ and $\Upsilon\circ j_S=j_{S'}\circ\theta$. Since $\theta$ preserves degrees, $\varepsilon(\theta b)=\varepsilon(b)$, and \eqref{1SCRZZPXA} gives, for every $k\ge1$,
\begin{align*}
\theta\bigl(\ell_k^{\b_S}(b_1,\dots,b_k)\bigr)
&=\varepsilon(b)\,P_{S'}\Bigl(\Upsilon\bigl[\cdots[[\pi_S,j_S(b_1)],j_S(b_2)],\dots,
j_S(b_k)\bigr]\Bigr)\\
&=\varepsilon(b)\,P_{S'}\bigl[\cdots[[\pi_{S'},j_{S'}(\theta b_1)],j_{S'}(\theta b_2)],\dots,
j_{S'}(\theta b_k)\bigr]\\
&=\ell_k^{\b_{S'}}(\theta b_1,\dots,\theta b_k),
\end{align*}
while $\ell_0^{\b_S}=0=\ell_0^{\b_{S'}}$. As $\theta$ is an isomorphism of graded algebras of degree $0$, it is a strict $L_\infty$-isomorphism.

Restriction of covectors commutes with $d\Theta$, so $r_{S'}\circ\Theta_L=(\Lambda^{-1})^*\circ r_S$, hence $\theta\circ\Phi_S=\Phi_{S'}\circ\widehat\Theta$. Combining this with $\Phi_S\circ\ell_k^{\a,S,\psi}=\ell_k^{\b_S}\circ\Phi_S^{\otimes k}$ and the strictness of $\theta$,
\[
\widehat\Theta\circ\ell_k^{\a,S,\psi}
=\Phi_{S'}^{-1}\theta\,\ell_k^{\b_S}\Phi_S^{\otimes k}
=\Phi_{S'}^{-1}\ell_k^{\b_{S'}}\theta^{\otimes k}\Phi_S^{\otimes k}
=\Phi_{S'}^{-1}\ell_k^{\b_{S'}}\Phi_{S'}^{\otimes k}\widehat\Theta^{\otimes k}
=\ell_k^{\a',S',\psi'}\circ\widehat\Theta^{\otimes k}
\]
for every $k\ge1$, and $\widehat\Theta(\ell_0^{\a,S,\psi})=0=\ell_0^{\a',S',\psi'}$ by flatness. Being a composite of algebra isomorphisms, $\widehat\Theta$ is bijective, so it is a strict $L_\infty$-isomorphism.
\end{proof}

\begin{lemma}\label{1SCFBTZ34}
Let $S,S'\subset M$ be Poisson transversal submanifolds containing $N$ coisotropically, and fix a tubular neighbourhood $\psi$ of $N$ in $S$. Then there are a tubular neighbourhood $\psi'$ of $N$ in $S'$ and a strict isomorphism of $L_\infty$-algebras $\Psi\colon(\a,\{\ell_k^{\a,S,\psi}\})\to(\a,\{\ell_k^{\a,S',\psi'}\})$.
\end{lemma}

\begin{proof}
By \cite[Theorem 4.3]{CattaneoZambon2009} there is a Poisson diffeomorphism $\varphi\colon S\to S'$ with $\varphi|_N=\id_N$, defined after shrinking $S$ and $S'$ to smaller tubular neighbourhoods of $N$. Since $\g_S$, $\pi_S$ and $\b_S$ depend only on the germ of $S$ along $N$, and likewise for $S'$, we may regard $\varphi$ as defined on all of $S$. Let $\Lambda\colon\nu_S\to\nu_{S'}$ be induced by $\varphi$, $\psi':=\varphi\circ\psi\circ\Lambda^{-1}$ and $\theta:=\wedge^\bullet\Lambda\colon\b_S\to\b_{S'}$. The first part of the proof of \Cref{1SCEXNP25} applies verbatim with $\varphi$ in place of $\Theta|_S$: it uses only that $\varphi$ is a Poisson diffeomorphism with $\varphi(N)=N$ and $\psi'\circ\Lambda=\varphi\circ\psi$, and shows that $\theta$ is a strict isomorphism $(\b_S,\{\ell_k^{\b_S,\psi}\})\to(\b_{S'},\{\ell_k^{\b_{S'},\psi'}\})$. Conjugating by the algebra isomorphisms $\Phi_S$ and $\Phi_{S'}$ of \eqref{1SCVX2YHQ}, the map $\Psi:=\Phi_{S'}^{-1}\circ\theta\circ\Phi_S$ is a strict isomorphism of $L_\infty$-algebras from $(\a,\{\ell_k^{\a,S,\psi}\})$ to $(\a,\{\ell_k^{\a,S',\psi'}\})$. 
\end{proof}

\begin{lemma}\label{1SC229L5W}
Let $S\subset M$ be a Poisson transversal containing $N$ coisotropically, and let $\psi,\psi'$ be tubular neighbourhoods of $N$ in $S$. Then $(\a,\{\ell_k^{\a,S,\psi}\})$ and $(\a,\{\ell_k^{\a,S,\psi'}\})$ are $L_\infty$-isomorphic.
\end{lemma}

\begin{proof}
By \cite[Theorem 4.3]{cattaneo2008equivalences}, whose proof checks the hypotheses of \cite[Theorem 3.2]{cattaneo2008equivalences} for an isotopy between $\psi$ and $\psi'$, there is an $L_\infty$-isomorphism $\varrho=\{\varrho_k\}_{k\ge1}$ from $(\b_S,\{\ell_k^{\b_S,\psi}\})$ to
$(\b_S,\{\ell_k^{\b_S,\psi'}\})$ with $\varrho_1$ bijective. The structures compared in loc.\ cit.\ are the graded-symmetric higher derived brackets $\{\cdots\}^{\psi}$ and $\{\cdots\}^{\psi'}$ of Voronov, that is, the brackets $P_S[\cdots[\pi_S,j_S(\cdot)],\dots,j_S(\cdot)]$ of \eqref{1SCRZZPXA} before the sign $\varepsilon(b)$ is applied \cite[Remark 4.1]{cattaneo2008equivalences}, and $\varrho$ is an $L_\infty[1]$-isomorphism between them. Applying $\operatorname{dec}_k$ of \eqref{EQDECALAGE} to its components converts it to a morphism for \eqref{1SCFV8US5}, because $\operatorname{dec}$ carries the symmetric morphism relations to \eqref{1SCFV8US5} with the sign $\chi'$ specified there, and the rescaling of \Cref{1SCRESCALE} with $t=-1$ then gives a morphism between the structures \eqref{1SCRZZPXA}; this conversion is independent of the desuspension convention used in loc.\ cit., as only the symmetric structures enter. We transport $\varrho$ to $\a$.

By \Cref{1SCU92AQH}(3) with $\Phi=\Phi'=\Phi_S$, the maps $\Phi_S^{-1}\circ\varrho_k\circ\Phi_S^{\otimes k}$ form an $L_\infty$-isomorphism from $(\a,\{\ell_k^{\a,S,\psi}\})$ to $(\a,\{\ell_k^{\a,S,\psi'}\})$.
\end{proof}

\begin{theorem}\label{1SC4NFQ6C}
The $P_\infty$-structure on $\a$ is independent, up to $L_\infty$-isomorphism, of the choice of Poisson transversal submanifold $S\subset M$ containing $N$ coisotropically and of the tubular neighbourhood $\psi$ of $N$ in $S$.
\end{theorem}

\begin{proof}
Let $S,S'\subset M$ be Poisson transversals containing $N$ coisotropically, with tubular neighbourhoods $\psi$ of $N$ in $S$ and $\psi''$ of $N$ in $S'$. Let $\psi'$ be the tubular neighbourhood of $N$ in $S'$ adapted to $\psi$ as in \Cref{1SCFBTZ34}, and let $\Psi$ be the resulting strict isomorphism from $\{\ell_k^{\a,S,\psi}\}$ to $\{\ell_k^{\a,S',\psi'}\}$. It is bijective of degree $0$ and satisfies
\begin{equation}\label{1SC8VUMG3}
\Psi\circ\ell_k^{\a,S,\psi}=\ell_k^{\a,S',\psi'}\circ\Psi^{\otimes k} \quad\text{for every }k\ge0.
\end{equation}
By \Cref{1SC229L5W} applied to $\psi'$ and $\psi''$ inside $S'$, there is an $L_\infty$-isomorphism $\tilde\varrho=\{\tilde\varrho_k\}_{k\ge1}$ from
$\{\ell_k^{\a,S',\psi'}\}$ to $\{\ell_k^{\a,S',\psi''}\}$ with $\tilde\varrho_1$
bijective. By \eqref{1SC8VUMG3}, $\{\ell_k^{\a,S,\psi}\}$ is the transport of $\{\ell_k^{\a,S',\psi'}\}$ along $\Psi$. Hence \Cref{1SCU92AQH}(3), with $\Phi=\Psi$ and $\Phi'=\id$, shows that the maps $\tilde\varrho_k\circ\Psi^{\otimes k}$ form an $L_\infty$-isomorphism from $(\a,\{\ell_k^{\a,S,\psi}\})$ to $(\a,\{\ell_k^{\a,S',\psi''}\})$.
\end{proof}

\subsection{Relation to symplectic reduction along a submanifold}\label{1SCVT88UN}
Crooks and Mayrand \cite[Proposition 2.23]{CrooksMayrand2022} attach a Poisson algebra to a pre-Poisson submanifold: the functions on the quotient of $N$ by a stabilizer subgroupoid. We show that this algebra is $H^0(\a,\ell_1)$ with the bracket induced by $\ell_2$. Both constructions pass through a Poisson transversal containing $N$ coisotropically, so we may compute the two brackets inside the submanifold $S$ used to define the $P_\infty$-structure; the comparison is then pointwise.

A \defn{stabilizer subgroupoid} of $N$ is an isotropic Lie subgroupoid $\HHH\tto N$ with Lie algebroid $L$, inside a symplectic groupoid integrating $(M,\pi)$ \cite[Definition 2.8]{CrooksMayrand2022}. We assume throughout that $\HHH$ is source-connected. Its orbits are then the connected integral leaves of $\rho_L(L)$. When $N/\HHH$ is a manifold, $q\colon N\to N/\HHH$ denotes the quotient map. Fix a Poisson transversal $S$ containing $N$ coisotropically and a tubular neighbourhood with projection $p\colon U\to N$. Write $\ell_k:=\ell_k^{\a,S,\psi}$ and $r:=r_S$ for the isomorphism \eqref{1SCPEDZZK}. We write $\{\cdot,\cdot\}$ for the bracket of $S$ and $\{\cdot,\cdot\}_M$ for that of $M$. For $f\in C^\infty(S)$ and $F\in C^\infty(M)$ we write $X^S_f:=\pi_S(df)$ and $X^M_F:=\pi(dF)$ for the Hamiltonian vector fields. Given $f\in C^\infty(N)$, a function $F\in C^\infty(V)$ on an open neighbourhood $V$ of $N$ in $M$ is an \defn{admissible extension} of $f$ if $F|_N=f$ and $dF\bigl(\pi(TN^\circ)\bigr)=0$ along $N$.

By \Cref{1SCMMB63H} the algebra $\a$ is flat and $\ell_1=d_L$. On $\a^0=C^\infty(N)$ we have $(d_Lf)(\xi)=\L_{\rho_L(\xi)}f$, so
\[
H^0(\a,\ell_1)=\bigl\{f\in C^\infty(N):df\bigl(\rho_L(L)\bigr)=0\bigr\}.
\]
By \cite[\S 2.5]{CrooksMayrand2022}, $\rho_L(L)=TN\cap\pi(TN^\circ)$ is the distribution tangent to the orbits of $\HHH$, so $H^0(\a,\ell_1)=C^\infty(N)^\HHH$. Since $\a$ is flat, $\ell_2$ descends to $H^\bullet(\a,\ell_1)$ and makes $\bigl(C^\infty(N)^\HHH,\ell_2\bigr)$ a Poisson algebra.

Since $\Phi_S$ is the identity in degree zero, \eqref{1SCRZZPXA} gives
\begin{equation}\label{1SCQHGY3K}
\ell_2(f,g)=\{p^*f,p^*g\}|_N,\qquad f,g\in C^\infty(N).
\end{equation}
Indeed $\varepsilon(f,g)=-1$ and $[\pi_S,p^*f]=-\pi_S(d\,p^*f)$ in the convention $[X,h]=X(h)$ for the Schouten bracket, with $\pi(\alpha):=\pi(\alpha,\cdot)$ and $\{h,h'\}:=\pi(dh,dh')$, so $\ell_2(f,g)=-([\pi_S,p^*f])(p^*g)|_N=\{p^*f,p^*g\}|_N$.

\begin{proposition}\label{1SCJHB7Z7}
\begin{itemize}
\item[(1)] For all $f,g\in H^0(\a,\ell_1)$ and all admissible extensions $F,G$ of $f$ and $g$ respectively, $\ell_2(f,g)=\{F,G\}_M|_N$.
\item[(2)] If $N/\HHH$ is a manifold and $q$ is a submersion, then
\[
q^*\colon\bigl(C^\infty(N/\HHH),\{\cdot,\cdot\}_{N/\HHH}\bigr)\to
\bigl(H^0(\a,\ell_1),\ell_2\bigr)
\]
is an isomorphism of Poisson algebras.
\end{itemize}
\end{proposition}

\begin{proof}
The function $\{F,G\}_M|_N$ does not depend on the admissible extensions (cf.\ the proof of \cite[Proposition 2.23]{CrooksMayrand2022}). Indeed, if $F|_N=0$ and $dF(\pi(TN^\circ))=0$ along $N$, then $dF\in TN^\circ\cap\pi^{-1}(TN)=L$ along $N$, so $X^M_F|_N\in\pi(L)\subset\pi(TN^\circ)$, which $dG$ annihilates; hence $\{F,G\}_M|_N=dG(X^M_F)|_N=0$, and symmetrically in $G$. It is therefore enough to exhibit one pair of admissible extensions for which the identity holds.

Since $S$ is a Poisson transversal, $TM|_S=TS\oplus\pi(TS^\circ)$, so $\pi(TS^\circ)$ realizes the normal bundle of $S$ in $M$. The tubular neighbourhood theorem provides a projection $\ppp_W\colon W\to S$, defined on an open neighbourhood $W$ of $S$ in $M$, with $\ppp_W|_S=\id_S$ and $\ker d\ppp_W|_S=\pi(TS^\circ)$. Set $V:=\ppp_W^{-1}(U)$, an open neighbourhood of
$N$ in $M$, and
\[
F:=(p\circ\ppp_W)^*f,\quad G:=(p\circ\ppp_W)^*g\ \in C^\infty(V).
\]
Then $F|_U=p^*f$ and $G|_U=p^*g$. Hence $F|_N=f$ and $G|_N=g$, and $dF$ and $dG$ annihilate $\ker d\ppp_W=\pi(TS^\circ)$ along $S$. By the proof of \cite[Proposition 2.23]{CrooksMayrand2022},
$\pi(TN^\circ)=\rho_L(L)\oplus\pi(TS^\circ)$. Also, $dF|_{TN}=df$ annihilates $\rho_L(L)$ because $f\in H^0(\a,\ell_1)$. Hence $F$ is an admissible extension of $f$, and likewise $G$ is one of $g$.

Fix $x\in N$. Since $dF_x$ annihilates $\pi(T_xS^\circ)$ and $\pi(T_xS^\circ)^\circ=\pi^{-1}(T_xS)$, we have $dF_x\in\pi^{-1}(T_xS)$. As $S$ is a Poisson transversal, restriction of covectors $\pi^{-1}(T_xS)\to T^*_xS$ is bijective, and $\pi_S$ is defined by $\pi_S(\xi)=\pi(\hat\xi)$ for $\hat\xi$ the preimage of $\xi$ \cite[\S 2.2]{CrooksMayrand2022}. Since $dF_x$ restricts to $d(p^*f)_x$,
\[X^M_F(x)=\pi(dF_x)=\pi_S\bigl(d(p^*f)_x\bigr)=X^S_{p^*f}(x)\in T_xS .
\]
Therefore $\{F,G\}_M(x)=dG_x\bigl(X^M_F(x)\bigr)
=d(p^*g)_x\bigl(X^S_{p^*f}(x)\bigr) =\{p^*f,p^*g\}(x)$, which is $\ell_2(f,g)(x)$ by \eqref{1SCQHGY3K}.

For (2), $q^*$ identifies $C^\infty(N/\HHH)$ with $C^\infty(N)^\HHH$ because $q$ is a submersion, and by (1) it carries $\{\cdot,\cdot\}_{N/\HHH}$ to $\ell_2$.
\end{proof}

\subsection{Local models for infinitesimal $1$-shifted coisotropics}\label{1SCHTFL8N}
We now pass from pre-Poisson submanifolds to regular infinitesimal $1$-shifted coisotropics. We replace each twisted Dirac manifold by a twisted Poisson local model in which it embeds coisotropically, and the map $\ccc$ by its graph, which is a pre-Poisson submanifold once the twist is removed. The $P_\infty$-structure is then transported along a canonical isomorphism between $B_\ccc$ and the Lie algebroid of this graph.

Let $(M,L_M,\phi)$ and $(N,L_N,\psi)$ be twisted Dirac manifolds, and let $\ccc\colon N\to M$ be a smooth map such that $\ccc^*\phi=\psi$ and
\begin{equation}\label{1SC8M26B9}
B_\ccc := \left\{ \bigl((v,\alpha),(w,\beta)\bigr)
\in L_N\times L_M : \ccc_*v=w,\; \alpha=\ccc^*\beta
\right\}
\end{equation}
has constant rank; that is, $(\ccc,L_N,L_M)$ is an infinitesimal $1$-shifted coisotropic \cite[Definition 3.13]{Maxence2025}. Then $B_\ccc$ is a Lie algebroid over $N$ with anchor $\rho_B\bigl((v,\alpha),(w,\beta)\bigr):=v$, and $\ppp_N\colon B_\ccc\to L_N$ is a morphism of Lie algebroids.

Write $E_M:=L_M\cap TM$ and $E_N:=L_N\cap TN$. The triple $(\ccc,L_N,L_M)$ is a \defn{regular datum} if $E_M$ and $E_N$ have constant rank. Set $\a:=C^\infty(\wedge^\bullet B_\ccc^*)$.

We first build a local model that replaces a twisted Dirac structure by a twisted Poisson one, together with a canonical lift of $L$ to the model. 

\begin{lemma}\label{1SCCGC8W3}
Let $(M,L,\phi)$ be a twisted Dirac manifold such that $E:=L\cap TM$ has constant rank. Fix a complement $TM=E\oplus V$, and let $\ppp_E\colon TM\to E$ be the projection along $V$. Let $p\colon E^*\to M$ be the bundle projection, $\iota\colon E^*\hookrightarrow T^*M$ the inclusion $\iota(\eta)=\eta\circ\ppp_E$, and $\theta$ the tautological $1$-form on $T^*M$. Set $\Omega:=-d(\iota^*\theta)$. Then there is a neighbourhood $S$ of the zero section $j\colon M\hookrightarrow E^*$ such that the following hold.
\begin{itemize}
\item[(1)] There is a $p^*\phi$-twisted Poisson bivector $\pi_S$ with $e^\Omega p^*L=\Gamma_{\pi_S}$.
\item[(2)] $j^*\Gamma_{\pi_S}=L$, and $j(M)$ is coisotropic\footnote{For an arbitrary bivector field $\pi$ on a manifold $S$, a submanifold $C\subset S$ is \emph{coisotropic} if $\pi(TC^\circ)\subset TC$, where $TC^\circ\subset T^*S|_C$ is the annihilator of $TC$.} in $(S,\pi_S)$.
\item[(3)] For every $(v,\alpha)\in L_x$ there is a unique $\zeta\in T^*_{j(x)}S$ with
\begin{equation}\label{1SCZJU22M}
\pi_S(\zeta)=j_*v,\quad j^*\zeta=\alpha,
\end{equation}
namely $h_L(v,\alpha):=p^*\alpha+\iota_{j_*v}\Omega$. The map $h_L\colon L\to j^*T^*S$ is smooth.
\item[(4)] $\dim S=\dim M+\rank E$.
\end{itemize}
\end{lemma}

\begin{proof}
We prove (1), then the first half of (2), then (3), and finally the coisotropy assertion in (2). Since $p$ is a submersion, $p^*L$ is a smooth $p^*\phi$-twisted Dirac structure, and since $d\Omega=0$, so is $e^\Omega p^*L$.

We first describe $e^\Omega p^*L$ along the zero section. At $0_x\in E^*$ there is a canonical splitting $T_{0_x}E^*=T_xM\oplus E_x^*$, in which $j_*v=(v,0)$ and $p_*(v,\eta)=v$. Since $\ppp_M\circ \iota=p$, where $\ppp_M\colon T^*M\to M$ is the projection, $(\iota^*\theta)_\eta(X)=\eta(\ppp_E\,p_*X)$. In particular $\iota^*\theta$ vanishes at the points of the zero section, and a direct computation gives
\begin{equation}\label{1SCV2A689}
\Omega_{0_x}\bigl((v,\eta),(w,\beta)\bigr)=\beta(\ppp_Ev)-\eta(\ppp_E w).
\end{equation}
Since $(p^*L)_{0_x}=\{((v,\eta),p^*\alpha):(v,\alpha)\in L_x,\ \eta\in E_x^*\}$,
\begin{equation}\label{1SCUL4NDR}
(e^\Omega p^*L)_{0_x}=\Bigl\{\bigl((v,\eta),\,p^*\alpha+\iota_{(v,\eta)}\Omega\bigr): (v,\alpha)\in L_x,\ \eta\in E_x^*\Bigr\}.
\end{equation}
Let $(v,\alpha)\in L_x$, $\eta\in E_x^*$ and $\zeta:=p^*\alpha+\iota_{(v,\eta)}\Omega$. By \eqref{1SCV2A689},
\begin{equation}\label{1SC7W492D}
\zeta(w,0)=\alpha(w)-\eta(\ppp_Ew),\quad \zeta(0,\beta)=\beta(\ppp_Ev).
\end{equation}

\emph{Claim: if $j^*\zeta=0$, then $\eta=0$, $\alpha=0$ and $v\in E_x$.} Indeed, \eqref{1SC7W492D} gives $\alpha=\eta\circ\ppp_E$. For $u\in E_x$ we have $(u,0)\in L_x$, so isotropy of $L_x$ gives $0=\alpha(u)=\eta(u)$. Hence $\eta=0$ and $\alpha=0$. Then $(v,0)\in L_x$, that is, $v\in E_x$.

(1) Let $((v,\eta),\zeta)$ be an element of \eqref{1SCUL4NDR} with $\zeta=0$. By the claim, $\eta=0$ and $v\in E_x$. By \eqref{1SC7W492D}, $\beta(v)=0$ for all $\beta\in E_x^*$, so $v=0$. Hence $(e^\Omega p^*L)_{0_x}\cap T_{0_x}E^*=0$. Transversality to $TE^*$ is an open condition, so it holds on an open neighbourhood $S$ of the zero section, on which $e^\Omega p^*L$ is the graph of a bivector $\pi_S$, which is $p^*\phi$-twisted Poisson. Part (4) holds since $S$ is open in $E^*$.

(2), first part. Since $(\iota^*\theta)_{0_x}=0$ for all $x$, we have $j^*(\iota^*\theta)=0$, hence $j^*\Omega=0$ and $j^*\Gamma_{\pi_S}=j^*(p^*L)$. Let $v\in T_xM$. By \eqref{1SCUL4NDR} with $\Omega$ omitted, $(j_*v,\zeta)\in p^*L$ if and only if $\zeta=p^*\alpha$ with $(v,\alpha)\in L_x$, and then $j^*\zeta=\alpha$ since $p\circ j=\id_M$. Hence $j^*\Gamma_{\pi_S}=L$.

(3) Let $(v,\alpha)\in L_x$ and $\zeta:=h_L(v,\alpha)$. By \eqref{1SCUL4NDR} with $\eta=0$, $(j_*v,\zeta)\in\Gamma_{\pi_S}$, so $\pi_S(\zeta)=j_*v$, and by \eqref{1SC7W492D}, $j^*\zeta=\alpha$. For uniqueness, let $\zeta_0$ be the difference of two solutions of \eqref{1SCZJU22M}. Then $\pi_S(\zeta_0)=0$ and $j^*\zeta_0=0$. The first condition says $(0,\zeta_0)\in\Gamma_{\pi_S}$, so by \eqref{1SCUL4NDR} there are $(v,\alpha_0)\in L_x$ and $\eta\in E_x^*$ with $(0,\zeta_0)=\bigl((v,\eta),p^*\alpha_0+\iota_{(v,\eta)}\Omega\bigr)$. Hence $v=0$, $\eta=0$, and $\zeta_0=p^*\alpha_0$. Since $p\circ j=\id_M$, $\alpha_0=j^*p^*\alpha_0=j^*\zeta_0=0$, so $\zeta_0=0$. Finally, $j^*T^*S$ is a smooth vector bundle over $M$, and $h_L$ is the restriction to the smooth subbundle $L\subset TM\oplus T^*M$ of the bundle map $(v,\alpha)\mapsto p^*\alpha+\iota_{j_*v}\Omega$, which is smooth since $p$, $j$ and $\Omega$ are.

(2), coisotropy. Let $\zeta\in(T_{0_x}j(M))^\circ$, so that $j^*\zeta=0$, and write $\pi_S(\zeta)=(v,\eta)$. Since $(\pi_S(\zeta),\zeta)\in\Gamma_{\pi_S}$, \eqref{1SCUL4NDR} gives $\zeta=p^*\alpha+\iota_{(v,\eta)}\Omega$ for some $\alpha$ with $(v,\alpha)\in L_x$. As $j^*\zeta=0$, the claim gives $\eta=0$ and $v\in E_x$. Hence $\pi_S(\zeta)=(v,0)=j_*v\in T_{0_x}j(M)$.
\end{proof}

We apply \Cref{1SCCGC8W3} to $(M,L_M,\phi)$ and $(N,L_N,\psi)$. This gives $(S_M,\pi_M,p_M,j_M)$ and $(S_N,\pi_N,p_N,j_N)$, together with lifts $h_M$ and $h_N$. Set
\begin{equation}\label{1SCJ9SYME}
\widehat M := S_M\times\overline{S_N},
\quad \widehat\pi := \pi_M\oplus(-\pi_N),
\quad \widehat\phi := \ppp_{S_M}^*p_M^*\phi - \ppp_{S_N}^*p_N^*\psi, 
\end{equation}
where $\overline{S_N}$ denotes $S_N$ equipped with the bivector $-\pi_N$. Then $\widehat\pi$ is a $\widehat\phi$-twisted Poisson bivector. We replace $\ccc$ by the embedding
\[J\colon N\to\widehat M,\quad J(n)=\bigl(j_M(\ccc(n)),j_N(n)\bigr),
\]
and set $\Delta:=J(N)$ and $\widehat B := T\Delta^\circ\cap \widehat\pi^{-1}(T\Delta)$.
Since $\ccc^*\phi=\psi$,
\begin{equation}\label{1SCRPX4HC}
J^*\widehat\phi = 0.
\end{equation}
If $\widehat B$ has constant rank, it carries the Lie algebroid structure induced by the twisted cotangent Lie algebroid $T^*_{\widehat\pi}\widehat M$ of $(\widehat M,\widehat\pi,\widehat\phi)$.

\begin{lemma}\label{1SCL3WW9S}
Let $(\ccc,L_N,L_M)$ be a regular datum. The map
\[
\varphi\colon B_\ccc\longrightarrow\widehat B,
\quad \varphi\bigl((v,\alpha),(\ccc_*v,\beta)\bigr)
= \bigl(h_M(\ccc_*v,\beta),-h_N(v,\alpha)\bigr),
\]
is an isomorphism of Lie algebroids over $J\colon N\to\Delta$.
\end{lemma}

\begin{proof}
We first show that $\varphi$ is an isomorphism of vector bundles. A covector $(\xi,\eta)$ lies in $\widehat B$ if and only if $\ccc^*j_M^*\xi+j_N^*\eta=0$ and
\[
\pi_M\xi=j_{M*}\ccc_*v,\quad
\pi_N(-\eta)=j_{N*}v
\]
for some $v\in T_nN$, which is then unique because $j_{N*}$ is injective.

Let $b=\bigl((v,\alpha),(\ccc_*v,\beta)\bigr)\in B_\ccc$. By \Cref{1SCCGC8W3}(3), $\pi_M h_M(\ccc_*v,\beta)=j_{M*}\ccc_*v$ and $\pi_N h_N(v,\alpha)=j_{N*}v$. Moreover
$\ccc^*j_M^*h_M(\ccc_*v,\beta)-j_N^*h_N(v,\alpha)=\ccc^*\beta-\alpha=0$.
Hence $\varphi(b)\in\widehat B$. If $\varphi(b)=0$, then $\beta=0$ and $j_{N*}v=0$, so $b=0$.

Conversely, let $(\xi,\eta)\in\widehat B$ with $v$ as above. Set $\beta:=j_M^*\xi$ and $\alpha:=-j_N^*\eta$. Since $j_M^*\Gamma_{\pi_M}=L_M$ and $j_N^*\Gamma_{\pi_N}=L_N$, we have $(\ccc_*v,\beta)\in L_M$ and $(v,\alpha)\in L_N$. The conormality condition gives $\alpha=\ccc^*\beta$. By uniqueness in \Cref{1SCCGC8W3}(3), $\xi=h_M(\ccc_*v,\beta)$ and $-\eta=h_N(v,\alpha)$. So $\varphi$ is an isomorphism of vector bundles.

The anchors are compatible by the tangency condition: $\widehat\pi\circ\varphi=J_*\circ\rho_B$.

It remains to compare the brackets. Let $\iota:=j_M\times j_N\colon M\times N\to\widehat M$, and $L_N^-:=\{(v,-\alpha):(v,\alpha)\in L_N\}=j_N^*\Gamma_{-\pi_N}$. Then $\iota^*\Gamma_{\widehat\pi}=L_M\times L_N^-$. The map $\zeta\mapsto(u,\iota^*\zeta)$, where $\widehat\pi\zeta=\iota_*u$, sends $\widehat\pi^{-1}(T\iota(M\times N))$ onto $L_M\times L_N^-$. It preserves brackets: if $\widehat\pi\zeta_i$ is $\iota$-related to $u_i$, then $[\widehat\pi\zeta_1,\widehat\pi\zeta_2]$ is $\iota$-related to $[u_1,u_2]$, and $\iota^*$ commutes with $\L$ (the Lie derivative along $\iota$-related vector fields), $d$ and contractions, so $\iota^*\courant{\zeta_1,\zeta_2}_{\widehat\phi}=\courant{\iota^*\zeta_1,\iota^*\zeta_2}_{\iota^*\widehat\phi}$ on covector parts. By \Cref{1SCCGC8W3}(3) applied to each factor, its inverse is
$\bigl((w,\beta),(v,-\alpha)\bigr)\mapsto\bigl(h_M(w,\beta),-h_N(v,\alpha)\bigr)$.
Since $(v,\alpha)\mapsto(v,-\alpha)$ is a Lie algebroid isomorphism $L_N\to L_N^-$, we obtain a Lie algebroid isomorphism $L_M\times L_N\to\widehat\pi^{-1}(T\iota(M\times N))$. Its restriction to the graph of $\ccc$ is $\varphi$, and $\iota(\ccc(n),n)=J(n)$. The bracket of $B_\ccc$ is the restriction of that of $L_N\times L_M$ to the graph of $\ccc$ \cite[Definition 3.13]{Maxence2025}, and that of $\widehat B$ is the restriction of the bracket of $T^*_{\widehat\pi}\widehat M$ to $\Delta$. Hence $[\varphi(b_1),\varphi(b_2)]_{\widehat B}=\varphi([b_1,b_2]_{B_\ccc})$.
\end{proof}

We now remove the background $3$-form without changing $\widehat\pi$ along $\Delta$.

\begin{lemma}\label{1SCPWPNPJ}
There are a neighbourhood $U$ of $\Delta$ in $\widehat M$, a $2$-form $\lambda\in\Omega^2(U)$ with $d\lambda=\widehat\phi|_U$ and $\lambda_q=0$ for all $q\in\Delta$,
and a Poisson bivector $\widehat\pi_0$ on $U$ with
\[
e^\lambda\Gamma_{\widehat\pi}=\Gamma_{\widehat\pi_0},
\quad
\widehat\pi_0|_\Delta=\widehat\pi|_\Delta.
\]
Moreover $A\colon T^*_{\widehat\pi}U\longrightarrow T^*_{\widehat\pi_0}U$, $A(\zeta)=\zeta+\iota_{\widehat\pi(\zeta)}\lambda$,
is an isomorphism of Lie algebroids with $A_q=\id$ for $q\in\Delta$. In particular
\[
\widehat B=T\Delta^\circ\cap\widehat\pi_0^{-1}(T\Delta),
\]
and the Lie algebroid structures induced on $\widehat B$ by $(\widehat\pi,\widehat\phi)$ and by $\widehat\pi_0$ coincide.
\end{lemma}

\begin{proof}
The form $\widehat\phi$ is closed and $J^*\widehat\phi=0$ by \eqref{1SCRPX4HC}, so by the relative Poincar\'e lemma \cite[Proposition 6.8]{da2004lectures} there are a neighbourhood $U$ of $\Delta$ and $\lambda\in\Omega^2(U)$ with $d\lambda=\widehat\phi|_U$ and $\lambda_q=0$ for all $q\in\Delta$.

By \cite[\S 3]{vsevera2001poisson}, the gauge transformation $e^\lambda(X,\alpha)=(X,\alpha+\iota_X\lambda)$ intertwines the $\widehat\phi$-twisted Courant bracket with the $(\widehat\phi-d\lambda)$-twisted one. Hence $e^\lambda\Gamma_{\widehat\pi}$ is untwisted, and
$e^\lambda\Gamma_{\widehat\pi}=\bigl\{(\widehat\pi\zeta,A\zeta):\zeta\in T^*U\bigr\}$.
Since $\lambda_q=0$, we have $A_q=\id$ for $q\in\Delta$, so $A$ is invertible after shrinking $U$. Then $e^\lambda\Gamma_{\widehat\pi}=\Gamma_{\widehat\pi_0}$ with $\widehat\pi_0:=\widehat\pi\circ A^{-1}$, which is a bivector because the graph is Lagrangian, and Poisson because the graph is an untwisted Dirac structure. At points of $\Delta$, $A_q=\id$ gives $\widehat\pi_0=\widehat\pi$.

The map $A$ is the composite of $\zeta\mapsto(\widehat\pi\zeta,\zeta)$, of $e^\lambda$, and of the projection $\Gamma_{\widehat\pi_0}\to T^*U$. Each preserves brackets, so $A$ is an isomorphism of Lie algebroids, and $\widehat\pi_0\circ A=\widehat\pi$.

Since $\widehat\pi_0=\widehat\pi$ along $\Delta$, the two descriptions of $\widehat B$ agree. Let $b_1,b_2\in\Gamma(\widehat B)$ with extensions $\zeta_1,\zeta_2$. As $A_q=\id$ along $\Delta$, the sections $A\zeta_i$ extend $b_i$, and every extension is of this form. Hence
\[
[A\zeta_1,A\zeta_2]_{\widehat\pi_0}\big|_\Delta
=A[\zeta_1,\zeta_2]_{\widehat\pi}\big|_\Delta
=[\zeta_1,\zeta_2]_{\widehat\pi}\big|_\Delta=[b_1,b_2]_{\widehat B},
\]
the last equality by \Cref{1SCL3WW9S}.
\end{proof}

The model depends on auxiliary choices: the complements $V_M,V_N$ of \Cref{1SCCGC8W3} and the primitive $\lambda$ of \Cref{1SCPWPNPJ}. We now show that it is unique up to a Poisson diffeomorphism fixing $\Delta$. Only $\Omega$, $\lambda$ and hence $\widehat\pi_0$ vary; $E_M^*$, $E_N^*$, $J$ and $\Delta$ do not.

We view a $2$-form $\lambda$ also as the map $X\mapsto\iota_X\lambda$, and write $\Theta^*\pi:=(\Theta^{-1})_*\pi$ for a bivector $\pi$. We use Moser's theorem in the following form \cite[\S 3.3]{AlekseevMeinrenken2007}: if $\pi_t$ is a smooth family of Poisson bivectors near $\Delta$ with $\Gamma_{\pi_t}=e^{\beta_t}\Gamma_{\pi_0}$, where $\beta_0=0$ and $\dot\beta_t=d\gamma_t$ with $\gamma_t$ vanishing at the points of $\Delta$, then the time-$1$ flow of $-\pi_t(\gamma_t)$ fixes $\Delta$ pointwise and carries $\pi_0$ to $\pi_1$.  

Write $\widehat p :=p_M\times p_N$, $\widehat\Omega:=\ppp_{S_M}^*\Omega_M-\ppp_{S_N}^*\Omega_N$ and $L_N^-:=\{(v,-\alpha):(v,\alpha)\in L_N\}$. By \Cref{1SCCGC8W3}(1),
\begin{equation}\label{1SC5T4EWN}
\Gamma_{\widehat\pi_0}=e^{\lambda+\widehat\Omega}\widehat p^*(L_M\times L_N^-).
\end{equation}

\begin{proposition}\label{1SC454C3Y}
Let $\widehat\pi_0$ and $\widehat\pi_0'$ come from two choices of $(V_M,V_N)$ and of primitive. Then there is a diffeomorphism $\Theta$ between neighbourhoods of $\Delta$ with $\Theta|_\Delta=\id_\Delta$ and $\Theta^*\widehat\pi_0'=\widehat\pi_0$. Its cotangent lift restricts to an isomorphism of Lie algebroids from $T\Delta^\circ\cap\widehat\pi_0'^{-1}(T\Delta)$ onto $\widehat B$ over $\id_\Delta$.
\end{proposition}

\begin{proof}
Denote by primes the objects of the second choice. For $t\in[0,1]$, the map $(1-t)\ppp_{E_M}+t\ppp_{E_M}'$ is a projection onto $E_M$; let $V_{M,t}$ be its kernel, and define $V_{N,t}$ similarly. Since $\iota^*\theta$ is linear in $\ppp_E$, the forms of \Cref{1SCCGC8W3} are $\widehat\Omega_t=(1-t)\widehat\Omega+t\widehat\Omega'$. Set $\lambda_t:=(1-t)\lambda+t\lambda'$, so $d\lambda_t=\widehat\phi$ and $\lambda_t$ vanishes at the points of $\Delta$. Let
\[
\beta_t:=t\bigl[(\lambda'-\lambda)+(\widehat\Omega'-\widehat\Omega)\bigr],
\quad
\Gamma_t:=e^{\lambda_t+\widehat\Omega_t}\widehat p^*(L_M\times L_N^-)=e^{\beta_t}\Gamma_{\widehat\pi_0}.
\]
Each $\Gamma_t$ is the graph of a Poisson bivector $\pi_t$ near $\Delta$. Its background $3$-form is $\widehat\phi-d\lambda_t-d\widehat\Omega_t=0$. Along $\Delta$, where $\lambda_t=0$, it is the model of \Cref{1SCCGC8W3} for $(V_{M,t},V_{N,t})$, hence transverse to $T\widehat M$ by the proof of \Cref{1SCCGC8W3}(1). Transversality is open and $[0,1]$ is compact, so it holds on $[0,1]\times U$ for some neighbourhood $U$ of $\Delta$. Here $\pi_0=\widehat\pi_0$ and $\pi_1=\widehat\pi_0'$.

We write $\dot\beta_t$ as $d\gamma$ with $\gamma$ vanishing at the points of $\Delta$. Set
\[
\gamma_V:=\ppp_{S_M}^*\bigl(\iota_M^*\theta_M-\iota_M'^*\theta_M\bigr)-\ppp_{S_N}^*\bigl(\iota_N^*\theta_N-\iota_N'^*\theta_N\bigr).
\]
Since $\Omega=-d(\iota^*\theta)$, we get $d\gamma_V=\widehat\Omega'-\widehat\Omega$, and $\gamma_V$ vanishes at the points of $\Delta$ because $\iota^*\theta$ and $\iota'^*\theta$ vanish at the zero section. The form $\lambda'-\lambda$ is closed and vanishes at the points of $\Delta$, so $\lambda'-\lambda=d\gamma_\lambda$ with $\gamma_\lambda$ vanishing there \cite[Proposition 6.8]{da2004lectures}. Take $\gamma:=\gamma_V+\gamma_\lambda$. Moser's theorem gives $\Theta$ with $\Theta|_\Delta=\id_\Delta$ and $\Theta^*\widehat\pi_0'=\widehat\pi_0$. Its cotangent lift is an isomorphism of cotangent Lie algebroids, and $d\Theta$ preserves $T\Delta$ along $\Delta$, so it maps $T\Delta^\circ\cap\widehat\pi_0'^{-1}(T\Delta)$ onto $\widehat B$, preserving brackets.
\end{proof}

\subsection{Construction and independence}\label{1SCSWUFL4}
We now construct the $P_\infty$-structure on $\a=C^\infty(\wedge^\bullet B_\ccc^*)$. By \Cref{1SCL3WW9S,1SCPWPNPJ} this reduces to the pre-Poisson case of \Cref{1SCS4XYSN}, applied to $\Delta\subset(U,\widehat\pi_0)$. Since $S$ and $\psi$ are already in use, we write $\Sigma$ for a Poisson transversal containing $\Delta$ coisotropically and $\varsigma$ for a tubular neighbourhood of $\Delta$ in $\Sigma$.

A \defn{choice} $\mathbf{e}$ consists of complements $V_M,V_N$ as in \Cref{1SCCGC8W3} and of a primitive $\lambda$ as in \Cref{1SCPWPNPJ}. For each $\mathbf{e}$, write $\widehat\pi_0^{\mathbf{e}}$, $\widehat B^{\mathbf{e}}$ and $\varphi^{\mathbf{e}}$ for the objects of \Cref{1SCL3WW9S,1SCPWPNPJ}, and set
\[
\widehat\a^{\mathbf{e}}:=\Gamma\bigl(\wedge^\bullet(\widehat B^{\mathbf{e}})^*\bigr),
\quad
\Phi^{\mathbf{e}}:=\wedge^\bullet\bigl((\varphi^{\mathbf{e}})^*\bigr)^{-1}\colon\a\to\widehat\a^{\mathbf{e}},
\]
where functions are identified through $J$. Then $\Phi^{\mathbf{e}}$ is an isomorphism of graded commutative algebras of degree $0$ and
\begin{equation}\label{1SCGM8GSM}
\Phi^{\mathbf{e}}\circ d_{B_\ccc}=d_{\widehat B^{\mathbf{e}}}\circ\Phi^{\mathbf{e}}.
\end{equation}

\begin{theorem}\label{1SCW34U8K}
Let $(\ccc,L_N,L_M)$ be a regular datum and $\mathbf{e}$ a choice.
\begin{itemize}
\item[(1)] $\Delta$ is a pre-Poisson submanifold of $(U,\widehat\pi_0^{\mathbf{e}})$, with Lie algebroid $\widehat B^{\mathbf{e}}$.
\item[(2)] For every choice $\tau=(\Sigma,\varsigma)$ as in \Cref{1SCMMB63H} for $\Delta\subset(U,\widehat\pi_0^{\mathbf{e}})$, the maps
\[
\ell_k^{(\mathbf{e},\tau)}:=(\Phi^{\mathbf{e}})^{-1}\circ\ell_k^{\widehat\a^{\mathbf{e}},\Sigma,\varsigma}\circ(\Phi^{\mathbf{e}})^{\otimes k},
\quad k\ge1,
\quad \ell_0^{(\mathbf{e},\tau)}:=0,
\]
form a flat $P_\infty$-structure on $\a$ with $\ell_1^{(\mathbf{e},\tau)}=d_{B_\ccc}$.
\item[(3)] The $L_\infty$-isomorphism class of $\{\ell_k^{(\mathbf{e},\tau)}\}$ does not depend on $(\mathbf{e},\tau)$.
\end{itemize}
\end{theorem}

\begin{proof}
(1) By \Cref{1SCL3WW9S}, $\varphi^{\mathbf{e}}$ is an isomorphism onto $\widehat B^{\mathbf{e}}$, so $\widehat B^{\mathbf{e}}$ has constant rank. By \Cref{1SCPWPNPJ}, $\widehat B^{\mathbf{e}}=T\Delta^\circ\cap(\widehat\pi_0^{\mathbf{e}})^{-1}(T\Delta)$. By \eqref{1SC2QNAK4} for $\Delta\subset(U,\widehat\pi_0^{\mathbf{e}})$, the rank of $T\Delta+\widehat\pi_0^{\mathbf{e}}(T\Delta^\circ)$ equals $\dim U-\rank\widehat B^{\mathbf{e}}$, hence is constant, so $\Delta$ is pre-Poisson. By \Cref{1SCPWPNPJ}, the Lie algebroid structure of $\widehat B^{\mathbf{e}}$ induced by $\widehat\pi_0^{\mathbf{e}}$ is that of \Cref{1SCL3WW9S}.

(2) By (1), \Cref{1SCMMB63H} gives a flat $P_\infty$-structure on $\widehat\a^{\mathbf{e}}$ with $\ell_1=d_{\widehat B^{\mathbf{e}}}$. \Cref{1SCU92AQH} applied to $\Phi^{\mathbf{e}}$, with \eqref{1SCGM8GSM}, gives (2).

(3) Let $(\mathbf{e},\tau)$ and $(\mathbf{e}',\tau')$ be two choices, and denote by primes the objects of $\mathbf{e}'$. By \Cref{1SC454C3Y} there is $\Theta$ with $\Theta|_\Delta=\id_\Delta$ and $\Theta_*\widehat\pi_0=\widehat\pi_0'$, and
$\Theta_B:=(d\Theta^{-1})^*|_{\widehat B}\colon\widehat B\to\widehat B'$
is an isomorphism of Lie algebroids over $\id_\Delta$. Set $\widehat\Theta:=\wedge^\bullet(\Theta_B^*)^{-1}$.

Let $\tau=(\Sigma,\varsigma)$ and let $\tau^\Theta:=(\Sigma',\varsigma')$ be its transport by $\Theta$ as in \Cref{1SCEXNP25}. That lemma gives
\[
\widehat\Theta\circ\ell_k^{\widehat\a,\Sigma,\varsigma}
=\ell_k^{\widehat\a',\tau^\Theta}\circ\widehat\Theta^{\otimes k},
\quad k\ge1.
\]
Set $\Psi:=(\Phi')^{-1}\circ\widehat\Theta\circ\Phi$, an algebra automorphism of $\a$ of degree $0$. By (2),
\[
\Psi\circ\ell_k^{(\mathbf{e},\tau)}
=(\Phi')^{-1}\widehat\Theta\,\ell_k^{\widehat\a,\Sigma,\varsigma}\Phi^{\otimes k}
=(\Phi')^{-1}\ell_k^{\widehat\a',\tau^\Theta}\widehat\Theta^{\otimes k}\Phi^{\otimes k}
=\ell_k^{(\mathbf{e}',\tau^\Theta)}\circ\Psi^{\otimes k}
\]
for $k\ge1$, and both structures are flat, so $\Psi$ is a strict isomorphism from $\{\ell_k^{(\mathbf{e},\tau)}\}$ to $\{\ell_k^{(\mathbf{e}',\tau^\Theta)}\}$.

It remains to compare $\tau^\Theta$ with $\tau'$ for $\mathbf{e}'$. By \Cref{1SC4NFQ6C} applied to $\Delta\subset(U',\widehat\pi_0')$, there is an $L_\infty$-isomorphism $\widehat\varrho$ from $\{\ell_k^{\widehat\a',\tau^\Theta}\}$ to $\{\ell_k^{\widehat\a',\tau'}\}$. By \Cref{1SCU92AQH}(3), applied to $\Phi'$ on both sides, the maps $\varrho_k:=(\Phi')^{-1}\circ\widehat\varrho_k\circ(\Phi')^{\otimes k}$ form an $L_\infty$-isomorphism from $\{\ell_k^{(\mathbf{e}',\tau^\Theta)}\}$ to $\{\ell_k^{(\mathbf{e}',\tau')}\}$. Since $\{\ell_k^{(\mathbf{e},\tau)}\}$ is the transport of $\{\ell_k^{(\mathbf{e}',\tau^\Theta)}\}$ along $\Psi$, \Cref{1SCU92AQH}(3), applied to $\Psi$ and $\id$, shows that the maps $\varrho_k\circ\Psi^{\otimes k}$ form an $L_\infty$-isomorphism from $\{\ell_k^{(\mathbf{e},\tau)}\}$ to $\{\ell_k^{(\mathbf{e}',\tau')}\}$.
\end{proof}

We now specialize to a Dirac structure on a submanifold of a Poisson manifold. Let $(M,\pi)$ be a Poisson manifold, $\iota\colon N\hookrightarrow M$ an embedded submanifold and $L_N$ a Dirac structure on $N$. Set
\begin{equation}\label{1SCCTBG6Z}
B:=\bigl\{\xi\in T^*M|_N\ :\ \pi(\xi)\in TN\ \text{ and }\
(\pi(\xi), \iota^*\xi)\in L_N\bigr\}.
\end{equation}
The triple $(\iota,L_N,\Gamma_\pi)$, with $\phi=\psi=0$, is of the type considered in \Cref{1SCHTFL8N}.

\begin{corollary}\label{1SC3AQRTR}
Assume that $B$ and $E_N=L_N\cap TN$ have constant rank. Then $(\iota,L_N,\Gamma_\pi)$ is a regular datum, and
\[
\Psi\colon B_\iota\to B,\quad \Psi\bigl((v,\alpha),(w,\beta)\bigr)=\beta,
\]
is an isomorphism of Lie algebroids over $\id_N$ onto the Lie subalgebroid $B$ of $T^*_\pi M$. In particular, transporting the structures of \Cref{1SCW34U8K} along $\wedge^\bullet\Psi^*$ gives a flat $P_\infty$-structure on $\Gamma(\wedge^\bullet B^*)$ with $\ell_1=d_B$, unique up to $L_\infty$-isomorphism.
\end{corollary}

\begin{proof}
Let $\bigl((v,\alpha),(w,\beta)\bigr)\in B_\iota$. Since $(w,\beta)\in\Gamma_\pi$, we have $w=\pi(\beta)$, so $\pi(\beta)=\iota_*v\in TN$ and $(\pi(\beta),\iota^*\beta)=(v,\alpha)\in L_N$. Hence $\Psi$ is a bijection onto $B$, with inverse $\xi\mapsto\bigl((\pi(\xi),\iota^*\xi),(\pi(\xi),\xi)\bigr)$. Since $\Gamma_\pi\cap TM=0$, $(\iota,L_N,\Gamma_\pi)$ is a regular datum. As for $\ppp_N$, the map $\Psi$ is the restriction of the projection $L_N\times\Gamma_\pi\to\Gamma_\pi\cong T^*_\pi M$, which is a morphism of Lie algebroids over $\iota$. This proves the first assertion. The last one follows from \Cref{1SCW34U8K,1SCU92AQH}.
\end{proof}

\begin{corollary}\label{1SCUASLVL}
Let $L_N$ be the graph of a closed $2$-form $\omega$ with $\ker\omega$ of constant rank. Then
\[
B=\bigl\{\xi\in T^*M|_N\ :\ \pi(\xi)\in TN\ \text{ and }\ \iota_{\pi(\xi)}\omega=\iota^*\xi\bigr\},
\]
which corresponds to the subbundle $A_{N,\omega}$ of \textnormal{\cite[Definition 2.1]{bualibanu2026reduction}}. If $(N,\omega)$ is a reduction level, that is, if $B$ has constant rank, then \Cref{1SC3AQRTR} applies.
\end{corollary}

\begin{proof}
The description of $B$ is \eqref{1SCCTBG6Z} with $L_N=\Gamma_\omega$. Moreover $E_N=\Gamma_\omega\cap TN=\ker\omega$ has constant rank, so the hypotheses of \Cref{1SC3AQRTR} hold as soon as $B$ has constant rank.
\end{proof}

\subsection{Comparison with the Poisson algebra of invariant functions}\label{1SCFRV7MU}
Let $\ccc\colon(\C\tto N,L_N)\to(\G\tto M,\omega,\phi)$ be a $1$-shifted coisotropic structure with induced base Dirac structure $L_M$. Then $L_N$ is $\ccc^*\phi$-twisted and $B_\ccc$ is the fibre product $L_N\times_\ccc L_M$ of \Cref{1SCYF566F}, so $(\ccc,L_N,L_M)$ is of the type considered in \Cref{1SCHTFL8N}. Assume that $E_M$ and $E_N$ have constant rank, so that $(\ccc,L_N,L_M)$ is a regular datum. Since $\a$ is flat, $\ell_2$ descends to a Poisson bracket on $H^0(\a,\ell_1)$. We show that this bracket is the bracket \eqref{1SC2RZH94} on invariant functions.

\begin{proposition}\label{1SC64QN3W}
Suppose in addition that $\C$ is source-connected. Then $H^0(\a,\ell_1)=C^\infty(N)^\C$, and for every $P_\infty$-structure $\{\ell_k\}$ on $\a$ given by \Cref{1SCW34U8K}, the bracket induced by $\ell_2$ on $H^0(\a,\ell_1)$ is the bracket \eqref{1SC2RZH94} of \Cref{1SCYF566F}.
\end{proposition}

\begin{proof}
Since $\ell_1=d_{B_\ccc}$, $H^0(\a,\ell_1)$ consists of the functions annihilated by $\rho_B(B_\ccc)$. By nondegeneracy $\rho_B(B_\ccc)=\im\rho_\C$, and since $\C$ is source-connected the functions annihilated by $\im\rho_\C$ are the $\C$-invariant ones. So $H^0(\a,\ell_1)=C^\infty(N)^\C$.

Fix a choice $(\mathbf{e},\tau)$ as in \Cref{1SCW34U8K} and let $f,g\in C^\infty(N)^\C$. In degree $0$ the map $\Phi$ is the transport of functions by $J$, so $\Phi(f)=f\circ J^{-1}$. Since $\Delta$ is pre-Poisson in $(U,\widehat\pi_0)$, \Cref{1SCJHB7Z7}(1) provides admissible extensions $F,G$ of $\Phi(f),\Phi(g)$ near $\Delta$ and gives $\ell_2(f,g)=dG(\widehat\pi_0(dF))\circ J$.

Fix $n\in N$, set $q:=J(n)$, and write $dF_q=(\xi,\eta)$ with $\xi$ on the $S_M$ factor and $\eta$ on the $S_N$ factor. Since $F$ is admissible, $dF_q$ annihilates $\widehat\pi_0(T_q\Delta^\circ)$, so by antisymmetry $\widehat\pi_0(dF_q)\in(T_q\Delta^\circ)^\circ=T_q\Delta=J_*(T_nN)$. By \Cref{1SCPWPNPJ}, $\widehat\pi_0=\widehat\pi=\pi_M\oplus(-\pi_N)$ at $q$, and $J=(j_M\circ\ccc,j_N)$. Hence there is $v\in T_nN$, unique since $j_{N*}$ is injective, with $\widehat\pi_0(dF_q)=J_*v$, that is
$\pi_M\xi=j_{M*}\ccc_*v$ and $\pi_N(-\eta)=j_{N*}v$.

Set $\beta:=j_M^*\xi$. By \Cref{1SCCGC8W3}(2), $j_M^*\Gamma_{\pi_M}=L_M$ and $j_N^*\Gamma_{\pi_N}=L_N$, so the last two identities give $(\ccc_*v,\beta)\in L_M$ and $(v,-j_N^*\eta)\in L_N$. Differentiating $F\circ J=f$ gives $df=\ccc^*\beta+j_N^*\eta$, hence $(v,\ccc^*\beta-df)\in L_N$. Multiplying by $-1$, the pair $v_f:=-v$, $\beta_f:=-\beta$ satisfies \eqref{1SCJCKYC8}.

Therefore $\widehat\pi_0(dF_q)=-J_*v_f$, and $J^*dG=dg$ gives $\ell_2(f,g)(n)=-dg(v_f)$. By \Cref{1SCDG4T4W} this is $df(v_g)=\{f,g\}(n)$.
\end{proof}

Together with Theorem \ref{1SCW34U8K}, this proves Theorem \ref{djsx2lmg} from the introduction.

\begin{remark}
By \Cref{1SC5ZDRV2,1SC64QN3W}, coisotropic Morita equivalent structures with source-connected groupoids have isomorphic $H^0(\a,\ell_1)$. Their $P_\infty$-algebras, however, need not be $L_\infty$-quasi-isomorphic. Let $\G=\{*\}$ and let $\operatorname{Pair}(S^1)=S^1\times S^1\tto S^1$ be the pair groupoid. Then $L_1:=0$ and $L_2:=TS^1\oplus0$ are $1$-shifted coisotropic structures on $\ccc_1\colon\{*\}\to\G$ and $\ccc_2\colon\operatorname{Pair}(S^1)\to\G$. Taking $\K:=\operatorname{Pair}(S^1)$, $\psi_1\colon\K\to\{*\}$, $\psi_2:=\id$ and $\L:=\{*\}$ in \eqref{1SCT3HPM7} gives a coisotropic Morita equivalence, since $\psi_1^*L_1=TS^1\oplus0=\psi_2^*L_2$. Since $B_{\ccc_1}=0$ and $B_{\ccc_2}\cong TS^1$,
\[
(\a_1,\ell_1)=(\R,0),\qquad (\a_2,\ell_1)=(\Omega^\bullet(S^1),d_{dR}).
\]
Hence $H^1(\a_1,\ell_1)=0\neq\R=H^1(\a_2,\ell_1)$, so $\a_1$ and $\a_2$ are not $L_\infty$-quasi-isomorphic.
\end{remark}

\section{Quantization of 1-shifted coisotropics}\label{1SCC6G85S}
We now quantize the $P_\infty$-structures of \S\ref{1SCYR8NZ4}, proving Theorem \ref{d4fmvhyq}. We first treat a pre-Poisson submanifold $N\subset(M,\pi)$, with $P_\infty$-algebra $\a=C^\infty(\wedge^\bullet L^*)$. We apply the relative formality theorem of Cattaneo and Felder to the coisotropic submanifold $N\subset S$ and transport the result along the graded algebra isomorphism $\Phi_S$ of \eqref{1SCVX2YHQ}. After proving independence of the quantization datum, we apply this construction to the local models of \S\ref{1SCHTFL8N} to obtain the quantization of general regular infinitesimal $1$-shifted coisotropics.

\subsection{Preliminaries on curved \texorpdfstring{$A_\infty$}{A-infinity}-algebras}\label{1SCS34X7R}
We recall the definitions and basic properties used below, following \cite{keller2001introduction,lada1995strongly} for $A_\infty$-algebras and \cite[\S\S 3.1, 4.4 and Appendix A]{Cattaneo2007} for the curved setting and its description by multidifferential operators.

A \defn{curved $A_\infty$-algebra} over a commutative ring $R$ is a graded $R$-module $\a$ with $R$-linear maps $\mu_k\colon\a^{\otimes k}\to\a$ of degree $2-k$, $k\ge0$, satisfying, for every $k\ge0$ and all homogeneous $a_1,\dots,a_k$,
\begin{equation}\label{1SC4Z7X3C}
\sum_{\substack{p+q+t=k\\ q\ge0}}(-1)^{p+q+pq+q(|a_1|+\cdots+|a_p|)}\,
\mu_{p+1+t}\bigl(a_1,\dots,a_p,\mu_q(a_{p+1},\dots,a_{p+q}),a_{p+q+1},\dots,a_k\bigr)=0 .
\end{equation}
The element $\mu_0\in\a^2$ is the \defn{curvature}. We call $\a$ \defn{flat} if $\mu_0=0$, and then simply an $A_\infty$-algebra. In the flat case $\mu_1$ is a differential and $\mu_2$ induces an associative product on cohomology. Relation \eqref{1SC4Z7X3C} is the Maurer--Cartan equation $[\mathsf m+\varpi,\mathsf m+\varpi]_G=0$ for the Gerstenhaber bracket of \cite[Eq.~(6)]{Cattaneo2007}, see \Cref{1SCZPJUKW}; in arity $1$ it reads $\mu_1^2(a)-\mu_2(\mu_0,a)+\mu_2(a,\mu_0)=0$, and in arity $2$ it is the Leibniz rule $\mu_1\mu_2(a,b)=\mu_2(\mu_1a,b)+(-1)^{|a|}\mu_2(a,\mu_1b)$. It differs from the convention of \cite{keller2001introduction} by the rescaling $\mu_k\mapsto(-1)^{k(k-1)/2}\mu_k$, which is the difference between the two desuspensions $(s^{\otimes k})^{-1}$ and $(s^{-1})^{\otimes k}$ under the Koszul rule.

Curved $A_\infty$-structures deforming a given product can be described as Maurer--Cartan elements, as follows. Let $\a$ be a graded algebra with product $\mmm$, and assume that $\a$ is the space of sections of a graded vector bundle, so that multidifferential operators make sense. For $q\ge0$ and $p\in\mathbb Z$, let $\D^{p,q}(\a)$ be the space of multidifferential operators $\a^{\otimes q}\to \a$ of internal degree $p$; such an operator has total degree $p+q-1$. Following \cite[\S4.4]{Cattaneo2007}, set $\widehat\D^{\,n}(\a):=\prod_{p+q-1=n}\D^{p,q}(\a)$ and
\begin{equation}\label{1SCNASR4D}
\widehat\D(\a)
:= \bigoplus_{n\in\mathbb Z}\widehat\D^{\,n}(\a).
\end{equation}
In each total degree we take the product over all arities; the sum over total degrees is direct. With the Gerstenhaber bracket of \cite[Eq.(6)]{Cattaneo2007} and differential $d=[\mmm,\cdot]$, $\widehat\D(\a)$ is a differential graded Lie algebra. Equip $\widehat\D(\a)\llbracket\epsilon\rrbracket$ with the complete $\epsilon$-adic filtration. If $\varpi=\sum_{k\ge0}\varpi_k\in\epsilon\,\widehat\D(\a)^1\llbracket\epsilon\rrbracket$, where $\varpi_k$ has arity $k$, then the Maurer--Cartan equation for $\varpi$ is equivalent to $[\mmm+\varpi,\mmm+\varpi]=0$, hence to the curved $A_\infty$-relations for
\begin{equation}\label{1SC32WQT7}
\mu_0:=\varpi_0,\quad
\mu_1:=\varpi_1,\quad
\mu_2:=\mmm+\varpi_2,\quad
\mu_k:=\varpi_k\quad(k\ge3).
\end{equation}
The Maurer--Cartan elements of $\epsilon\,\widehat\D(\a)\llbracket\epsilon\rrbracket$ are thus the curved $A_\infty$-structures on $\a\llbracket\epsilon\rrbracket$ reducing modulo $\epsilon$ to the original product.

\begin{lemma}\label{1SC2EHE94}
Let $\a,\b$ be graded $R$-modules and $\Phi\colon \a\to \b$ an isomorphism of degree $0$. Let $(\b,\{\mu_k^\b\}_{k\ge0})$ be a curved $A_\infty$-algebra, and set $\mu_k^\a:=\Phi^{-1}\circ\mu_k^\b\circ \Phi^{\otimes k}$ for $k\ge1$, and $\mu_0^\a:=\Phi^{-1}(\mu_0^\b)$. Then:
\begin{itemize}
\item[(1)] $(\a,\{\mu_k^\a\})$ is a curved $A_\infty$-algebra, and $\Phi$ is a strict isomorphism from $(\a,\{\mu_k^\a\})$ to $(\b,\{\mu_k^\b\})$.
\item[(2)] $(\a,\{\mu_k^\a\})$ is flat if and only if $(\b,\{\mu_k^\b\})$ is.
\item[(3)] If $R=\R\llbracket\epsilon\rrbracket$, $\a=\a_0\llbracket\epsilon\rrbracket$ and $\b=\b_0\llbracket\epsilon\rrbracket$ for graded algebras $\a_0,\b_0$, and $\Phi$ is the $\R\llbracket\epsilon\rrbracket$-linear extension of an algebra isomorphism $\a_0\to\b_0$, then $\mu_2^\a$ reduces modulo $\epsilon$ to the product of $\a_0$ if and only if $\mu_2^\b$ reduces to that of $\b_0$.
\item[(4)] If $\a_0,\b_0$ are the spaces of sections of graded vector bundles over a common manifold $M$, and $\Phi$ is induced by an isomorphism of vector bundles covering a diffeomorphism of $M$, then $\{\mu_k^\a\}$ is multidifferential as soon as $\{\mu_k^\b\}$ is.
\end{itemize}
\end{lemma}

\begin{proof}
Since $\Phi$ preserves degrees, it preserves Koszul signs, and conjugation by $\Phi$ carries the relations \eqref{1SC4Z7X3C} for $\{\mu_k^\b\}$ termwise into those for $\{\mu_k^\a\}$. This gives (1) and (2). Parts (3) and (4) follow from the definition of $\mu_k^\a$.
\end{proof}

Here and throughout, $\Alt_k:=\sum_{\sigma\in\mathfrak S_k}\chi(\sigma;\cdot)\,\sigma$ is the \emph{unnormalized} graded alternation, so that $(\mu\circ\Alt_k)(a_1,\dots,a_k)=\sum_{\sigma}\chi(\sigma;a)\,\mu(a_{\sigma(1)},\dots,a_{\sigma(k)})$. This is the operator of \cite[Proposition 3.1]{Cattaneo2007}. In the flat case, the antisymmetrization $\{\mu_k\circ\Alt_k\}$ of an $A_\infty$-structure is an $L_\infty$-structure \cite[Theorem 3.1]{lada1995strongly}.

\medskip
The gauge arguments of \Cref{1SCQPGX26} use the coalgebraic encoding of the same structures. Set $V:=\a\llbracket\epsilon\rrbracket[1]$. All tensor products are over $\R\llbracket\epsilon\rrbracket$; we write $\widehat\otimes$ for the $\epsilon$-adically completed tensor product and set $V^{\widehat\otimes0}:=\R\llbracket\epsilon\rrbracket$.

\begin{definition}\label{1SCB2252C}
Let $T^c(V):=\bigoplus_{n\ge0}V^{\widehat\otimes n}$ be the conilpotent tensor coalgebra. Its \defn{$\epsilon$-adic completion} is $\widehat T(V):=\varprojlim_N T^c(V)/\epsilon^N T^c(V)$, equivalently
\[
\widehat T(V)
\cong
\left\{(x_n)_{n\ge0}\in\prod_{n\ge0}V^{\widehat\otimes n}
\;|\;
\forall N,\ x_n\in\epsilon^N V^{\widehat\otimes n}\text{ for }n\gg0\right\}.
\]
The deconcatenation coproduct extends uniquely by $\epsilon$-adic continuity to $\widehat T(V)$.

A \defn{coderivation} of degree $p$ is a continuous map $X\colon\widehat T(V)\longrightarrow\widehat T(V)$ of degree $p$ with
\begin{equation}\label{1SCAYFZTX}
\Delta\circ X=(X\widehat\otimes\id+\id\widehat\otimes X)\circ\Delta.
\end{equation}
Its \defn{corestriction} is the family $X_k:=\ppp_V\circ X|_{V^{\widehat\otimes k}}\colon V^{\widehat\otimes k}\longrightarrow V$, $k\ge0$.
\end{definition}

\begin{lemma}\label{1SCMFJKDK}
A continuous coderivation is determined by its corestriction. Conversely, every family of continuous $\R\llbracket\epsilon\rrbracket$-linear maps $X_k\colon V^{\widehat\otimes k}\longrightarrow V$, $k\ge0$, is the corestriction of a unique continuous coderivation of $\widehat T(V)$, given on $V^{\widehat\otimes n}$ by
\begin{equation}\label{1SCH8SX6M}
X(v_1\otimes\cdots\otimes v_n)
=
\sum_{k=0}^{n}\sum_{i=0}^{n-k}\pm\,
v_1\otimes\cdots\otimes v_i
\otimes X_k(v_{i+1},\dots,v_{i+k})
\otimes v_{i+k+1}\otimes\cdots\otimes v_n,
\end{equation}
with the usual Koszul signs, where $X_0()$ is inserted in every position, including at either end.
\end{lemma}

\begin{proof}
For the uncompleted tensor coalgebra this is \cite[Lemma A.1]{Cattaneo2007}. The formula \eqref{1SCH8SX6M} is $\R\llbracket\epsilon\rrbracket$-linear, hence $\epsilon$-adically continuous, so both statements extend uniquely to $\widehat T(V)$.
\end{proof}

\begin{proposition}\label{1SCZPJUKW}
Under the identification of  \Cref{1SCMFJKDK}, the operations \eqref{1SC32WQT7} are encoded by the degree-$1$ coderivation $Q$ whose corestriction is determined by
\begin{equation}\label{1SCVK82U3}
s\circ Q_k=\mu_k\circ s^{\otimes k}\quad(k\ge0),
\end{equation}
where $s\colon \a\llbracket\epsilon\rrbracket[1]\to \a\llbracket\epsilon\rrbracket$ is the tautological map of degree $1$ \textnormal{\cite[\S3.1]{Cattaneo2007}}, and $s^{\otimes k}$ is applied with the Koszul rule, so that $s^{\otimes k}(v_1\otimes\cdots\otimes v_k)=(-1)^{\sum_{i}(k-i)|v_i|}\,sv_1\otimes\cdots\otimes sv_k$. Since the Gerstenhaber bracket of \textnormal{\cite[Eq.(6)]{Cattaneo2007}} is the transport of the commutator of coderivations, $Q$ corresponds to
\begin{equation}\label{1SCKPRT9G}
Q\ \longleftrightarrow\ \mmm+\varpi,
\end{equation}
and the Maurer--Cartan equation for $\varpi$ is equivalent to $Q^2=0$.
\end{proposition}

\begin{proof}
This is \eqref{1SC32WQT7} translated through   \Cref{1SCMFJKDK} and \cite[\S3.1 and \S A.3]{Cattaneo2007}; since $Q$ has odd degree, $[Q,Q]=2Q^2$. Expanding $Q^2=0$ with \eqref{1SCH8SX6M} and \eqref{1SCVK82U3} gives exactly \eqref{1SC4Z7X3C}; the sign $(-1)^{p+q+pq}$ arises from the Koszul rule for $s^{\otimes k}$ and for the coderivation extension.
\end{proof}

\begin{definition}\label{1SCAZ35B7}
Let $\a,\a'$ be graded modules, set $V=\a\llbracket\epsilon\rrbracket[1]$ and $V'=\a'\llbracket\epsilon\rrbracket[1]$, and let $Q,Q'$ be degree-$1$ coderivations with square zero. A \defn{morphism} of curved $A_\infty$-algebras is a continuous coalgebra morphism $\varrho\colon\widehat T(V)\to\widehat T(V')$ of degree $0$ with
\begin{equation}\label{1SC7XE5FK}
Q'\circ \varrho=\varrho\circ Q.
\end{equation}
Its components are $\varrho_k:=\ppp_{V'}\circ \varrho|_{V^{\widehat\otimes k}}$, $k\ge0$. After desuspension these are the usual maps $\a^{\otimes k}\to \a'$ of degree $1-k$. We require $\varrho_0\in\epsilon V'$, which makes the standard tensor-coalgebra formula $\epsilon$-adically convergent. We write $\varrho\colon(\a\llbracket\epsilon\rrbracket,\mu)\lto(\a'\llbracket\epsilon\rrbracket,\mu')$. The morphism is \defn{strict} if $\varrho_0=0$ and $\varrho_k=0$ for $k\ge2$, and an \defn{isomorphism} if $\varrho_1$ is bijective.
\end{definition}

A \defn{change of base point} is the special case $\varrho_1=\id$, $\varrho_k=0$ for $k\ge2$, and $\varrho_0$ given by an element whose desuspension, still denoted $a$, lies in $\epsilon \a\llbracket\epsilon\rrbracket^1$. In terms of the unsuspended operations it gives
\begin{equation}\label{1SCT49PYH}
\widetilde\mu_0=\sum_{k\ge0}\mu_k(a,\dots,a),
\end{equation}
and similarly for the higher operations, by inserting $a$ in all possible slots; the series converge $\epsilon$-adically.

For a degree-$0$ coderivation $X$, whenever $e^X:=\sum_{j\ge0}\frac{X^j}{j!}$ is well defined on the completed tensor coalgebra, it is a coalgebra automorphism with inverse $e^{-X}$. The next proposition converts gauge equivalences of Maurer--Cartan elements into isomorphisms of curved $A_\infty$-algebras; it is used three times in \Cref{1SCQPGX26}.

\begin{proposition}\label{1SCVEVYP9}
Let $\widehat\D(\b)\llbracket\epsilon\rrbracket$ be equipped with a complete descending filtration $\FFF^\bullet$ compatible with $d=[\mmm,\cdot]$ and the bracket. Let $\alpha_0,\alpha_1\in\FFF^1$ be Maurer--Cartan elements in $\epsilon\,\widehat\D(\b)\llbracket\epsilon\rrbracket$, gauge equivalent through a degree-$0$ element $X\in\FFF^1$ whose arity-zero component lies in $\epsilon\,\b\llbracket\epsilon\rrbracket$, and assume that $e^{\widehat X}$ is well defined on the completed tensor coalgebra, where $\widehat X$ is the coderivation with components $X$. Then $e^{\widehat X}$ is an isomorphism between the corresponding curved $A_\infty$-structures.
\end{proposition}

\begin{proof}
Since $X$ has degree $0$, $dX=[\mmm,X]=-\ad_X\mmm$, hence $\frac{e^{\ad_X}-1}{\ad_X}(dX)=-(e^{\ad_X}-1)\mmm$, and \eqref{1SCFR6FY2} gives $\alpha_1=e^{\ad_X}\alpha_0+e^{\ad_X}\mmm-\mmm$, that is,
\[
\mmm+\alpha_1=e^{\ad_X}(\mmm+\alpha_0).
\]
Under  \Cref{1SCMFJKDK} the Gerstenhaber bracket is the commutator of coderivations, so $e^{\ad_X}$ corresponds to $Q\mapsto e^{\widehat X}Qe^{-\widehat X}$. Hence $e^{\widehat X}(\mmm+\alpha_0)=(\mmm+\alpha_1)e^{\widehat X}$, which is \eqref{1SC7XE5FK}. The component $\varrho_0=X_0+\tfrac12X_1(X_0)+\cdots$ of $e^{\widehat X}$ lies in $\epsilon V$ because $X_0$ does, and $\varrho_1$ is invertible modulo $\epsilon$, with inverse the linear component of $e^{-\widehat X}$, hence, by $\epsilon$-adic completeness, bijective. So $e^{\widehat X}$ is a curved $A_\infty$-isomorphism. 
\end{proof}

\begin{lemma}\label{1SCPN48SN}
Let $\Phi\colon\a\to\b$ and $\Phi'\colon\a'\to\b'$ be isomorphisms of graded algebras of degree $0$, and let $\varrho\colon(\b\llbracket\epsilon\rrbracket,\mu)\lto(\b'\llbracket\epsilon\rrbracket,\mu')$ be a morphism. Equip $\a\llbracket\epsilon\rrbracket$ and $\a'\llbracket\epsilon\rrbracket$ with the transported structures of \Cref{1SC2EHE94}. Then $\varrho^\a_k:=\Phi'^{-1}\circ \varrho_k\circ\Phi^{\widehat\otimes k}$ for $k\ge1$, and $\varrho^\a_0:=\Phi'^{-1}(\varrho_0)$, define a morphism between the transported structures, which is an isomorphism, respectively strict, if and only if $\varrho$ is.
\end{lemma}

\begin{proof}
The maps $\Phi[1]$ and $\Phi'[1]$ induce continuous coalgebra isomorphisms $T(\Phi):=\widehat T(\Phi[1])\colon\widehat T(\a\llbracket\epsilon\rrbracket[1])\longrightarrow\widehat T(\b\llbracket\epsilon\rrbracket[1])$ and $T(\Phi'):=\widehat T(\Phi'[1])\colon\widehat T(\a'\llbracket\epsilon\rrbracket[1])\longrightarrow\widehat T(\b'\llbracket\epsilon\rrbracket[1])$. By \Cref{1SC2EHE94}, the transported coderivations are $T(\Phi)^{-1}QT(\Phi)$ and $T(\Phi')^{-1}Q'T(\Phi')$. The continuous coalgebra morphism $T(\Phi')^{-1}\varrho\,T(\Phi)$ has components $\varrho^\a_k$, and it intertwines the transported coderivations:
\begin{align*}
T(\Phi')^{-1}Q'T(\Phi')\,T(\Phi')^{-1}\varrho\,T(\Phi)
&=T(\Phi')^{-1}Q'\varrho\,T(\Phi)\\
&=T(\Phi')^{-1}\varrho\,Q\,T(\Phi)\\
&=T(\Phi')^{-1}\varrho\,T(\Phi)\,T(\Phi)^{-1}QT(\Phi).
\end{align*}
Since $\Phi$ and $\Phi'$ are bijective, transport preserves bijectivity of the linear component and strictness.
\end{proof}

\subsection{The curved \texorpdfstring{$A_\infty$}{A-infinity}-structure}\label{1SC9UZXPX}
Let $N$ be a pre-Poisson submanifold of $(M,\pi)$, with stabilizer Lie algebroid $L:=TN^\circ\cap\pi^{-1}(TN)$ and $\a:=C^\infty(\wedge^\bullet L^*)$. Choose a Poisson transversal $S\subset M$ containing $N$ coisotropically, with induced Poisson bivector $\pi_S$. As in \Cref{1SCS4XYSN}, write $\nu:=TS|_N/TN$, $\b:=C^\infty(\wedge^\bullet\nu)$, and $\Phi:=\Phi_S\colon\a\to\b$ for the isomorphism \eqref{1SCVX2YHQ}. We quantize the $P_\infty$-structure on $\a$ by applying the Cattaneo--Felder relative formality theorem to $N\subset S$ and transporting the result along $\Phi$. Independence of choices is proved in \Cref{1SCQPGX26}.

Throughout, $\epsilon$ is a formal parameter and all maps are $\R\llbracket\epsilon\rrbracket$-linear unless stated otherwise. The graded Lie algebra $\g_S$ of formal multivector fields along $N$ carries the Schouten bracket and the zero differential, with degrees shifted so that bivector fields have degree $1$. Thus $[\pi_S,\pi_S]=0$ says that $\epsilon\,\pi_S$ is a Maurer--Cartan element of $\epsilon\,\g_S\llbracket\epsilon\rrbracket$. Finally, $\widehat\D(\cdot)$ is the completed differential graded Lie algebra of multidifferential operators \eqref{1SCNASR4D}.

A \defn{quantization datum} is a triple $\mathbf{d}:=(S,\psi,\nabla)$: a Poisson transversal $S\subset M$ containing $N$ coisotropically, a tubular neighbourhood $\psi$ of $N$ in $S$ which is \defn{normalized} (it restricts to the inclusion on $N$ and has normal derivative $\id_\nu$ along $N$), and a connection $\nabla$ on $\nu\to N$. We fix once and for all a torsion-free connection on $N$ and an associated formal exponential map.

We use the fibrewise Fourier transform of Cattaneo--Felder \cite[\S\S 4.2--4.4]{Cattaneo2007}, which identifies formal multivector fields along $N$ in $S$ with completed multiderivations of $\b$. We choose the sign convention for which $\mathfrak F(j(b))=b$, so that it agrees with the derived brackets in \eqref{1SCRZZPXA}. The Hochschild--Kostant--Rosenberg map is normalized by $1/k!$ in arity $k$, as in \cite[\S 4.5]{Cattaneo2007}, whereas $\Alt_k$ is unnormalized. We now restate the Cattaneo--Felder theorem in these conventions.

\begin{proposition}[Cattaneo--Felder]\label{1SCKTNBKS}
Let $N$ be a submanifold of a Poisson manifold $(M,\pi)$, and let $\b_N:=C^\infty(\wedge^\bullet\nu_N)$. Fix a tubular neighbourhood of $N$ in $M$ and a connection on $\nu_N$. Then:
\begin{itemize}
\item[(1)] there is an $L_\infty$-quasi-isomorphism $\U\colon\g_{N}\lto\widehat\D(\b_N)$, where $\g_{N}$ is the graded Lie algebra of multivector fields on the formal neighbourhood of $N$, and $\U_1$ is the fibrewise Fourier transform followed by the Hochschild--Kostant--Rosenberg map.
\item[(2)] $\tau:=\U_*(\epsilon\pi)$, defined by \eqref{1SCS3HGKV}, is a curved $A_\infty$-structure on $\b_N\llbracket\epsilon\rrbracket$ with multidifferential structure maps, whose operation $\mu_2$ reduces modulo $\epsilon$ to the product of $\b_N$, while $\mu_k\equiv0\pmod\epsilon$ for $k\neq2$.
\item[(3)] for every $k\ge0$ the operator $\mu_k\circ\Alt_k$ is divisible by $\epsilon$, and $\frac1\epsilon(\mu_k\circ\Alt_k)\bmod\epsilon=(-1)^{(k-1)(k-2)/2}\ell_k$, where $\ell_k$ is the derived-bracket $P_\infty$-structure in the conventions of \eqref{1SCRZZPXA}.
\item[(4)] $\ell_0=-P_N(\pi)$. In particular $\mu_0=\epsilon\,P_N(\pi)+O(\epsilon^2)$, and $\mu_0=O(\epsilon^2)$ if and only if $N$ is coisotropic in $M$.
\end{itemize}
\end{proposition}

\begin{proof}
(1) is \cite[Theorem 4.8]{Cattaneo2007}: the tubular neighbourhood identifies $\g_{N}$ with the multivector fields on the formal neighbourhood of the zero section of $\nu_N$, and the connection enters through the globalization of \cite[\S A.9]{Cattaneo2007}.

For (2), extend $\U$ $\R\llbracket\epsilon\rrbracket$-linearly. Since $\U_k(\epsilon^{a_1}X_1,\dots,\epsilon^{a_k}X_k)=\epsilon^{a_1+\cdots+a_k}\U_k(X_1,\dots,X_k)$, the morphism $\U$ is filtered for the $\epsilon$-adic filtration in the sense of \eqref{1SCLC9EC8}, and that filtration is complete. As $\epsilon\pi\in\epsilon\,\g_{N}^1\llbracket\epsilon\rrbracket$ and $[\pi,\pi]=0$, \Cref{1SCMF2KEY}(1) shows that $\tau$ is a Maurer--Cartan element of $\epsilon\,\widehat\D(\b_N)\llbracket\epsilon\rrbracket$. By \eqref{1SC32WQT7} it is a curved $A_\infty$-structure on $\b_N\llbracket\epsilon\rrbracket$ reducing modulo $\epsilon$ to the product. This is \cite[Theorem 3.2]{Cattaneo2007}.

For (3), the $k$-th summand of \eqref{1SCS3HGKV} at $\epsilon\pi$ is $\frac{\epsilon^k}{k!}\U_k(\pi,\dots,\pi)$, so $\tau=\epsilon\,\U_1(\pi)+O(\epsilon^2)$. Divisibility by $\epsilon$ of $\mu_k\circ\Alt_k$ holds because $\mu_k\equiv0$ for $k\neq2$ and because $\mu_2\equiv \mmm$ with $(\mmm\circ\Alt_2)(a,b)=\mmm(a,b)-(-1)^{|a||b|}\mmm(b,a)=0$, the algebra $\b_N$ being graded commutative. By (1), $\U_1=\varphi_{\mathrm{HKR}}\circ\mathfrak F$. To identify its antisymmetrization, let $Q^{(1)}$ be the degree-$1$ coderivation of $\widehat T(\b_N[1])$ corresponding to $\U_1(\pi)$ under \Cref{1SCMFJKDK}, so that $(\U_1(\pi))_k=s\circ Q^{(1)}_k\circ(s^{\otimes k})^{-1}$ by \eqref{1SCVK82U3}. Write $\operatorname{dec}_k=(-1)^{\sum_i(k-i)|b_i|}$ for the sign of \eqref{EQDECALAGE}, which is the Koszul sign of $(s^{-1})^{\otimes k}$; under the Koszul rule $(s^{\otimes k})^{-1}=(-1)^{k(k-1)/2}(s^{-1})^{\otimes k}$. Since $\operatorname{dec}_k\circ\Alt_k=\Sym_k\circ\operatorname{dec}_k$, where $\Sym_k$ is graded symmetrization for the shifted degrees, and since $\ell_k=(-1)^{k-1}\operatorname{dec}_k\{\cdots\}$ by \eqref{1SCRZZPXA}, the identity $(\U_1(\pi))_k\circ\Alt_k=(-1)^{(k-1)(k-2)/2}\ell_k$, in which $(-1)^{(k-1)(k-2)/2}=(-1)^{k(k-1)/2}(-1)^{k-1}$, is equivalent to
\[
s\circ Q^{(1)}_k\circ\Sym_k\circ(s^{-1})^{\otimes k}=\{\,\cdot\,,\dots,\cdot\,\}
=\operatorname{pr}_0\bigl[\cdots[\mathfrak F(\pi),\,\cdot\,],\dots,\,\cdot\,\bigr],
\]
a statement in which no desuspension sign appears. Here $\operatorname{pr}_0$ denotes projection onto multivectors of arity zero on $\nu_N^*[1]$. With our Fourier convention, $\mathfrak F(j(b))=b$ and $\operatorname{pr}_0\mathfrak F=P_N$, and $\mathfrak F$ preserves the Schouten bracket, so the right-hand side is $P_N[\cdots[\pi,j(b_1)],\dots,j(b_k)]$. The left-hand side is the evaluation formula for the HKR map \cite[Theorem 4.6]{Cattaneo2007}, whose symmetrization is taken with the Koszul signs of the shifted degrees: graded symmetrization cancels its factor $1/k!$, and the identity follows on products of derivations, hence on formal multivectors, arity by arity. For $k=0$ it reads $(\U_1(\pi))_0=P_N(\pi)$. This proves (3) in every arity.

For (4), \eqref{1SCRZZPXA} gives $\ell_0=-P_N(\pi)$. The case $k=0$ of (3), where the sign is $-1$, gives $\mu_0=-\epsilon\ell_0+O(\epsilon^2)=\epsilon P_N(\pi)+O(\epsilon^2)$. Finally, $P_N(\pi)=0$ says that $\pi|_N$ lies in $\ker(\wedge^2\ppp_N)$, that is, $\pi(\nu_N^*)\subset TN$.
\end{proof}

Let $\mathbf{d}=(S,\psi,\nabla)$ be a quantization datum. Applying \Cref{1SCKTNBKS} to $(S,N)$ with the data $\psi$ and $\nabla$ gives a curved $A_\infty$-structure $\mu^{\b,\mathbf{d}}$ on $\b\llbracket\epsilon\rrbracket$. Set
\[
\mu_k^{\mathbf{d}}:=\Phi^{-1}\circ\mu_k^{\b,\mathbf{d}}\circ\Phi^{\otimes k}\ \ (k\ge1),
\quad
\mu_0^{\mathbf{d}}:=\Phi^{-1}\bigl(\mu_0^{\b,\mathbf{d}}\bigr).
\]

\begin{theorem}\label{1SCL9J8YC}
Let $\mathbf{d}=(S,\psi,\nabla)$ be a quantization datum. Then $\{\mu_k^{\mathbf{d}}\}_{k\ge0}$ is an $\R\llbracket\epsilon\rrbracket$-linear curved $A_\infty$-structure on $\a\llbracket\epsilon\rrbracket$ with multidifferential structure maps, and:
\begin{itemize}
\item[(1)] modulo $\epsilon$, $\mu_2^{\mathbf{d}}$ is the product of $\a$ and $\mu_k^{\mathbf{d}}=0$ for $k\neq2$;
\item[(2)] $\mu_0^{\mathbf{d}}=O(\epsilon^2)$;
\item[(3)]$\frac1\epsilon\bigl(\mu_k^{\mathbf{d}}\circ\Alt_k\bigr)\bmod\epsilon=(-1)^{(k-1)(k-2)/2}\ell_k^{\a,S,\psi}$ for every $k\ge0$. In particular, $\frac1\epsilon\mu_1^{\mathbf{d}}\bmod\epsilon=d_L$.
\end{itemize}
\end{theorem}

\begin{proof}
The map $\Phi$ is an isomorphism of graded algebras of degree $0$, so \Cref{1SC2EHE94} applies: $\{\mu_k^{\mathbf{d}}\}$ is a curved $A_\infty$-structure with multidifferential structure maps, and (1) follows from $\mu_2^{\b,\mathbf{d}}\equiv \mmm\pmod\epsilon$.

(2) is \Cref{1SCKTNBKS}(4) for $(S,N)$, where $N$ is coisotropic, together with the $\R\llbracket\epsilon\rrbracket$-linearity of $\Phi^{-1}$.

For (3), $\Phi$ preserves degrees, so $\chi(\sigma;\Phi a)=\chi(\sigma;a)$ for every $\sigma\in\mathfrak S_k$ and every homogeneous tuple $a$, whence
\begin{equation}\label{1SCNXARAB}
\Phi^{\otimes k}\circ\Alt_k=\Alt_k\circ\Phi^{\otimes k}
\quad\text{and}\quad
\mu_k^{\mathbf{d}}\circ\Alt_k=\Phi^{-1}\circ\bigl(\mu_k^{\b,\mathbf{d}}\circ\Alt_k\bigr)\circ\Phi^{\otimes k}.
\end{equation}
Since $\Phi$ is bijective and $\R\llbracket\epsilon\rrbracket$-linear, divisibility by $\epsilon$ transfers from $\b$ to $\a$. \Cref{1SCKTNBKS}(3) gives $\frac1\epsilon(\mu_k^{\b,\mathbf{d}}\circ\Alt_k)\bmod\epsilon=(-1)^{(k-1)(k-2)/2}\ell_k^{\b_S,\psi}$, the brackets of \Cref{1SCS4XYSN}. Conjugating by $\Phi$ and using $\ell_k^{\a,S,\psi}=\Phi^{-1}\circ\ell_k^{\b_S,\psi}\circ\Phi^{\otimes k}$ gives (3). The last assertion is the case $k=1$, where $\Alt_1=\id$, together with $\ell_1^{\a,S,\psi}=d_L$ from \Cref{1SCMMB63H}.
\end{proof}

\begin{remark}\label{1SCASEHBT}
For $k=2$, \Cref{1SCL9J8YC}(3) reads
\[
\ell_2^{\a,S,\psi}(a,b)
=\frac1\epsilon\Bigl(\mu_2^{\mathbf{d}}(a,b)-(-1)^{|a||b|}\mu_2^{\mathbf{d}}(b,a)\Bigr)\bmod\epsilon,
\]
the semiclassical limit of a graded commutator; on functions, by \eqref{1SCQHGY3K}, the commutator of $\mu_2^{\mathbf d}$ is $\epsilon\{f,g\}+O(\epsilon^2)$, as for a star product.
\end{remark}

Write $H^\bullet(L)$ for the cohomology of $(\a,d_L)$, that is, the Lie algebroid cohomology of $L$ in the sense of \eqref{1SCYQ3DDG}. By \Cref{1SCVT88UN}, $H^0(L)=C^\infty(N)^\HHH$. Since $r\colon L\to\nu^*$ is an isomorphism of Lie algebroids and $\Phi$ intertwines $d_L$ with $\ell_1^{\b}$ (\Cref{1SCMMB63H}), $\Phi$ is an isomorphism of complexes, so
\begin{equation}\label{1SCG5RSQW}
H^k(L)\cong H^k(\b,\ell_1^{\b})\quad\text{for all }k .
\end{equation}

\begin{theorem}\label{1SCJQ3RB7}
Let $\mathbf{d}=(S,\psi,\nabla)$ be a quantization datum. For $a\in\epsilon\,\a^1\llbracket\epsilon\rrbracket$, write $\mu^{\mathbf d,a}$ for the change of base point of $\{\mu_k^{\mathbf d}\}$ by $a$, in the sense of \eqref{1SCT49PYH}.
\begin{itemize}
\item[(1)] If $H^2(L)=0$, there is $a\in\epsilon\,\a^1\llbracket\epsilon\rrbracket$ such that $\mu^{\mathbf d,a}$ is flat. In particular $\mu_2^{\mathbf{d},a}$ induces an associative product on $H^\bullet(\a\llbracket\epsilon\rrbracket,\mu_1^{\mathbf{d},a})$.
\item[(2)] If moreover $H^1(L)=0$, then $H^0(\a\llbracket\epsilon\rrbracket,\mu_1^{\mathbf{d},a})$ is a topologically free $\R\llbracket\epsilon\rrbracket$-module, that is, isomorphic to $H^0(L)\llbracket\epsilon\rrbracket$ with the $\epsilon$-adic topology, with
\[
H^0\bigl(\a\llbracket\epsilon\rrbracket,\mu_1^{\mathbf{d},a}\bigr)\big/\epsilon\,H^0\bigl(\a\llbracket\epsilon\rrbracket,\mu_1^{\mathbf{d},a}\bigr)\ \cong\ H^0(L),
\]
and the induced product is a star product on $H^0(L)$ deforming the bracket induced by $\ell_2^{\a,S,\psi}$, i.e.\ the Poisson bracket of \Cref{1SCJHB7Z7}. Under the assumptions of \Cref{1SCVT88UN}, $H^0(L)=C^\infty(N)^\HHH$.
\end{itemize}
\end{theorem}

\begin{proof}
Both statements are \cite[Corollary 3.3]{Cattaneo2007} applied to $(S,N)$ and transported along $\Phi$. The transport is justified as follows.
\begin{itemize}
\item $\Phi$ is bijective, of degree $0$, and $\R\llbracket\epsilon\rrbracket$-linear. So it carries the element $\widehat a\in\epsilon\,\b^1\llbracket\epsilon\rrbracket$ produced by \cite[Corollary 3.3]{Cattaneo2007} to $a:=\Phi^{-1}(\widehat a)\in\epsilon\,\a^1\llbracket\epsilon\rrbracket$, and it identifies the two cohomologies by \eqref{1SCG5RSQW}.
\item The change of base point by $\widehat a$ is an invertible morphism with components $(\widehat a,\id)$ to $\mu^{\b,\mathbf d}$. By \Cref{1SCPN48SN}, its transport along $\Phi$ is a morphism with components $(a,\id)$ to $\mu^{\mathbf d}$. Since such a morphism determines its source, $\mu^{\mathbf d,a}$ is the transport of the change of base point of $\mu^{\b,\mathbf d}$ by $\widehat a$. It is therefore flat by \Cref{1SC2EHE94}(2).
\item The two hypotheses are the vanishing of the obstruction classes of that corollary. These classes lie in $H^2(\b,\ell_1^\b)$ and $H^1(\b,\ell_1^\b)$, hence in $H^2(L)$ and $H^1(L)$. The induction there starts at order $\epsilon^2$ because $\mu_0^{\mathbf{d}}=O(\epsilon^2)$, that is, because $N$ is coisotropic in $S$.
\end{itemize}
Finally, the terms of $\mu^{\mathbf d,a}_2$ containing $a$ involve some $\mu^{\mathbf d}_k$ with $k\ge3$, so $\mu^{\mathbf d,a}_2\equiv\mu^{\mathbf d}_2\pmod{\epsilon^2}$. The product therefore deforms $\ell_2^{\a,S,\psi}$ by \Cref{1SCL9J8YC}(3).
\end{proof}

\begin{remark}\label{1SCAWUTDP}
Part (2) recovers \cite[Theorem 6.2]{CattaneoZambon2009}. \Cref{1SCL9J8YC} refines it: without assuming the vanishing of $H^1(L)$ or $H^2(L)$, it constructs a curved $A_\infty$-quantization of the full $P_\infty$-algebra $\a=C^\infty(\wedge^\bullet L^*)$. If $H^2(L)=0$, the curvature can be removed by a change of base point. If in addition $H^1(L)=0$, the degree-zero cohomology of the resulting flat $A_\infty$-algebra carries the star product of \cite{CattaneoZambon2009}.
\end{remark}

The construction is summarized in \Cref{1SCQT9DWL}. Its two rows are the classical and the quantum structures attached to $\mathbf d$, and its vertical arrows are the semiclassical limit of \Cref{1SCL9J8YC}(3).

\begin{figure}[H]
\centering
{\hypersetup{linkcolor=figref}
\begin{tikzpicture}[
  font=\small,
  ent/.style={align=flush center, inner sep=3pt},
  lab/.style={font=\scriptsize, align=flush center, inner sep=2pt},
  ref/.style={
    font=\scriptsize,
    text=figref,
    align=flush center,
    inner sep=2pt
  },
  ar/.style={-{To[length=1.4mm]}, line width=.55pt, draw=figarrow},
  sq/.style={
    -{To[length=1.4mm]},
    line width=.55pt,
    draw=figarrow,
    decorate,
    decoration={
      zigzag,
      amplitude=.8pt,
      segment length=3.2pt,
      pre=lineto,
      pre length=1.5pt,
      post=lineto,
      post length=2.5pt
    }
  }
]
\def\A{0}\def\B{6.2}\def\C{12.4}
\def\YT{0}\def\YC{-1.9}\def\YQ{-4.0}

\node[ent] (s1) at (\A,\YT)
  {\textsc{Data}\\[1pt]
   $(M,\pi)$, $N\subset M$ pre-Poisson\\
   $L=\pi^{-1}(TN)\cap TN^\circ$};

\node[ent] (s3) at (\C,\YT)
  {\textsc{Datum} $\mathbf d=(S,\psi,\nabla)$\\[1pt]
   $S$ Poisson transversal,\\
   $N$ coisotropic in $S$};

\draw[ar] (s1) --
  node[lab,above]{choose a Poisson transversal,\\
    tubular neighbourhood, and connection}
  node[ref,below]{\S\ref{1SCS4XYSN}}
  (s3);

\node[ent] (g) at (\A,\YC)
  {$\pi_S\in\MC(\g_S)$\\
   $[\pi_S,\pi_S]=0$};

\node[ent] (lb) at (\B,\YC)
  {$(\b,\ell^{\b})$ flat $P_\infty$\\
   $\ell^{\b}_0=-P_S(\pi_S)=0$};

\node[ent] (la) at (\C,\YC)
  {$(\a,\ell^{\a,S,\psi})$ flat $P_\infty$\\
   $\ell_1=d_L$};

\draw[ar, rounded corners=2pt]
  (s3.south) -- ++(0,-0.4) -| (g.north);

\draw[ar] (g) --
  node[lab,above]{derived brackets\\ of $\pi_S$}
  node[ref,below]{Voronov, \eqref{1SCRZZPXA}}
  (lb);

\draw[ar] (lb) --
  node[lab,above]{$\Phi^{-1}(\cdot)\Phi^{\otimes k}$}
  node[ref,below]{Prop.~\ref{1SCMMB63H}}
  (la);

\node[ent] (eg) at (\A,\YQ)
  {$\epsilon\pi_S\in
    \MC\bigl(\epsilon\,\g_S\llbracket\epsilon\rrbracket\bigr)$};

\node[ent] (mub) at (\B,\YQ)
  {$\mu^{\b,\mathbf d}$ curved $A_\infty$\\
   on $\b\llbracket\epsilon\rrbracket$};

\node[ent] (mua) at (\C,\YQ)
  {$\mu^{\mathbf d}$ curved $A_\infty$
   on $\a\llbracket\epsilon\rrbracket$\\
   $\mu^{\mathbf d}_0=O(\epsilon^2)$};

\draw[sq] (eg) --
  node[lab,above,yshift=1pt]
    {$\U_*$\\ relative formality}
  node[ref,below]{Prop.~\ref{1SCKTNBKS}}
  (mub);

\draw[ar] (mub) --
  node[lab,above]{$\Phi^{-1}(\cdot)\Phi^{\otimes k}$}
  node[ref,below]{Thm.~\ref{1SCL9J8YC}}
  (mua);

\draw[ar] (eg) --
  node[lab,right]{coefficient\\ of $\epsilon$}
  (g);

\draw[ar] (mub) --
  node[lab,right,align=flush left]
    {$\pm\tfrac1\epsilon(\cdot\circ\Alt_k)\bmod\epsilon$\\
     {\color{figref}Prop.~\ref{1SCKTNBKS}(3)}}
  (lb);

\draw[ar] (mua) --
  node[lab,left,align=flush right]
    {$\pm\tfrac1\epsilon(\cdot\circ\Alt_k)\bmod\epsilon$\\
     {\color{figref}Thm.~\ref{1SCL9J8YC}(3)}}
  (la);

\end{tikzpicture}}
\caption{The classical and quantum structures attached to a quantization datum.}
\label{1SCQT9DWL}
\end{figure}
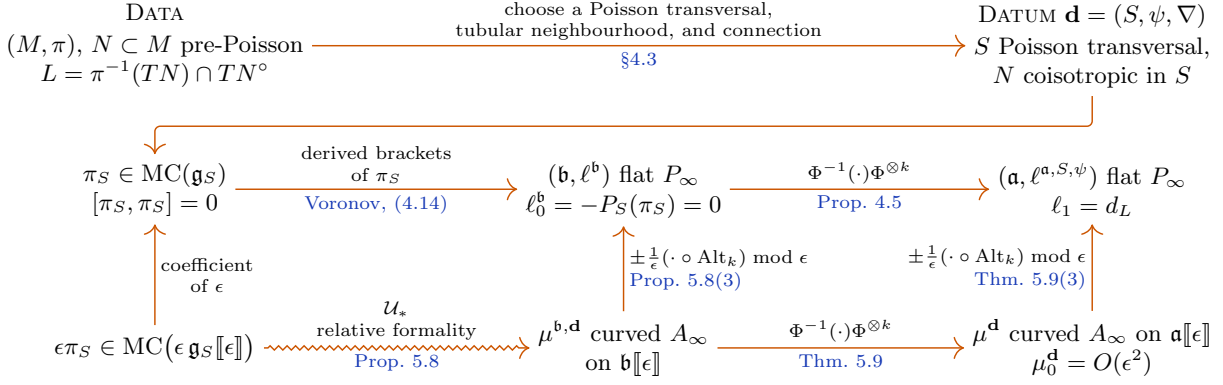

\subsection{Independence of the quantization datum}\label{1SCQPGX26}
The curved $A_\infty$-structure of \Cref{1SCL9J8YC} is built from a quantization datum $\mathbf d=(S,\psi,\nabla)$. We show that, up to isomorphism, it does not depend on $\mathbf d$. We compare the three entries one at a time: the transversal (\Cref{1SC82SQTE}), the tubular neighbourhood (\Cref{1SCLVACSU}), and the connection (\Cref{1SCABX7MH}). \Cref{1SCDN9DFE} composes them.

Throughout, $q:=\rank\nu$ and $\g_\nu$ denotes the multivector fields on the formal neighbourhood of the zero section of $\nu$. We set $\pi_\psi:=(\psi^{-1})_*\pi_S\in\g_\nu$, and write $\U^{\nu,\nabla}\colon\g_\nu\lto\widehat\D(\b)$ for the $L_\infty$-morphism of \Cref{1SCKTNBKS}(1) for $N\subset S$. By construction, $\mu^{\b,\mathbf d}$ corresponds by \eqref{1SC32WQT7} to the Maurer--Cartan element $\U^{\nu,\nabla}_*(\epsilon\pi_\psi)$ of $\epsilon\,\widehat\D(\b)\llbracket\epsilon\rrbracket$. Since $\Phi_S$ is a strict isomorphism from $\mu^{\mathbf d}$ to $\mu^{\b,\mathbf d}$ (\Cref{1SC2EHE94}(1)), an isomorphism $\varrho\colon\mu^{\b,\mathbf d}\lto\mu^{\b',\mathbf d'}$ gives, by \Cref{1SCPN48SN}, the isomorphism $\bigl(\Phi_{S'}^{-1}\circ \varrho_k\circ\Phi_S^{\otimes k}\bigr)_{k\ge0}$ from $\mu^{\mathbf d}$ to $\mu^{\mathbf d'}$. We may therefore work on $\b\llbracket\epsilon\rrbracket$.

The first ingredient is the naturality of $\U^{\nu,\nabla}$ under bundle isomorphisms.

\begin{lemma}\label{1SCVVM8AU}
Let $\Lambda\colon\nu\to\nu'$ be a vector bundle isomorphism covering $\id_N$, $\b':=C^\infty(\wedge^\bullet\nu')$, and $\theta:=\wedge^\bullet\Lambda\colon\b\to\b'$. Then
\[
\theta_*\U^{\nu,\nabla}_n(X_1,\dots,X_n)=\U^{\nu',\Lambda_*\nabla}_n(\Lambda_*X_1,\dots,\Lambda_*X_n)
\]
for all $X_1,\dots,X_n\in\g_\nu$.
\end{lemma}

\begin{proof}
Let $W\subset N$ be a chart with coordinates $(x^i)$, $(e_\beta)$ a frame of $\nu$ over $W$, and $e'_\beta:=\Lambda e_\beta$, a frame of $\nu'$ since $\Lambda$ covers $\id_N$. Let $(y^\beta)$, $(y'^\beta)$ be the fibre coordinates, $\zeta_\beta:=e_\beta\in\b$ and $\zeta'_\beta:=e'_\beta\in\b'$.

Since $\Lambda(y^\beta e_\beta)=y^\beta e'_\beta$, we have $y'^\beta\circ\Lambda=y^\beta$. So $\Lambda$ reads $(x,y)\mapsto(x,y)$ in these coordinates, and $\Lambda_*\frac{\partial}{\partial y^\beta}=\frac{\partial}{\partial y'^\beta}$, $\Lambda_*\frac{\partial}{\partial x^i}=\frac{\partial}{\partial x^i}$. Likewise $\theta(\zeta_\beta)=\zeta'_\beta$, and $\theta\frac{\partial}{\partial\zeta_\beta}\theta^{-1}(\zeta'_\gamma)=\delta^\gamma_\beta$ gives $\theta_*\frac{\partial}{\partial\zeta_\beta}=\frac{\partial}{\partial\zeta'_\beta}$. Hence $\Lambda_*$ and $\theta_*$ both act on coefficients by the substitution $(y,\zeta)\mapsto(y',\zeta')$. In particular, the Fourier rules $y^\beta\mapsto-\frac{\partial}{\partial\zeta_\beta}$, $\frac{\partial}{\partial y^\beta}\mapsto\zeta_\beta$ give $\mathfrak F'\circ\Lambda_*=\theta_*\circ\mathfrak F$ over $W$, where $\mathfrak F,\mathfrak F'$ are the Fourier transforms; recall that the Fourier transform is canonical and does not depend on the connection \cite[Theorem 4.1 and Lemma 4.2]{Cattaneo2007}.

If $\nabla e_\beta=\omega^\gamma_{\ \beta}e_\gamma$, then $(\Lambda_*\nabla)e'_\beta=\Lambda(\nabla e_\beta)=\omega^\gamma_{\ \beta}\,e'_\gamma$.

Over $W$, the local expressions defining $\U^{\nu,\nabla}$ (Fourier transform, local formality, Fedosov globalization \cite[\S A.9]{Cattaneo2007}) depend only on the coordinates, the frame, the connection matrix of $\nabla$ in that frame, and the data fixed on $N$. The Fourier transforms are intertwined as above, the local formality morphism is equivariant under linear changes of fibre coordinates, and the connection matrices of $\nabla$ and $\Lambda_*\nabla$ in the frames $(e_\beta)$ and $(e'_\beta)$ coincide. Hence $\U^{\nu',\Lambda_*\nabla}_n$ is obtained from $\U^{\nu,\nabla}_n$ by $(y,\zeta)\mapsto(y',\zeta')$, which is the identity over $W$. Both sides are multidifferential, so the identity is local, and $W$ is arbitrary.
\end{proof}

\noindent\textbf{Comparing two Poisson transversals.}
Let $\varphi\colon S\to S'$ be a Poisson diffeomorphism with $\varphi|_N=\id_N$, let $\nu':=\nu_{S'}$, and set
\begin{equation}\label{1SCD3LY62}
\Lambda:=\nu(\varphi),\quad
\psi':=\varphi\circ\psi\circ\Lambda^{-1},\quad
\nabla':=\Lambda_*\nabla,\quad
\theta:=\wedge^\bullet\Lambda .
\end{equation}

\begin{proposition}\label{1SC82SQTE}
The triple $\mathbf d':=(S',\psi',\nabla')$ is a quantization datum, and
\[
\Psi:=\Phi_{S'}^{-1}\circ\theta\circ\Phi_S
\]
is a strict isomorphism from $\{\mu^{\mathbf d}_k\}$ to $\{\mu^{\mathbf d'}_k\}$.
\end{proposition}

\begin{proof}
Clearly $\psi'|_N=\id_N$. Since $\Lambda^{-1}$ is linear, $\nu(\Lambda^{-1})=\Lambda^{-1}$, so
$\nu(\psi')=\Lambda\circ\nu(\psi)\circ\Lambda^{-1}=\id_{\nu'}$, and $\psi'$ is normalized. Since $\varphi$ is Poisson,
\begin{equation}\label{1SCMS9RZV}
\pi_{\psi'}=\Lambda_*(\psi^{-1})_*(\varphi^{-1})_*\pi_{S'}=\Lambda_*\pi_\psi .
\end{equation}
By \Cref{1SCVVM8AU} and \eqref{1SCMS9RZV},
\[
\theta_*\U^{\nu,\nabla}_*(\epsilon\pi_\psi)=\U^{\nu',\nabla'}_*(\epsilon\pi_{\psi'}).
\]
Since $\theta$ is an algebra isomorphism, $\theta_*\mmm=\mmm$. Hence $\theta$ is a strict isomorphism from $\mu^{\b,\mathbf d}$ to $\mu^{\b',\mathbf d'}$, and $\Psi$ is the corresponding isomorphism on $\a\llbracket\epsilon\rrbracket$.
\end{proof}

\noindent \textbf{Comparing two tubular neighbourhoods.} Let $\psi_0,\psi_1$ be normalized tubular neighbourhoods of $N$ in $S$, and
\[
\delta:=\psi_1^{-1}\circ\psi_0 .
\]
Then $\delta|_N=\id_N$, $\nu(\delta)=\nu(\psi_1)^{-1}\circ\nu(\psi_0)=\id_\nu$, and, since $\psi_0=\psi_1\circ\delta$,
\begin{equation}\label{1SCGY3AA2}
\pi_{\psi_1}=\delta_*\pi_{\psi_0}.
\end{equation}

We will show that $\epsilon\pi_{\psi_0}$ and $\epsilon\pi_{\psi_1}$ are gauge equivalent and transport this equivalence through $\U^{\nu,\nabla}$ by \Cref{1SCMF2KEY}. The difficulty is that $\delta_{*}-\id$ is not divisible by $\epsilon$, so the $\epsilon$-adic filtration does not control the change of tubular neighbourhood. We need a notion of smallness that measures vanishing in the fibre directions and is compatible with multivector fields. The natural one is the fibrewise grading given by the degree in the fibre coordinates minus the number of fibre derivatives, which also appears in Fedosov globalization \cite{dolgushev2005covariant}. We use it to control $\delta$ and its logarithm, and  then combine it with the $\epsilon$-adic order into a filtration in which both the gauge generator and the Maurer--Cartan element have positive degree.

Let $(x,y)$ be local coordinates on $\nu$, with $(x^i)$ on $N$ and $(y^j)$ linear fibre coordinates. For $t\in\R^\times$, let
\[
m_t(x,y):=(x,ty),\quad \theta_t:=\wedge^\bullet m_t .
\]
Then $m_t$ covers $\id_N$ and $(m_t)_*\nabla=\nabla$.

\begin{definition}\label{1SCNVG8WE}
An element $X\in\g_\nu$ has \defn{weight} $k$ if $(m_t)_*X=t^{-k}X$ for all $t\in\R^\times$. An element $D\in\widehat\D(\b)$ has weight $k$ if $(\theta_t)_*D=t^{-k}D$ for all $t\in\R^\times$. We write $\wt(X)\ge a$ if every homogeneous component of $X$ has weight $\ge a$.
\end{definition}

\begin{lemma}\label{1SCSR6LN4}
\begin{itemize}
\item[(1)] For every multi-index $\alpha$,
\[
\wt\Bigl(f(x)y^\alpha\frac{\partial}{\partial x^{i_1}}\wedge\cdots\wedge\frac{\partial}{\partial x^{i_p}}\wedge\frac{\partial}{\partial y^{j_1}}\wedge\cdots\wedge\frac{\partial}{\partial y^{j_s}}\Bigr)=|\alpha|-s .
\]
\item[(2)] If $D\colon C^\infty(\wedge^{p_1}\nu)\otimes\cdots\otimes C^\infty(\wedge^{p_n}\nu)\to C^\infty(\wedge^{p_0}\nu)$, then
$\wt(D)=p_1+\cdots+p_n-p_0$.
\item[(3)] The weight is additive for the Schouten and Gerstenhaber brackets, and $\wt(\mmm)=0$.
\item[(4)] If $\wt(X_i)\ge a_i$ for $i=1,\dots,n$, then
$\wt\bigl(\U^{\nu,\nabla}_n(X_1,\dots,X_n)\bigr)\ge a_1+\cdots+a_n$.
\end{itemize}
In particular, all weights are $\ge-q$, and the weights of an $n$-ary operator are $\le nq$.
\end{lemma}

\begin{proof}
(1) We have $(m_t)_*y^\alpha=t^{-|\alpha|}y^\alpha$ and $(m_t)_*\frac{\partial}{\partial y^j}=t\,\frac{\partial}{\partial y^j}$, while $f(x)$ and $\frac{\partial}{\partial x^i}$ are fixed.

(2) The map $\theta_t$ is multiplication by $t^p$ on $C^\infty(\wedge^p\nu)$, so
$\theta_t\circ D\circ(\theta_t^{-1})^{\otimes n}=t^{\,p_0-(p_1+\cdots+p_n)}D$.

(3) The brackets are natural under $(m_t)_*$ and $(\theta_t)_*$, and $\theta_t$ is an algebra automorphism.

(4) \Cref{1SCVVM8AU} with $\Lambda=m_t$ gives
\[
(\theta_t)_*\U^{\nu,\nabla}_n(X_1,\dots,X_n)=\U^{\nu,\nabla}_n\bigl((m_t)_*X_1,\dots,(m_t)_*X_n\bigr).
\]
This proves (4) for homogeneous $X_i$. The general case follows since $\U^{\nu,\nabla}_n$ is well defined on completed inputs \cite[Theorem 4.7]{Cattaneo2007}.
\end{proof}

The only term of $\pi_\psi$ that could have weight $-2$ vanishes by coisotropy.

\begin{lemma}\label{1SC545338}
For every normalized tubular neighbourhood $\psi$ of $N$ in $S$, $\wt(\pi_\psi)\ge-1$.
\end{lemma}

\begin{proof}
By \Cref{1SCSR6LN4}(1), the only possible term of weight $-2$ is
$\pi_\psi^{jk}(x,0)\,\frac{\partial}{\partial y^j}\wedge\frac{\partial}{\partial y^k}$.
Since $\nu(\psi)=\id_\nu$, this term is $(\wedge^2\ppp_S)(\pi_S|_N)=P_S(\pi_S)$, which vanishes because $N$ is coisotropic in $S$.
\end{proof}

The gauge generator comes from the fact that $\delta$ fixes $N$ pointwise and induces the identity on $\nu$.

\begin{lemma}\label{1SC3P8VEF}
There is a formal vector field $X$ on $\nu$ with $\wt(X)\ge1$ and $\delta_*=e^{\ad_X}$.
\end{lemma}

\begin{proof}
Write
\begin{equation}\label{1SC4DPVZX}
\delta(x,y)=\bigl(x+a(x,y),\,y+b(x,y)\bigr).
\end{equation}
Since $\delta|_N=\id_N$ and $\nu(\delta)=\id_\nu$,
\[
a(x,0)=0,\quad b(x,0)=0,\quad \frac{\partial b}{\partial y}(x,0)=0 .
\]
Hence $\wt(a)\ge1$ and $\wt(b)\ge2$. Note that $\frac{\partial a}{\partial y}(x,0)$ need not vanish.

We claim that $R:=\delta^*-\id$ raises the weight by at least $1$. For $f=g(x)y^\alpha$,
\[
\delta^*f-f=\bigl(g(x+a)-g(x)\bigr)(y+b)^\alpha+g(x)\bigl((y+b)^\alpha-y^\alpha\bigr),
\]
and each term on the right contains a factor $a$, or a factor $b$ in place of a factor $y$.

It follows that
\[
Y:=\log\delta^*=\sum_{n\ge1}\frac{(-1)^{n+1}}{n}R^n
\]
converges and raises the weight by at least $1$. We claim that $Y$ is a derivation. Set $T:=\delta^*=e^Y$. Modulo any fixed level of the weight filtration, $e^{tY}$ is polynomial in $t$. For $f,g$, the expression
$e^{tY}(fg)-e^{tY}(f)e^{tY}(g)$
is polynomial in $t$ modulo that level and vanishes for every $t=n\in\mathbb N$, since $e^{nY}=T^n$ is an algebra automorphism. Hence it vanishes identically. Differentiating at $t=0$ gives $Y(fg)=Y(f)g+fY(g)$. Thus $Y$ is a formal vector field with $\wt(Y)\ge1$. On multivector fields,
\[
(\delta^{-1})_*A=e^Y\circ A\circ(e^{-Y})^{\otimes k}=e^{\ad_Y}A .
\]
Hence $X:=-Y$ satisfies $\delta_*=e^{\ad_X}$.
\end{proof}

The weight controls $X$ and the factor $\epsilon$ controls $\epsilon\pi_\psi$. We combine them into a single  filtration. For $k\in\Z$, set
\begin{equation}\label{1SCG34NBG}
\FFF^k\g_\nu\llbracket\epsilon\rrbracket:=\Bigl\{\sum_a\epsilon^aA_a:\wt(A_a)+2a\ge k\Bigr\},\quad
\FFF^k\widehat\D(\b)\llbracket\epsilon\rrbracket:=\Bigl\{\sum_a\epsilon^aD_a:\wt(D_a)+2a\ge k\Bigr\}.
\end{equation}

\begin{lemma}\label{1SC7A3T2F}
\begin{itemize}
\item[(1)] The filtrations \eqref{1SCG34NBG} are complete, and $\U^{\nu,\nabla}$ is filtered in the sense of \eqref{1SCLC9EC8}.
\item[(2)] $\epsilon\pi_\psi\in\FFF^1$ for every normalized $\psi$, and $X\in\FFF^1$.
\item[(3)] Let $Y\in\FFF^1\widehat\D(\b)\llbracket\epsilon\rrbracket$ have degree $0$, and let $\widehat Y$ be the coderivation of the tensor coalgebra of $V:=\b\llbracket\epsilon\rrbracket[1]$ with components $Y$. Then $e^{\widehat Y}$ is well defined.
\end{itemize}
\end{lemma}

\begin{proof}
(1) Compatibility with $d=[\mmm,\cdot]$ and with the brackets, and filteredness, follow from \Cref{1SCSR6LN4}(3),(4).

Let $I^r\subset\g_\nu$ be the elements whose monomials satisfy $|\alpha|\ge r$, and $W^k\subset\g_\nu$ the elements whose homogeneous components have weight $\ge k$. Since $0\le s\le q$, \Cref{1SCSR6LN4}(1) gives $|\alpha|-q\le\wt\le|\alpha|$ for every monomial, hence
$I^{k+q}\subset W^k\subset I^k$.
So $W^\bullet$ and $I^\bullet$ define the same topology, and $\g_\nu$ is complete and separated for $W^\bullet$ because it is for $I^\bullet$.

By definition, $\sum_a\epsilon^aA_a\in\FFF^k$ if and only if $A_a\in W^{k-2a}$ for all $a$. Since all weights are $\ge-q$, $\FFF^{-q}=\g_\nu\llbracket\epsilon\rrbracket$, and $\bigcap_k\FFF^k=0$. Let $Z_r=\sum_a\epsilon^aA_{r,a}$ be a Cauchy sequence. For each $k$ there is $R_k$ with
\[
A_{r,a}-A_{s,a}\in W^{k-2a}\quad(r,s\ge R_k,\ a\ge0).
\]
For fixed $a$, taking $k=j+2a$ with $j$ arbitrary shows that $(A_{r,a})_r$ is Cauchy for $W^\bullet$, so it converges to some $A_a\in\g_\nu$. Set $Z:=\sum_a\epsilon^aA_a$. Let $r\ge R_k$. For each $a$, choose $s\ge R_k$ with $A_{s,a}-A_a\in W^{k-2a}$. Then
$A_{r,a}-A_a=(A_{r,a}-A_{s,a})+(A_{s,a}-A_a)\in W^{k-2a}$,
so $Z_r-Z\in\FFF^k$. Hence $Z_r\to Z$.

By \Cref{1SCSR6LN4}(2), an operator of arity $n$ has weights in $[-q,nq]$. So $W^k$ contains no nonzero operator of arity $n$ when $k>nq$, and a Cauchy sequence in $\widehat\D(\b)$ has stationary components in each arity. It converges to the element whose components are these stationary values. Thus $\widehat\D(\b)$ is complete and separated for $W^\bullet$, and the argument above applies verbatim.

(2) By \Cref{1SC545338}, $\wt(\pi_\psi)+2\ge1$. By \Cref{1SC3P8VEF}, $\wt(X)\ge1$.

(3) Write $Y=\sum\epsilon^aY_{a,w}$ with $\wt(Y_{a,w})=w$ and $w+2a\ge1$. By \Cref{1SCSR6LN4}(2), $\widehat Y_{a,w}$ lowers the total exterior degree by $w$. On $V^{\otimes n}$, this degree lies in $[0,nq]$, and it stays $\ge0$ after any number of applications, even though the tensor length may grow. Hence, if $\widehat Y_{a_1,w_1}\cdots\widehat Y_{a_j,w_j}$ is nonzero on $V^{\otimes n}$, then
\[
nq\ \ge\ w_1+\cdots+w_j\ \ge\ j-2(a_1+\cdots+a_j).
\]
So $\widehat Y^j\equiv0\bmod\epsilon^N$ on $V^{\otimes n}$ for $j>nq+2N-2$.
\end{proof}

\begin{remark}\label{1SCTLB42Z}
The filtration \eqref{1SCG34NBG} starts at $\FFF^{-q}$ rather than at $\FFF^1$. However, $\FFF^1$ is a sub-DGLA, since it is closed under $d$ and $[\FFF^1,\FFF^1]\subset\FFF^2\subset\FFF^1$. It contains $\epsilon\pi_\psi$ and $X$, and $\mathcal U^{\nu,\nabla}$ maps $\FFF^1\g_\nu\llbracket\epsilon\rrbracket$ into $\FFF^1\widehat\D(\b)\llbracket\epsilon\rrbracket$ by (1). We may therefore apply \Cref{1SCMF2KEY} to the restriction of $\U^{\nu,\nabla}$ to $\FFF^1$, with the induced filtration $\FFF^1\supset\FFF^2\supset\cdots$.
\end{remark}

\begin{proposition}\label{1SCLVACSU}
The curved $A_\infty$-algebras of $(S,\psi_0,\nabla)$ and $(S,\psi_1,\nabla)$ are isomorphic.
\end{proposition}

\begin{proof}
Since $\g_\nu$ has zero differential, \eqref{1SCGY3AA2}, \Cref{1SCV8CXM4} and \Cref{1SC3P8VEF} give
\[
\epsilon\pi_{\psi_1}=X\cdot(\epsilon\pi_{\psi_0}),\quad X\in\FFF^1 .
\]
By \Cref{1SC7A3T2F}(1), \Cref{1SCTLB42Z} and \Cref{1SCMF2KEY}(2), the Maurer--Cartan elements $\varpi_i:=\U^{\nu,\nabla}_*(\epsilon\pi_{\psi_i})$ are gauge equivalent through some $Y\in\FFF^1$ of degree $0$. The arity-zero component of $Y$ is an element of $\b$, so by \Cref{1SCSR6LN4}(2) its weight is $\le0$, and the condition $\wt+2a\ge1$ in \eqref{1SCG34NBG} forces $a\ge1$: it lies in $\epsilon\,\b\llbracket\epsilon\rrbracket$. By \Cref{1SC7A3T2F}(3), $e^{\widehat Y}$ is well defined, and \Cref{1SCVEVYP9} shows that it is an isomorphism from $\mu^{\b,(S,\psi_0,\nabla)}$ to $\mu^{\b,(S,\psi_1,\nabla)}$.
\end{proof}

\noindent\textbf{Comparing two connections.} Let $\nabla,\nabla'$ be connections on $\nu$. Here $\epsilon\pi_\psi$ is fixed, and only the morphism $\U^{\nu,\nabla}$ changes. We show that $\U^{\nu,\nabla}$ and $\U^{\nu,\nabla'}$ are homotopic, following the proof of \cite[Theorem 4.13]{kraftschnitzer2021}, in which only the shortening step needs to change.

Let $(x^i)$ be local coordinates on $N$, $(u^i)$ the fibre coordinates of $TN$ used in the Fedosov construction, and $(\zeta_\beta)$ the generators of $\b$ given by a local frame of $\nu$. Let $D,D'$ be the Fedosov differentials of $\nabla,\nabla'$, let $\delta_{\mathrm K}:=dx^i\frac{\partial}{\partial u^i}$, and let $\kappa:=\omega'-\omega$ be the difference of the connection matrices.

\begin{lemma}\label{1SCJ5VX3L}
There is a fibrewise vector field $\xi$ of form degree $0$ and positive Fedosov filtration degree, with coefficients vanishing at $u=0$, such that
$D'=e^{\ad_\xi}D$.
Its first term is
\[
\xi_1=u^i\kappa^\gamma_{\ i\beta}(x)\,\zeta_\gamma\frac{\partial}{\partial\zeta_\beta}.
\]
\end{lemma}

\begin{proof}
We adapt the recursive argument of \cite[Appendix C]{bursztyndolgushevwaldmann2012}; see also \cite[Remark 4.6 and Proposition 4.7]{kraftschnitzer2021}. We use the Fedosov filtration of \cite[Remark 4.6]{kraftschnitzer2021}: $dx^i$ and $u^i$ have degree $1$, $\frac{\partial}{\partial u^i}$ has degree $-1$, and $\zeta_\beta,\frac{\partial}{\partial\zeta_\beta}$ have degree $0$; on fibrewise multidifferential operators we pass to the $u$-adic completion, as in loc.\ cit. This degree is additive under brackets and compositions. For it, $\delta_{\mathrm K}$ has degree $0$, the connection terms of $D$ and $D'$ have degree $1$, and the Fedosov corrections have degree $\ge2$. Hence $D=-\delta_{\mathrm K}+(\text{degree}\ge1)$, and $D'-D$ has degree $\ge1$, with degree-$1$ component
\[
E_1=dx^i\kappa^\gamma_{\ i\beta}(x)\,\zeta_\gamma\frac{\partial}{\partial\zeta_\beta}.
\]

Suppose we have constructed fibrewise vector fields $\xi_1,\dots,\xi_r$ of form degree $0$, with $\xi_s$ of degree $s$ and coefficients vanishing at $u=0$, such that
$D_r:=e^{\ad_{\xi_r}}\cdots e^{\ad_{\xi_1}}D$
satisfies $E:=D'-D_r$ of degree $\ge r+1$ (for $r=0$, $D_0=D$ and $E=D'-D$). Both $D_r$ and $D'$ square to zero, since $D_r$ is conjugate to $D$. Hence
\[
0=(D_r+E)^2-D_r^2=[D_r,E]+E^2 .
\]
In degree $r+1$ the only contribution is $[-\delta_{\mathrm K},E_{r+1}]$, where $E_{r+1}$ is the component of $E$ of degree $r+1$: all other terms of $[D_r,E]$ have degree $\ge r+2$, and $E^2$ has degree $\ge2r+2\ge r+2$. Thus $\delta_{\mathrm K}E_{r+1}=0$. Moreover, $E_{r+1}$ is the adjoint action of a fibrewise vector field of form degree $1$. Since $\delta_{\mathrm K}\delta_{\mathrm K}^{-1}+\delta_{\mathrm K}^{-1}\delta_{\mathrm K}$ is the identity in positive form degree, $E_{r+1}=\delta_{\mathrm K}(\delta_{\mathrm K}^{-1}E_{r+1})$. Set
$\xi_{r+1}:=\delta_{\mathrm K}^{-1}E_{r+1}$.
It is a fibrewise vector field of form degree $0$ and of degree $r+1$, since $\delta_{\mathrm K}^{-1}$ preserves the degree, and its coefficients have positive $u$-degree, hence vanish at $u=0$. Since $[\xi_{r+1},-\delta_{\mathrm K}]=\delta_{\mathrm K}(\xi_{r+1})=E_{r+1}$, the difference $D'-e^{\ad_{\xi_{r+1}}}D_r$ has degree $\ge r+2$. For $r=0$ this gives $\xi_1=\delta_{\mathrm K}^{-1}E_1=u^i\kappa^\gamma_{\ i\beta}\zeta_\gamma\frac{\partial}{\partial\zeta_\beta}$.

Since the degrees of the $\xi_r$ increase, the Baker--Campbell--Hausdorff product $\xi$ of $\xi_1,\xi_2,\dots$ (in this order) converges, and $e^{\ad_\xi}=\lim_r e^{\ad_{\xi_r}}\cdots e^{\ad_{\xi_1}}$. Brackets of fibrewise vector fields of degree $\ge1$ whose coefficients vanish at $u=0$ again have both properties, so $\xi$ has positive degree and coefficients vanishing at $u=0$, and $D'=e^{\ad_\xi}D$.

Finally, $\ad_\xi$ raises the degree by at least $1$, also on fibrewise multidifferential operators containing $\frac{\partial}{\partial u^i}$: for instance $\bigl[\xi_1,\frac{\partial}{\partial u^j}\bigr]=-\kappa^\gamma_{\ j\beta}\,\zeta_\gamma\frac{\partial}{\partial\zeta_\beta}$ has degree $0$, while $\frac{\partial}{\partial u^j}$ has degree $-1$. Hence $(\ad_\xi)^m$ raises the degree by at least $m$, and $e^{\ad_\xi}$ converges on the completed Fedosov resolutions. In particular $\xi$ lies in the first filtration step, as required in \cite[Proposition 3.10]{kraftschnitzer2021}.
\end{proof}

\begin{remark}\label{1SCNRLG2F}
Here $\xi_1$ is only linear in $u$, and
$\bigl[\xi_1,\frac{\partial}{\partial u^j}\bigr]=-\kappa^\gamma_{\ j\beta}\,\zeta_\gamma\frac{\partial}{\partial\zeta_\beta}$
need not vanish at $u=0$. So the shortening argument of \cite[Corollary 4.10]{kraftschnitzer2021} does not apply verbatim. \Cref{1SCJEC9RK} replaces it.
\end{remark}

Let $\tau,\lambda$ and $\tau',\lambda'$ be the lifting and shortening maps of $D$ and $D'$. For a $D$-flat multivector field or multidifferential operator $w$ of form degree $0$,
\[
\lambda(w)(a_0,\dots,a_k)=\mathsf{ev}_0\bigl(w(\tau(a_0),\dots,\tau(a_k))\bigr),
\]
and similarly for $\lambda'$, where $\mathsf{ev}_0$ is evaluation at $u=0$.

\begin{lemma}\label{1SCJEC9RK}
\begin{itemize}
\item[(1)] $e^\xi\tau(a)=\tau'(a)$ for every $a\in\b$.
\item[(2)] $e^{\ad_\xi}$ sends $D$-flat multivector fields and multidifferential operators of form degree $0$ to $D'$-flat ones, and $\lambda'\circ e^{\ad_\xi}=\lambda$.
\end{itemize}
In particular,
\begin{equation}\label{1SCTE6QPW}
e^{\ad_\xi}\circ\lambda^{-1}=(\lambda')^{-1}.
\end{equation}
\end{lemma}

\begin{proof}
Since the coefficients of $\xi$ vanish at $u=0$,
\begin{equation}\label{1SCQGUY28}
\mathsf{ev}_0(e^{\pm\xi}b)=\mathsf{ev}_0(b)
\end{equation}
for every fibrewise function $b$.

(1) We have $D'(e^\xi\tau(a))=e^\xi D\tau(a)=0$, and $\mathsf{ev}_0(e^\xi\tau(a))=a$ by \eqref{1SCQGUY28}. A $D'$-flat function is determined by its value at $u=0$.

(2) We have $[D',e^{\ad_\xi}w]=e^{\ad_\xi}[D,w]=0$. Since $e^{\ad_\xi}w=e^\xi\circ w\circ(e^{-\xi})^{\otimes(k+1)}$, (1) and \eqref{1SCQGUY28} give
\[
\lambda'(e^{\ad_\xi}w)(a_0,\dots,a_k)=\mathsf{ev}_0\bigl(e^\xi\,w(\tau(a_0),\dots,\tau(a_k))\bigr)=\lambda(w)(a_0,\dots,a_k). \qedhere
\]
\end{proof}

\begin{lemma}\label{1SC3ZLF6S}
The morphisms $\U^{\nu,\nabla}$ and $\U^{\nu,\nabla'}$ are $L_\infty$-homotopic, through a homotopy compatible with the $\epsilon$-adic filtration.
\end{lemma}

\begin{proof}
We follow the proof of \cite[Theorem 4.13]{kraftschnitzer2021}, with \Cref{1SCJ5VX3L} in place of \cite[Proposition 4.7]{kraftschnitzer2021}. The ingredients are the following.
\begin{itemize}
\item[(1)] The conjugacy $D'=e^{\ad_\xi}D$ with $\xi$ of positive Fedosov filtration degree (\Cref{1SCJ5VX3L}), so that the difference of the two Fedosov twisting elements is gauge equivalent to zero, and the general comparison of twisted morphisms \cite[Proposition 3.10]{kraftschnitzer2021}, which applies verbatim.
\item[(2)] Let $A_x:=\R[[u^1,\dots,u^m]]\otimes\wedge^\bullet\nu_x$ be the fibre algebra of the Fedosov resolution, and $\T(A_x)$ its multivector fields \cite[\S A.5]{Cattaneo2007}, generated over $A_x$ by $\frac{\partial}{\partial u^i}$ and $\frac{\partial}{\partial\zeta_\beta}$. For vector fields $v,v_1,\dots,v_n\in\T(A_x)$, that is, elements linear in these generators, $\U_1(v)=v$ and $\U_n(v_1,\dots,v_n)=0$ $(n\ge2)$,
as for Kontsevich's morphism \cite[\S7]{Kontsevich2003}. The first identity holds because $\U_1$ is the HKR map. For the second, in \cite[Eq.(9)]{Cattaneo2007} each vertex of the first type must have exactly one outgoing edge, since $v_i$ is linear in the generators. So a contributing graph has $n$ edges, and $|E_\Gamma|=2n+m-2$ forces $(n,m)=(2,0)$. The only graph is then the two-cycle, whose weight vanishes by \cite[Lemma 7.3]{Kontsevich2003}. The same holds with coefficients in differential forms. Hence $\xi$ and the difference of the two Fedosov twisting elements are mapped to themselves.

\item[(3)] The compatibility of the shortening maps, which in \cite[Corollary 4.10]{kraftschnitzer2021} uses that the gauge generator is at least quadratic in the fibre coordinates, and which is replaced here by \eqref{1SCTE6QPW}.
\item[(4)] The comparison of the two projections onto the zero $D$-cochains, as in \cite[Lemma 4.12]{kraftschnitzer2021}. Its proof uses the analogue of \cite[Lemma 4.11]{kraftschnitzer2021}, namely \Cref{1SCJEC9RK}(2), together with the Fedosov homotopy and the splitting of \cite[Proposition 4.2 and Lemma 4.5]{kraftschnitzer2021}, which hold verbatim for the resolutions of \cite[\S A.9]{Cattaneo2007}.
\end{itemize}
Assembling (1)--(4) as in \cite[Theorem 4.13]{kraftschnitzer2021} gives an $L_\infty$-homotopy between $\U^{\nu,\nabla}$ and $\U^{\nu,\nabla'}$. Its $\R\llbracket\epsilon\rrbracket$-linear extension is compatible with the $\epsilon$-adic filtration.
\end{proof}

\begin{proposition}\label{1SCABX7MH}
The curved $A_\infty$-algebras of $(S,\psi,\nabla)$ and $(S,\psi,\nabla')$ are isomorphic.
\end{proposition}

\begin{proof}
By \Cref{1SC3ZLF6S} and \Cref{1SCMF2KEY}(3), $\U^{\nu,\nabla}_*(\epsilon\pi_\psi)$ and $\U^{\nu,\nabla'}_*(\epsilon\pi_\psi)$ are gauge equivalent through some $Y\in\epsilon\,\widehat\D(\b)\llbracket\epsilon\rrbracket$ of degree $0$. Then $e^{\widehat Y}$ converges $\epsilon$-adically, its arity-zero component lies in $\epsilon\,\b\llbracket\epsilon\rrbracket$, and \Cref{1SCVEVYP9} shows that $e^{\widehat Y}$ is an isomorphism from $\mu^{\b,(S,\psi,\nabla)}$ to $\mu^{\b,(S,\psi,\nabla')}$.
\end{proof}

We can now compose the three comparisons.

\begin{theorem}\label{1SCDN9DFE}
For any two quantization data $\mathbf d$ and $\mathbf d'$, the curved $A_\infty$-algebras $\{\mu^{\mathbf d}_k\}$ and $\{\mu^{\mathbf d'}_k\}$ are isomorphic.
\end{theorem}

\begin{proof}
Let $\mathbf d=(S,\psi,\nabla)$ and $\mathbf d'=(S',\psi',\nabla')$. As in the proof of \Cref{1SCFBTZ34}, all the data depend only on germs along $N$, and \cite[Theorem 4.3]{CattaneoZambon2009} gives a Poisson diffeomorphism $\varphi\colon S\to S'$ with $\varphi|_N=\id_N$. Let $\Lambda:=\nu(\varphi)$ and
$\psi_1:=\varphi^{-1}\circ\psi'\circ\Lambda$.
Then $\psi_1|_N=\id_N$ and $\nu(\psi_1)=\Lambda^{-1}\circ\id_{\nu'}\circ\Lambda=\id_\nu$, so $\psi_1$ is normalized. Since $\varphi\circ\psi_1\circ\Lambda^{-1}=\psi'$,
\[
(S,\psi,\nabla)\xrightarrow{\ \text{\Cref{1SCLVACSU}}\ }(S,\psi_1,\nabla)
\xrightarrow{\ \text{\Cref{1SC82SQTE}}\ }(S',\psi',\Lambda_*\nabla)
\xrightarrow{\ \text{\Cref{1SCABX7MH}}\ }(S',\psi',\nabla'),
\]
and each arrow is an isomorphism of curved $A_\infty$-algebras.
\end{proof}

\begin{remark}\label{1SCRMKTFC}
The structures $\mu^{\mathbf d}$ also depend on the torsion-free connection on $N$ and the formal exponential map fixed in \Cref{1SC9UZXPX}, that is, on the globalization data of \cite[\S 4.1]{kraftschnitzer2021}. Their isomorphism class does not depend on these data. Indeed, a change of torsion-free connection modifies the connection part of the Fedosov differential by $[\delta_{\mathrm K}s,\cdot\,]$ with $s=-\tfrac12u^iu^j(\Gamma'^k_{ij}-\Gamma^k_{ij})\frac{\partial}{\partial u^k}$, and the normalization of the Fedosov correction has no constant or linear term in $u$; neither involves the generators $\zeta_\beta$ of $\b$. The recursion of \Cref{1SCJ5VX3L} then gives a gauge generator at least quadratic in $u$, as in \cite[Proposition 4.7]{kraftschnitzer2021}. The proof of \Cref{1SC3ZLF6S} applies, with the arguments of \cite[Corollary 4.10 and Lemmas 4.11--4.12]{kraftschnitzer2021} in place of \Cref{1SCJEC9RK}, and gives an $L_\infty$-homotopy between the two formality morphisms. \Cref{1SCMF2KEY}(3) and \Cref{1SCVEVYP9} conclude as in \Cref{1SCABX7MH}. The same applies to the data fixed on $\Delta$ in \Cref{1SCQ7ZKH4}, which are transported from $N$ by $J$.
\end{remark}

\Cref{1SCIQFIG1} summarizes the argument.
\begin{figure}[H]
\centering
{\hypersetup{linkcolor=figref}
\begin{tikzpicture}[
  font=\small,
  ent/.style={align=flush center, inner sep=3pt},
  lab/.style={font=\scriptsize, align=flush center, inner sep=2pt},
  vlab/.style={font=\scriptsize, align=flush left, inner sep=3pt},
  ref/.style={
    font=\scriptsize,
    text=figref,
    align=flush center,
    inner sep=2pt
  },
  ar/.style={-{To[length=1.4mm]}, line width=.55pt, draw=figarrow},
  sq/.style={
    -{To[length=1.4mm]},
    line width=.55pt,
    draw=figarrow,
    decorate,
    decoration={
      zigzag,
      amplitude=.8pt,
      segment length=3.2pt,
      pre=lineto,
      pre length=1.5pt,
      post=lineto,
      post length=2.5pt
    }
  }
]
\def\A{0}\def\B{4.4}\def\C{8.8}\def\D{13.2}

\node[ent] (d0) at (\A,0) {$(S,\psi,\nabla)$};
\node[ent] (d1) at (\B,0) {$(S,\psi_1,\nabla)$};
\node[ent] (d2) at (\C,0) {$(S',\psi',\Lambda_*\nabla)$};
\node[ent] (d3) at (\D,0) {$(S',\psi',\nabla')$};

\draw[ar] (d0) --
  node[lab,above]{change $\psi$}
  coordinate[midway] (m1) (d1);
\draw[ar] (d1) --
  node[lab,above]{change $S$}
  coordinate[midway] (m2) (d2);
\draw[ar] (d2) --
  node[lab,above]{change $\nabla$}
  coordinate[midway] (m3) (d3);

\node[ent] (t1) at ([yshift=-1.5cm]m1)
  {$\delta=\psi_1^{-1}\psi$, $\delta_*=e^{\ad_X}$\\
   $\wt(X)\ge1$, $\wt(\pi_\psi)\ge-1$\\
   filtration $\mathrm{wt}+2a$\\[1pt]
   {\scriptsize\color{figref}
    Lem. \ref{1SC3P8VEF}, Lem. \ref{1SC545338},
    \eqref{1SCG34NBG}}};

\node[ent] (t2) at ([yshift=-1.5cm]m2)
  {$\varphi\colon S\to S'$ Poisson\\
   $\varphi|_N=\id_N$, $\pi_{\psi'}=\Lambda_*\pi_\psi$\\[1pt]
   {\scriptsize\color{figref}\eqref{1SCMS9RZV}}};

\node[ent] (t3) at ([yshift=-1.5cm]m3)
  {$D'=e^{\ad_\xi}D$\\
   $e^{\ad_\xi}\circ\lambda^{-1}=(\lambda')^{-1}$\\[1pt]
   {\scriptsize\color{figref}
    Lem. \ref{1SCJ5VX3L}, Lem. \ref{1SCJEC9RK}}};

\draw[figref!60, dashed, line width=.4pt] (m1) -- (t1.north);
\draw[figref!60, dashed, line width=.4pt] (m2) -- (t2.north);
\draw[figref!60, dashed, line width=.4pt] (m3) -- (t3.north);

\node[ent] (r1) at ([yshift=-3.6cm]m1)
  {$\mathcal U^{\nu,\nabla}_*(\epsilon\pi_{\psi})
    \sim
    \mathcal U^{\nu,\nabla}_*(\epsilon\pi_{\psi_1})$\\
   $e^{\widehat Y}$ isomorphism};

\node[ent] (r2) at ([yshift=-3.6cm]m2)
  {$\theta=\wedge^\bullet\Lambda$\\
   strict isomorphism};

\node[ent] (r3) at ([yshift=-3.6cm]m3)
  {$\mathcal U^{\nu,\nabla}_*(\epsilon\pi_\psi)
    \sim
    \mathcal U^{\nu,\nabla'}_*(\epsilon\pi_\psi)$\\
   $e^{\widehat Y}$ isomorphism};

\draw[ar] (t1) --
  node[vlab,right]
    {gauge\\
     {\color{figref}Lem. \ref{1SCMF2KEY}(2)}}
  (r1);

\draw[ar] (t2) --
  node[vlab,right]
    {naturality of $\mathcal U$\\
     {\color{figref}Lem.~\ref{1SCVVM8AU}}}
  (r2);

\draw[ar] (t3) --
  node[vlab,right]
    {$\mathcal U^{\nu,\nabla}\simeq\mathcal U^{\nu,\nabla'}$\\
     {\color{figref}
      Lem. \ref{1SC3ZLF6S}, Lem. \ref{1SCMF2KEY}(3)}}
  (r3);

\node[ent] (f) at ([yshift=-5.0cm]m2)
  {$\mu^{\mathbf d}\cong\mu^{\mathbf d'}$
   for all quantization data $\mathbf d,\mathbf d'$\\[1pt]
   {\scriptsize\color{figref}Thm. \ref{1SCDN9DFE}}};

\draw[sq] (r1.south) |-
  node[ref,pos=.25,left,xshift=-1pt]
    {Prop. \ref{1SCLVACSU}}
  (f.west);

\draw[ar] (r2) --
  node[ref,right]{Prop. \ref{1SC82SQTE}}
  (f);

\draw[sq] (r3.south) |-
  node[ref,pos=.25,right,xshift=1pt]
    {Prop. \ref{1SCABX7MH}}
  (f.east);

\end{tikzpicture}}
\caption{Independence of the quantization datum: each entry of
$\mathbf d=(S,\psi,\nabla)$ is changed in turn, and the three
isomorphisms are composed.}
\label{1SCIQFIG1}
\end{figure}
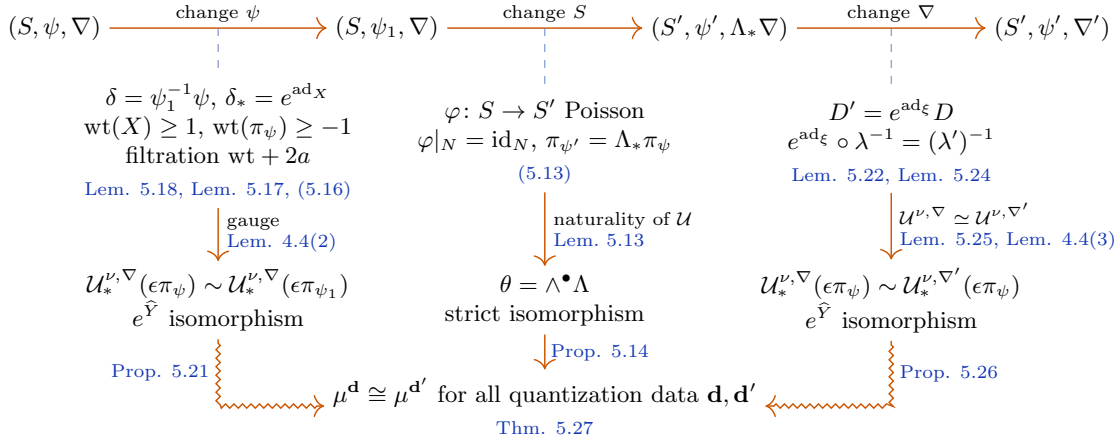

\subsection{Quantization of infinitesimal $1$-shifted coisotropics}\label{1SCQ7ZKH4}
We now apply the preceding construction to the $P_\infty$-algebra of \Cref{1SCW34U8K}: we quantize $\Delta\subset(U,\widehat\pi_0^{\mathbf e})$ and transport along $\Phi^{\mathbf e}$.

Let $(\ccc,L_N,L_M)$ be a regular datum. On $\Delta$, we fix the torsion-free connection and the formal exponential map transported from $N$ by $J$. They do not depend on $\mathbf e$. Let $\mathbf d=(\Sigma,\varsigma,\nabla)$ be a quantization datum for $\Delta\subset(U,\widehat\pi_0^{\mathbf e})$, and let $\mu^{\widehat\a^{\mathbf e},\mathbf d}$ be the curved $A_\infty$-structure of \Cref{1SCL9J8YC} on $\widehat\a^{\mathbf e}\llbracket\epsilon\rrbracket$. We extend $\Phi^{\mathbf e}\colon\a\to\widehat\a^{\mathbf e}$ $\R\llbracket\epsilon\rrbracket$-linearly to $\a\llbracket\epsilon\rrbracket\to\widehat\a^{\mathbf e}\llbracket\epsilon\rrbracket$, and set
\[
\mu_k^{(\mathbf e,\mathbf d)}:=(\Phi^{\mathbf e})^{-1}\circ\mu_k^{\widehat\a^{\mathbf e},\mathbf d}\circ(\Phi^{\mathbf e})^{\otimes k}\ \ (k\ge1),
\quad
\mu_0^{(\mathbf e,\mathbf d)}:=(\Phi^{\mathbf e})^{-1}\bigl(\mu_0^{\widehat\a^{\mathbf e},\mathbf d}\bigr).
\]

\begin{theorem}\label{1SCQTHM01}
Let $(\ccc,L_N,L_M)$ be a regular datum, $\mathbf e$ a choice, and $\mathbf d=(\Sigma,\varsigma,\nabla)$ a quantization datum for $\Delta\subset(U,\widehat\pi_0^{\mathbf e})$. Then $\{\mu_k^{(\mathbf e,\mathbf d)}\}_{k\ge0}$ is a curved $A_\infty$-structure on $\a\llbracket\epsilon\rrbracket$ with multidifferential structure maps, and:
\begin{itemize}
\item[(1)] modulo $\epsilon$, $\mu_2^{(\mathbf e,\mathbf d)}$ is the product of $\a$ and $\mu_k^{(\mathbf e,\mathbf d)}=0$ for $k\ne2$;
\item[(2)] $\mu_0^{(\mathbf e,\mathbf d)}=O(\epsilon^2)$;
\item[(3)] $\frac1\epsilon\bigl(\mu_k^{(\mathbf e,\mathbf d)}\circ\Alt_k\bigr)\bmod\epsilon=(-1)^{(k-1)(k-2)/2}\ell_k^{(\mathbf e,\tau)}$ for every $k\ge0$, where $\tau=(\Sigma,\varsigma)$. In particular $\frac1\epsilon\mu_1^{(\mathbf e,\mathbf d)}\bmod\epsilon=d_{B_\ccc}$;
\item[(4)] the isomorphism class of $\{\mu_k^{(\mathbf e,\mathbf d)}\}$ does not depend on $(\mathbf e,\mathbf d)$.
\end{itemize}
\end{theorem}

\begin{proof}
By \Cref{1SCW34U8K}(1), \Cref{1SCL9J8YC} applies to $\Delta\subset(U,\widehat\pi_0^{\mathbf e})$. Identify $N$ with $\Delta$ through $J$. Then $\Phi^{\mathbf e}$ is an algebra isomorphism of degree $0$, induced by the vector bundle isomorphism $\varphi^{\mathbf e}\colon B_\ccc\to\widehat B^{\mathbf e}$. So \Cref{1SC2EHE94} gives the first assertion and (1). Part (2) follows from \Cref{1SCL9J8YC}(2).

(3) Since $\Phi^{\mathbf e}$ preserves degrees, \eqref{1SCNXARAB} holds for $\Phi^{\mathbf e}$. Conjugate \Cref{1SCL9J8YC}(3) by $\Phi^{\mathbf e}$ and use \Cref{1SCW34U8K}(2).

(4) Let $(\mathbf e',\mathbf d')$ be another choice. Write $\widehat\pi_0:=\widehat\pi_0^{\mathbf e}$, $\widehat\a:=\widehat\a^{\mathbf e}$, $\Phi:=\Phi^{\mathbf e}$, and denote by primes the objects of $\mathbf e'$. By \Cref{1SC454C3Y} there is a Poisson diffeomorphism $\Theta\colon(V,\widehat\pi_0)\to(V',\widehat\pi_0')$, where $V\subset U$ and $V'\subset U'$ are open neighbourhoods of $\Delta$, with $\Theta|_\Delta=\id_\Delta$. The data depend only on germs along $\Delta$, so we may assume $\Sigma\subset V$. Set
\[
\Sigma':=\Theta(\Sigma),\quad \Lambda:=\nu(\Theta|_\Sigma)\colon\nu_\Sigma\to\nu_{\Sigma'},\quad
\mathbf d^\Theta:=\bigl(\Sigma',\,\Theta\circ\varsigma\circ\Lambda^{-1},\,\Lambda_*\nabla\bigr).
\]
Here $\Lambda$ is a vector bundle isomorphism over $\id_\Delta$. Then $\Sigma'$ is a Poisson transversal of $(V',\widehat\pi_0')$ containing $\Delta$ coisotropically, and $\Theta|_\Sigma\colon\Sigma\to\Sigma'$ is a Poisson diffeomorphism which is the identity on $\Delta$. The proof of \Cref{1SC82SQTE} uses only these two facts. Hence $\mathbf d^\Theta$ is a quantization datum for $\Delta\subset(U',\widehat\pi_0')$, and $\wedge^\bullet\Lambda\colon\b_\Sigma\to\b_{\Sigma'}$ is a strict isomorphism from $\mu^{\b_\Sigma,\mathbf d}$ to $\mu^{\b_{\Sigma'},\mathbf d^\Theta}$. The data fixed on $\Delta$ are unchanged, since $\Theta|_\Delta=\id_\Delta$.

Let $\Phi_\Sigma\colon\widehat\a\to\b_\Sigma$ and $\Phi_{\Sigma'}\colon\widehat\a'\to\b_{\Sigma'}$ be the maps \eqref{1SCVX2YHQ} for $\Delta$, and $\widehat\Theta:=\wedge^\bullet(\Theta_B^*)^{-1}\colon\widehat\a\to\widehat\a'$, with $\Theta_B\colon\widehat B\to\widehat B'$ as in the proof of \Cref{1SCW34U8K}(3). By the proof of \Cref{1SCEXNP25}, $(\wedge^\bullet\Lambda)\circ\Phi_\Sigma=\Phi_{\Sigma'}\circ\widehat\Theta$. By \Cref{1SCPN48SN}, $\widehat\Theta$ is a strict isomorphism from $\mu^{\widehat\a,\mathbf d}$ to $\mu^{\widehat\a',\mathbf d^\Theta}$, and
\[
\Psi:=(\Phi')^{-1}\circ\widehat\Theta\circ\Phi\colon\a\to\a
\]
is a strict isomorphism from $\{\mu_k^{(\mathbf e,\mathbf d)}\}$ to $\{\mu_k^{(\mathbf e',\mathbf d^\Theta)}\}$.

By \Cref{1SCDN9DFE} for $\Delta\subset(U',\widehat\pi_0')$ and \Cref{1SCPN48SN} with $\Phi'$ on both sides, there is an isomorphism $\varrho$ from $\{\mu_k^{(\mathbf e',\mathbf d^\Theta)}\}$ to $\{\mu_k^{(\mathbf e',\mathbf d')}\}$. Since $\Psi$ is strict, the composite $\varrho\circ\widehat T(\Psi[1])$ has components $\varrho_0$ and $\varrho_k\circ\Psi^{\otimes k}$, and $\varrho_1\circ\Psi$ is bijective. So it is an isomorphism from $\{\mu_k^{(\mathbf e,\mathbf d)}\}$ to $\{\mu_k^{(\mathbf e',\mathbf d')}\}$.
\end{proof}

Write $H^\bullet(B_\ccc)$ for the Lie algebroid cohomology of $B_\ccc$. By \eqref{1SCGM8GSM}, $\Phi^{\mathbf e}$ induces an isomorphism $H^\bullet(B_\ccc)\to H^\bullet(\widehat B^{\mathbf e})$.

\begin{corollary}\label{1SCQSTAR2}
Let $(\ccc,L_N,L_M)$ be a regular datum and $(\mathbf e,\mathbf d)$ as in \Cref{1SCQTHM01}. For $a\in\epsilon\,\a^1\llbracket\epsilon\rrbracket$, let $\mu^{(\mathbf e,\mathbf d),a}$ be the change of base point of $\{\mu_k^{(\mathbf e,\mathbf d)}\}$ by $a$, in the sense of \eqref{1SCT49PYH}.
\begin{itemize}
\item[(1)] If $H^2(B_\ccc)=0$, there is $a\in\epsilon\,\a^1\llbracket\epsilon\rrbracket$ such that $\mu^{(\mathbf e,\mathbf d),a}$ is flat.
\item[(2)]  If moreover $H^1(B_\ccc)=0$, then $H_a:=H^0(\a\llbracket\epsilon\rrbracket,\mu_1^{(\mathbf e,\mathbf d),a})$ is a topologically free $\R\llbracket\epsilon\rrbracket$-module, isomorphic to $H^0(B_\ccc)\llbracket\epsilon\rrbracket$ with the $\epsilon$-adic topology, and reduction modulo $\epsilon$ induces an isomorphism $H_a/\epsilon H_a\to H^0(B_\ccc)$. Through any $\R\llbracket\epsilon\rrbracket$-linear isomorphism $H^0(B_\ccc)\llbracket\epsilon\rrbracket\to H_a$ reducing to the identity modulo $\epsilon$, the product $\mu^{(\mathbf e,\mathbf d),a}_2$ is a star product on $H^0(B_\ccc)$ deforming the bracket induced by $\ell_2^{(\mathbf e,\tau)}$.
\item[(3)] If moreover $(\ccc,L_N,L_M)$ comes from a $1$-shifted coisotropic structure $\ccc\colon\C\to\G$ with $\C$ source-connected, this is a star product on $C^\infty(N)^\C$ deforming \eqref{1SC2RZH94}.
\end{itemize}
\end{corollary}

\begin{proof}
(1) and (2). Apply \Cref{1SCJQ3RB7} to $\Delta\subset(U,\widehat\pi_0^{\mathbf e})$ and $\mathbf d$. Its hypotheses hold since $H^\bullet(\widehat B^{\mathbf e})\cong H^\bullet(B_\ccc)$. Let $\widehat a\in\epsilon\,(\widehat\a^{\mathbf e})^1\llbracket\epsilon\rrbracket$ be the resulting element, and $a:=(\Phi^{\mathbf e})^{-1}(\widehat a)\in\epsilon\,\a^1\llbracket\epsilon\rrbracket$. As in the proof of \Cref{1SCJQ3RB7}, $\mu^{(\mathbf e,\mathbf d),a}$ is the transport along $\Phi^{\mathbf e}$ of the change of base point of $\mu^{\widehat\a^{\mathbf e},\mathbf d}$ by $\widehat a$. So it is flat, and $\Phi^{\mathbf e}$ induces an isomorphism of $\R\llbracket\epsilon\rrbracket$-algebras from $H_a$ onto the degree-zero cohomology of \Cref{1SCJQ3RB7}(2) for $\Delta$. The terms of $\mu^{(\mathbf e,\mathbf d),a}_2$ containing $a$ involve $\mu^{(\mathbf e,\mathbf d)}_k$ with $k\ge3$, so they are $O(\epsilon^2)$. The semiclassical limit then follows from \Cref{1SCQTHM01}(3).

(3) This is \Cref{1SC64QN3W}.
\end{proof}

Together with Theorem \ref{1SCQTHM01}, this proves Theorem \ref{d4fmvhyq} from the introduction.

\begin{remark}[Quantization of generalized reductions]
Consider a $1$-shifted coisotropic $\C\to\{*\}$ as in Example \ref{39d9asdt}(2), with $\C\tto N$ source-connected. Here $B_\ccc=\ker L_N\cong A_\C$, so the datum is automatically regular. Theorems \ref{1SCW34U8K} and \ref{1SCQTHM01} therefore give a flat $P_\infty$-structure on $\Gamma(\wedge^\bullet A_\C^*)$ and a curved $A_\infty$-quantization, with the independence of auxiliary choices stated there. By \Cref{1SC64QN3W}, its classical bracket in degree zero is the reduced Poisson bracket on $C^\infty(N)^\C$. If $H^1(A_\C)=H^2(A_\C)=0$, Corollary \ref{1SCQSTAR2} gives a star product on this algebra. In particular, under the hypotheses of Example \ref{39d9asdt}(2), this applies to symplectic and quasi-Hamiltonian reduction, and their generalizations to symplectic reduction along a submanifold \cite{CrooksMayrand2022} and reduction of strong Dirac maps \cite{bualibanu2026reduction}.
\end{remark}

\bibliographystyle{amsplain}
\bibliography{poisson}

\end{document}